\documentclass[11pt, reqno]{amsart}
 
\usepackage{fullpage}
\usepackage{amsmath}
\usepackage{amssymb}
\usepackage{hyperref}
\usepackage{graphicx}

\newcommand{\N}{\mathbb{N}}
\newcommand{\Q}{\mathbb{Q}}
\newcommand{\R}{\mathbb{R}}
\newcommand{\C}{\mathbb{C}}
\newcommand{\Z}{\mathbb{Z}}
\newcommand{\mc}{\mathcal}

\renewcommand{\Re}{\operatorname{Re}}
\renewcommand{\Im}{\operatorname{Im}}

\newcommand{\adj}{\operatorname{adj}}

\newtheorem{lemma}{Lemma}[section]
\newtheorem{theorem}[lemma]{Theorem}
\newtheorem{corollary}[lemma]{Corollary}
\newtheorem{proposition}[lemma]{Proposition}

\theoremstyle{remark}
\newtheorem{remark}[lemma]{Remark}

\theoremstyle{definition}
\newtheorem{definition}[lemma]{Definition}

\numberwithin{equation}{section}

\title{Self-similar blowup for the cubic Schr\"odinger equation}

\author{Roland Donninger}
\address{Universität Wien, Fakultät für Mathematik,
  Oskar-Morgenstern-Platz 1, 1090 Vienna, Austria}
\email{roland.donninger@univie.ac.at}

\thanks{This work was supported by the Austrian Science Fund FWF,
  Project P34560. Furthermore, the authors would like to thank the Erwin
  Schr\"odinger International Institute for Mathematics and Physics (ESI)
  for hospitality. This work was finalized during the thematic program
  ``Nonlinear Waves and Relativity'' at ESI in 2024.}

\author{Birgit Sch\"orkhuber}
\address{Universität Innsbruck, Institut für Mathematik,
  Technikerstraße 13, 6020 Innsbruck, Austria}
\email{birgit.schoerkhuber@uibk.ac.at}

\begin{document}

\maketitle

\begin{abstract}
  We give a rigorous proof for the existence of a finite-energy, self-similar
  solution to the focusing cubic Schr\"odinger equation in three
  spatial dimensions. The proof is computer-assisted and relies on a
  fixed point argument that shows the existence of a solution in the
  vicinity of a numerically constructed approximation. The latter is
  obtained by a standard pseudo-spectral method. The computer-assisted
  part of the rigorous proof uses nothing but fraction arithmetic in order to obtain quantitative
  bounds for the fixed point argument.
\end{abstract}

\section{Introduction}

\noindent This paper is concerned with the cubic, focusing Schr\"odinger
equation
\begin{equation}
  \label{eq:nls3}
  i\partial_t \psi(t,x)+\Delta_x\psi(t,x)+\psi(t,x)|\psi(t,x)|^2=0
\end{equation}
for an unknown $\psi: I\times \R^3\to\C$, where $I\subset \R$ is an
interval. Eq.~\eqref{eq:nls3} is the basic model for describing
the competition of dispersion and nonlinear, focusing effects. As a
consequence, the cubic Schr\"odinger equation appears in virtually every
problem where nonlinear wave interaction is crucial,
e.g.~in nonlinear optics or in the study of propagation of waves
in plasmas. We refer the interested reader to \cite{SulSul99} and
references therein for the background. In addition, nonlinear Schr\"odinger equations are
fundamental for many-particle quantum mechanics where they occur as
effective equations in suitable limits of large particle
numbers. From a pure mathematical perspective, the desire to
understand nonlinear
Schr\"odinger equations is one of the main driving forces in many
branches of contemporary analysis.

The natural setup for studying
Eq.~\eqref{eq:nls3} is the \emph{Cauchy problem}, that is to say, one
prescribes a function $\psi_0: \R^3\to\C$ (the \emph{data}) and attempts to understand the
behavior of the solution $\psi$ of Eq.~\eqref{eq:nls3}
that satisfies $\psi(0,\cdot)=\psi_0$. It is well-known that at least for short
times $t$, this is a well-posed
problem in the sense that a unique solution exists that depends
continuously on the data, provided $\psi_0$ is sufficiently
regular, see e.g.~\cite{SulSul99} and references therein. For large times, however, the situation is much more
delicate. In fact, it is known that solutions may cease to exist after a
finite time. This can be shown by a qualitative argument
\cite{Gla77} which, however, gives little information on what actually happens to
the solution. Consequently, there have been many attempts to quantify
the breakdown and numerical studies indicate the existence of a
nontrivial (meaning not identically zero)
\emph{self-similar solution} $\psi_*$ of the form
\begin{equation}
  \label{eq:psistar}
  \psi_*(t,x)=\frac{1}{\sqrt{2\alpha}}(1-t)^{-\frac12-\frac{i}{2\alpha}}Q\left
    (\frac{x}{\sqrt{2\alpha(1-t)}}\right),
  \end{equation}
defined on $(-\infty,1)\times \R^3$, where $Q: \R^3\to\C$ is a smooth,
radial function in $L^4(\R^3)\cap \dot H^1(\R^3)$ and $\alpha\approx 0.917$, see \cite{SulSul99}, p.~121.
Such a solution provides a quantitative example of the breakdown but what
is even more exciting is the conjectured \emph{universality} of
$\psi_*$. Indeed,
in numerical simulations one observes that for sufficiently large but
otherwise arbitrary
data, the solution develops a singularity and
close to the breakdown, the
shape of the solution is independent of the particular form of
the data and approaches $\psi_*$ (modulo symmetries of the equation). 
In other words, shortly before
breaking down, 
the solution ``forgets'' all of its features and takes on a
universal shape given by $\psi_*$. This means that the
self-similar solution $\psi_*$ appears to describe the \emph{generic}
breakdown behavior of Eq.~\eqref{eq:nls3} and is thus a key feature of
the model. As a consequence, in order
to rigorously understand the dynamics of Eq.~\eqref{eq:nls3}, it is of
tantamount importance to prove the existence of $\psi_*$. Unfortunately, this seems to be a hard problem that resisted
all attempts so far. In this paper, we give the first rigorous proof
for the existence of the self-similar solution $\psi_*$.

\begin{theorem}[Main theorem, qualitative version]
  \label{thm:qual}
  There exists a nontrivial, radial function $Q\in L^4(\R^3)\cap \dot
  H^1(\R^3)\cap C^\infty(\R^3)$ and an $\alpha>0$ such that $\psi_*:
  (-\infty,1)\times \R^3\to\C$, defined by
  \[ \psi_*(t,x):=\frac{1}{\sqrt{2\alpha}}(1-t)^{-\frac12-\frac{i}{2\alpha}}Q\left
        (\frac{x}{\sqrt{2\alpha(1-t)}}\right), \]
    satisfies
    \[ i\partial_t
      \psi_*(t,x)+\Delta_x\psi_*(t,x)+\psi_*(t,x)|\psi_*(t,x)|^2=0 \]
    for all $(t,x)\in (-\infty,1)\times \R^3$. In particular, $Q$
    satisfies the \emph{profile equation}
    \begin{equation}
\label{eq:profile}
  \Delta Q(\xi)-Q(\xi)+i\alpha\left [\xi\cdot\nabla
    Q(\xi)+Q(\xi)\right]+Q(\xi)|Q(\xi)|^2 =0
\end{equation}
for all $\xi\in \R^3$.
  \end{theorem}

  In addition to the pure existence statement in Theorem \ref{thm:qual}, we obtain an extremely
  precise approximation to $Q$ with rigorously proven error
  bounds. 

 \begin{definition}[Chebyshev polynomials]
       Let $n\in\N_0$. Then $T_n: \C\to\C$ is defined to be the unique polynomial
       that satisfies
       \[ T_n(\cos(z))=\cos(nz) \]
       for all $z\in\C$.
     \end{definition}

  \begin{definition}
    \label{def:gstar}
     Let
     \[ P_*(y):=\sum_{n=0}^{50} c_n(P_*) T_n(y) \]
     with $(c_n(P_*))_{n=0}^{50}\subset \C$ given in Appendix
     \ref{apx:Pstar}. We set
     \[ a_*:=\tfrac{772201763088846}{841768781900003} \] and
     define $g_*: [0,\infty)\times \R\to\C$ by
     \[ g_*(r,a):=P_*\left (\frac{r-1}{r+1}\right)+\frac{2i(a_*+a)}{1-i(a_*+a)}P_*'(1)-P_*(1). \]
   \end{definition}

   The value of $a_*$ and the polynomial $P_*$ come from a numerical
   approximation of the self-similar profile, see Appendix
   \ref{apx:numerical} for a detailed description of the
   numerical method. The function $g_*$ is a modification of the
   numerically obtained profile that is necessary to match the exact asymptotics of the
   solution. All of this will be explained in much more detail
in Section \ref{sec:idea} below. For now we just introduce all the necessary quantities in
   order to
   formulate the theorem.

   \begin{theorem}[Main theorem, quantitative version]
     \label{thm:quant}
    There exists an $a\in [-10^{-10},10^{-10}]$ and a
    $g: [0,\infty)\to\C$ satisfying
    \[ 2\sup_{r>0}r|g'(r)|+\sup_{r>0}|g(r)|\leq
      1.2\cdot 10^{-6} \]
    such that $Q: \R^3\to\C$, defined by
        \[ Q(x):=(1+|x|)^{-1-\frac{i}{a_*+a}}\big
            [g_*(|x|,a)+g(|x|)\big ], \]
      belongs to $L^4(\R^3)\cap \dot H^1(\R^3)\cap C^\infty(\R^3)$ and 
    \[ \psi_*(t,x):=\frac{1}{\sqrt{2(a_*+a)}}(1-t)^{-\frac12-\frac{i}{2(a_*+a)}}Q\left
        (\frac{x}{\sqrt{2(a_*+a)(1-t)}}\right) \]
    satisfies
 \[ i\partial_t
      \psi_*(t,x)+\Delta_x\psi_*(t,x)+\psi_*(t,x)|\psi_*(t,x)|^2=0 \]
    for all $(t,x)\in (-\infty,1)\times \R^3$. 
  \end{theorem}

  \begin{remark}
    Note that
    \[ a_*=\tfrac{772201763088846}{841768781900003}=0.9173561430\dots
    \]
    and
    \[ g_*(0,0)=-1.88566\dots . \]
These particular values appear in the numerical literature
\cite{KaiRouZha19}
and strongly
suggest that the solution we construct coincides with the one that
has been studied numerically for a long time.
\end{remark}

\begin{remark}
  With a little more work one can also show that the map $r\mapsto
  |Q(re_1)|: [0,\infty)\to (0,\infty)$, where $e_1:=(1,0,0)\in \R^3$,
  is monotonically decreasing. In fact, away from the origin
the monotonicity follows immediately from the explicit form of
$g_*$ and the stated bound on $g'$.
Since the bound on $g'$ degenerates at the origin, one
needs to supplement this with a straightforward local analysis near
$0$. We leave this to the interested reader.
\end{remark}

Evidently, Theorem \ref{thm:quant} implies Theorem \ref{thm:qual}.

\begin{figure}[h]
  \centering
  \includegraphics[width=0.5\textwidth]{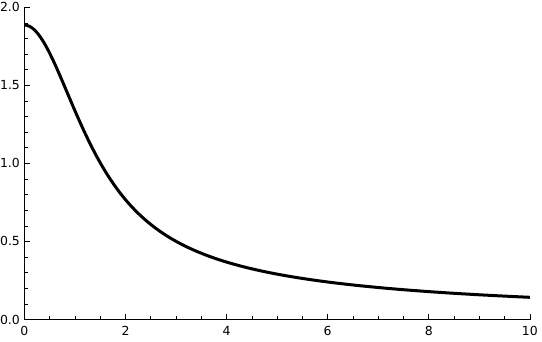}
  \caption{A plot of $|Q_*(re_1)|$ for $r\in [0,10]$, where $Q_*:
    \R^3\to\C$ is given by
    \[ Q_*(x):=(1+|x|)^{-1-\frac{i}{a_*}}g_*(|x|,0), \]
    see Definition \ref{def:gstar}. The function $Q_*$ is an
    approximation to the profile $Q$ of Theorems \ref{thm:qual} and
    \ref{thm:quant}.}
\end{figure}

\subsection{History of the problem}
The formation of singularities in the three-dimensional focusing cubic Schr\"odinger equation has received considerable attention ever since the model was derived
by Zakharov \cite{Zak72}, \cite{Zak84} in the context of plasma physics. There,
Eq.~\eqref{eq:nls3} arises as a limit of the so-called Zakharov system
which provides a simplified description of the propagation of Langmuir
waves in an ionized plasma. With $\psi$ modelling the envelope of the electric field, finite-time blowup (referred to as \textit{wave collapse}) via radially symmetric self-similar solutions  of the form \eqref{eq:psistar} was proposed as a mechanism for energy dissipation.

From a rigorous point of view,  a first ``obstructive" blowup result for Eq.~\eqref{eq:nls3} was obtained by Glassey \cite{Gla77} for initial data with finite variance and negative energy, where the latter is a conserved quantity of the flow,  
\begin{align*}\label{NLS_energy}
 E(\psi(t,\cdot)) = \frac{1}{2} \int_{\R^3} |\nabla_x  \psi(t,x)|^2 dx -\frac{1}{4} \int_{\R^3} |\psi(t,x)|^4 dx = E(\psi(0,\cdot)),
\end{align*}
see also \cite{OgaTsu91} for an analogous statement in the radial case without requiring finite variance.
However, it is in the nature of these types of results that they do not provide much information on the details of singularity formation.

McLaughlin, Papanicolaou, Sulem, and Sulem \cite{McLPapSulSul86} investigated the problem numerically and confirmed Zakharov's predictions on the asymptotically self-similar nature of the blowup for Eq.~\eqref{eq:nls3} by looking at different types of radially symmetric initial data. More importantly, they observed convergence to a universal limiting blowup profile corresponding to a solution of Eq.~\eqref{eq:profile} for $\alpha  \approx 0.917$. Later,  Landman, Papanicolaou, Sulem, Sulem, and Wang \cite{LanPapSulSulWan91} came to the same conclusion even for nonradial data. Recent studies by Yang, Roudenko, and Zhao \cite{KaiRouZha19} confirmed the earlier findings in the radial setting.  Due to these results it is  widely believed that the asymptotic behavior of solutions to Eq.~ \eqref{eq:nls3} corresponding to ``generic" large initial data is well described (up to symmetries) by a self-similar solution of the above form
with a radially symmetric blowup profile $Q$ corresponding to $\alpha  \approx 0.917$, where $|Q|$ decreases monotonically.

Rigorous results on the profile equation  \eqref{eq:profile} were
first obtained by Wang \cite{Wan90}, who proved the existence of radially symmetric solutions for any $\alpha \in \R\setminus\{0\}$ and prescribed $Q(0) \in \R$. However, the solutions display highly oscillatory behavior at infinity in general. More precisely, for large values of $|\xi|$, $Q = c_1 Q_1+ c_2 Q_2$, for $c_1,c_2 \in \C$, where 
\[ Q_1(\xi) \sim |\xi|^{-1-\frac{i}{\alpha}} , \qquad Q_2(\xi) \sim |\xi|^{-2+\frac{i}{\alpha}} e^{-i\alpha \frac{|\xi|^2}{2}}, \]
see e.g.~\cite{LeMPapSulSul88}. In particular, solutions have finite
energy\footnote{In fact, in this case the energy must be $0$, see \cite{SulSul99}.}
only if $c_2 = 0$. The latter property is conjectured to hold only for
specific values of $\alpha$. In this respect we would like to mention
the recent preprint \cite{Tro22} that erroneously\footnote{The flaw in
  the argument of \cite{Tro22} seems to be in the step from Eq.~(2.35) to (2.36).} claims that this is the case for every $\alpha > 0$ and $Q(0) > 0$. In fact, the solution provided by Theorem \ref{thm:quant} is the first rigorous example of this kind and its properties strongly suggest that it corresponds to the numerically observed limiting blowup profile.

\subsection{Related results}
The dynamics for the cubic NLS  crucially depend on the underlying
space dimension. In the one-dimensional case, the equation is
completely integrable and solutions exist globally in time.  For
$d=2$, the problem is mass critical, i.e., the  $L^2$-norm is
invariant under the natural scaling of the equation. In this case,
singularities may form in finite time and the generic blowup behavior
is conjectured to be described by a dynamically rescaling soliton with
a nonself-similar blowup rate. The literature in this field is vast
and since the present paper is solely concerned with the mass
supercritical case, we restrict our discussion to this situation.
We remark that the critical space for Eq.~\eqref{eq:nls3} is $\dot H^{\frac{1}{2}}(\R^3)$, which renders the problem mass supercritical and energy subcritical. 

For the cubic NLS in dimensions $d > 2$ sufficiently close to the mass critical case, Kopell and Landman \cite{KopLan95} as well as Rottschäfer and Kaper \cite{RotKap02}  proved the  existence of radial self-similar profiles, which are monotonically decreasing in absolute value and display the ``good'' asymptotic behavior at infinity leading to finite energy. The condition on the dimension translates into smallness of the parameter $\alpha = \alpha(d)$ in the corresponding profile equation. In a similar spirit, but using different methods, radially symmetric self-similar solutions have been constructed recently by Bahri, Martel, and Rapha\"el \cite{BahMarRap21} in the more general framework of slightly $L^2$-supercritical nonlinear Schrödinger equations, 
\begin{align}\label{eq:NLSp}
i \partial_t \psi(t,x) + \Delta_x \psi(t,x)  + \psi(t,x)|\psi(t,x)|^{p-1} = 0,
\end{align}
for $d \geq 1$ and $0 < p - p_c \ll 1$, where $p_c = 1 + \frac{4}{d}$.
More precisely, they proved that for any $p$ sufficiently close to
$p_c$, there exists an $\alpha = \alpha(p)$ and a corresponding
nontrivial finite-energy solution of the corresponding profile equation.
We note that in this slightly supercritical case also non-monotonic
self-similar profiles are known to exist, see the work of Rottschäfer
and Kaper \cite{RotKap03} as well as Budd, Chen, and Russel
\cite{BudCheRus99}, and Budd \cite{Bud01}, see also \cite{KaiRouZha19}.

Merle, Rapha\"el, and Szeftel \cite{MerRapSze10} investigated more
general dynamics in this perturbative regime and proved for
$1 \leq d \leq 5$ the existence of an open set of initial data in $H^1(\R^d)$ which lead to solutions blowing up at a self-similar rate. Their approach does not rely on the
knowledge of exact self-similar profiles, instead they use rough
approximations and extend techniques from the mass critical case.

A completely different type of blowup solutions, so-called ``standing
ring solutions'', have been constructed for Eq.~\eqref{eq:NLSp} with
$p=5$  by Rapha\"el \cite{Rap06} ($d=2$), and Rapha\"el and Szeftel
\cite{RapSze09} ($d \geq 3$) in the radial setting, see also \cite{HolRou12}. The considered
dimensions correspond to mass supercritical (but energy subcritical),
energy critical, and energy supercritical regimes. Moreover, these
solutions are  stable under small, smooth radial perturbations. Adopting this approach,  Holmer and Roudenko \cite{HolRou11}  and Zwiers \cite{Zwi11} studied Eq.~\eqref{eq:nls3} in axial symmetry and constructed solutions blowing up in finite time on a circle with a nonself-similar rate. Stability of these solutions holds within this particular symmetry class.
Fibich, Gavish and Wang \cite{FibGavWan05} revealed yet another blowup
mechanism for Eq.~\eqref{eq:nls3} by numerically observing solutions which focus on a sphere that collapses towards the origin, see also  \cite{FibGavWan07}. The existence of such ``collapsing
ring solutions'' was then proved by Merle, Rapha\"el, and Szeftel
\cite{MerRapSze14} for Eq.~\eqref{eq:NLSp} for $d \geq 2$, $p_c < p < 5$ and independently by
Holmer, Perelman, Roudenko \cite{HolPerRou15} for $d=3$ and $p=3$.
Both standing and collapsing ring solutions are conjectured to be
unstable under nonradial perturbations. 

Finally, we note that for Eq.~\eqref{eq:NLSp} in the energy supercritical regime $d \geq 11$ and $p$ sufficiently large, blowup via concentration of a soliton profile has been proved by Merle, Rapha\"el and Rodnianski \cite{MerRapRod15}. 

Summarizing the situation for the three-dimensional cubic NLS
\eqref{eq:nls3}, the ring-type solutions of \cite{HolRou11, Zwi11, MerRapSze14, HolPerRou15} 
and the self-similar solution $\psi_*$ provided by
Theorem \ref{thm:quant} are, to the best of our knowledge, the only explicit examples for finite-time
blowup. Recently, a general framework for analyzing the stability of
self-similar blowup solutions in the mass supercritical (but energy
subcritical) regime has been presented by Li \cite{Zex23}. His results imply finite-codimension asymptotic stability of $\psi_*$
in $\dot H^{s}(\R^3)$ for $0 < s - \frac{1}{2} \ll 1$. However, in
view of the numerical observations, a much stronger stability result
is expected to hold.

Lastly, we would like to mention that Eq.~\eqref{eq:nls3} admits
solitary waves. They have been studied extensively in the past and
play a key role in the description of large data dynamics in
$H^1(\R^3)$, see in particular the works by Nakanishi and Schlag
\cite{NakSch12} and Duyckaerts and Roudenko \cite{DuyRou10} as well as
the references therein. The blowup solution $\psi_*$ (as any
self-similar solution) does not have finite mass and hence does not
fit into this framework.
  
\subsection{Computer-assisted proofs}
Our approach relies on computer assistance. Needless to say, this is
not an original feature but has a
long tradition that goes back to the 1980s at least.
Currently, however, computer-assisted methods in partial differential equations
become more and more popular and move from a niche into the
mainstream. As a consequence, there is now a surge of such results,
see e.g.~\cite{WilZgl20, Bre22, DahGom23, BucCaoGom23, CadLesNav24,
  BreChu24, WilZgl24,
  BucCaoGom25, CheHou25}
for a first impression
and \cite{Gom19, NakPluWat19}.
We would like to contrast our work with
other computer-assisted methods.

The main philosophy of our approach is to stay as close as possible to
traditional pen-and-paper mathematics and to keep the computer assistance
at a bare minimum. We believe that this is crucial in order to still gain
valuable analytic insight that might eventually lead to a traditional proof,
which remains the gold standard. In addition, we find it
important to make rigorous mathematical proofs independent of current
technology. Software and hardware that we are now using on a daily
basis will become deprecated and inaccessible in the future. In view of
this, we follow the most rigorous standards when it comes to computer
assistance and strictly adhere to the following principles.

\begin{itemize}
\item Mathematical objects need to be defined in an unambiguous and
  reproducible way on paper, independent of hardware, software packages, and digital
  storage technology.
\item Computations need to be independent of hardware and software
  standards that are
  subject to change
  in the future\footnote{This excludes for instance floating-point
  arithmetic. However, one may argue that this is unnecessarily restrictive. Indeed, the
computer-assisted parts of this paper can be much more efficiently
implemented using floating-point computations combined with interval
arithmetic.}.
  \item Mathematical proofs need to be independent of downloads from digital
    repositories and must be reproducible at any point in time solely
    based on the information on paper.
\end{itemize}
We thoroughly implement these rigorous principles by essentially sticking to integer
arithmetic and exact computations, see below. This sets our approach apart from
other current developments.

\section{Idea of proof}
\label{sec:idea}

\noindent We give an outline of the proof describing the main difficulties
and how we overcome them. 

\subsection{Setup of the problem}
To begin with, we rewrite Eq.~\eqref{eq:nls3} in self-similar variables, i.e., we
define a new unknown $u$ by 
\[ \psi(t,x)=\frac{1}{\sqrt{2\alpha(1-t)}}u\left (-\frac{1}{2\alpha}\log
    (1-t),\frac{x}{\sqrt{2\alpha(1-t)}}\right ), \]
where $\alpha>0$, $t<1$, and $x\in \R^3$.
Then $\psi$ satisfies Eq.~\eqref{eq:nls3} if and only if $u$ satisfies
\begin{equation}
  \label{eq:nls3ss}
  i\partial_\tau u(\tau,\xi)+\Delta_\xi
  u(\tau,\xi)+i\alpha \left [\xi\cdot\nabla_\xi
  u(\tau,\xi)+u(\tau,\xi)\right ]+u(\tau,\xi)|u(\tau,\xi)|^2=0
\end{equation}
for $\tau\in \R$ and $\xi\in\R^3$.
We are looking for self-similar profiles, that is to say, we insert
the ansatz $u(\tau,\xi)=e^{i\tau}Q(\xi)$. This yields the
profile equation Eq.~\eqref{eq:profile}
  for the function $Q:\R^3\to\C$.
By writing $Q(\xi)=q(|\xi|)$, Eq.~\eqref{eq:profile} transforms into
\begin{equation}
  \label{eq:profilerad}
  \mc E(\alpha,q)(r):=q''(r)+\left (\frac{2}{r}+i\alpha r\right )q'(r)+(i\alpha-1)q(r)
  +q(r)|q(r)|^2=0
\end{equation}
for $r>0$.
Consequently, the proof boils down to showing the existence of a
suitable solution of Eq.~\eqref{eq:profilerad}.

\subsection{Difficulties}
The following observations indicate
that solving Eq.~\eqref{eq:profilerad} might be a challenging problem.
\begin{itemize}
\item The domain on which Eq.~\eqref{eq:profilerad} is posed is unbounded.
  \item The equation is nonlinear and depends on a parameter $\alpha$
    that is unknown itself.
    \item The equation is singular at $0$ and at $\infty$. Moreover,
      the singularity at $\infty$ is irregular in the Fuchsian sense and the
      value of $\alpha$ determines the leading order asymptotics near $\infty$.
    \item There is a degeneracy due to phase invariance: If
      $q$ is a solution then so is $e^{i\theta}q$ for any $\theta\in\R$.
    \item There is no hope for a perturbative treatment as there is no
      known solution that is in any sense ``close'' to the desired
      solution.
    \end{itemize}
    Since there is no object to ``perturb off'', one obvious strategy is to first construct an approximate solution
$(\alpha_*,q_*)$ with the desired properties by some numerical means. In
a second step, one may try to find a correction $(\widetilde\alpha,\widetilde q)$ so that
$(\alpha_*+\widetilde\alpha, q_*+\widetilde q)$ becomes a true solution.
To this end, it is natural to linearize at $(\alpha_*,q_*)$ and write
\[ 0=\mc E(\alpha_*+\widetilde\alpha,q_*+\widetilde q)=\mc E(\alpha_*,q_*)+\mc E'(\alpha_*,q_*)(\widetilde\alpha,\widetilde q)+\mbox{nonlinear
    terms}. \]
If the approximation $(\alpha_*,q_*)$ is accurate enough, i.e., if $\mc E(\alpha_*,q_*)$
is sufficiently small, 
one might hope to solve this equation by a fixed point
argument. To this end, it is necessary to invert the linear operator
$\mc E'(\alpha_*,q_*)$ and the operator norm of its inverse will determine the
required quality of the approximation $(\alpha_*,q_*)$. This naive
approach faces some serious difficulties and is ultimately bound
to fail for a number of reasons: 
\begin{itemize}
\item Already from a pure numerical perspective, constructing
  highly accurate approximate solutions to Eq.~\eqref{eq:profilerad}
  is challenging because of the unbounded domain and the singularity
  at infinity. In fact, a substantial number of numerics papers is
  concerned with just that.
  \item The linear operator is not linear over $\C$ but merely linear
    over $\R$
    due to the form of the nonlinearity. This is more of a technical
    complication but nevertheless increases the algebraic
    complexity of the problem considerably.
  \item The linearized equation
    \[ \mc E'(\alpha_*,q_*)(\widetilde\alpha,\widetilde q)=0 \]
    cannot be solved explicitly. Thus, it is unclear how to obtain a
    fundamental system in order to invert
    $\mc E'(\alpha_*,q_*)$. 
  \item In fact, the linear operator $\mc E'(\alpha_*,q_*)$ has a nontrivial kernel due to the phase
    invariance and is hence not invertible.
    \item The correction $\widetilde\alpha$ cannot be obtained by a
      purely perturbative argument because it determines the leading order
      asymptotics of the solution.
\end{itemize}

\subsection{Strategy of the proof}
\label{sec:strategy}
In order to overcome all of these difficulties, we choose a more
refined strategy which we outline now. The following overview is
a simplified description of the actual proof. We refrain from
mentioning many of the technicalities one faces and try to explain the
overarching strategy. One of the main technical difficulties we
completely suppress at this point is the fact that one actually has to
separate real and imaginary parts and deal with a system. This makes
the notation much more cluttered and distracting. As a consequence,
some of the
objects mentioned here have slightly different names in the technical
part but should be easy to identify given the references we provide.

\begin{enumerate}
\item \label{itm:compactify}
  First, we compactify the problem
by setting
\[ q(r)=(1+r)^{-1-\frac{i}{\alpha}}f\left (\frac{r-1}{r+1}\right). \]
In particular, we factor out the expected asymptotic oscillations of
the solution.
Then, $q$ satisfies Eq.~\eqref{eq:profilerad} if and only if $\mc
R(\alpha,f)=0$ on $(-1,1)$, where
\[
  \mc R(\alpha,f)(y):=f''(y)+p_0(y,\alpha)f'(y)+q_0(y,\alpha)f(y)+\frac{f(y)|f(y)|^2}{(1-y)^2}, \]
and
\begin{align*}
  p_0(y,\alpha)&:=\frac{4i\alpha}{(1-y)^3}-\frac{2i\alpha}{(1-y)^2}-\frac{2+\frac{2i}{\alpha}}{1-y}+\frac{2}{1+y} \\
  q_0(y,\alpha)&:=-\frac{2-2i\alpha}{(1-y)^3}-\frac{\frac{1}{\alpha^2}-\frac{i}{\alpha}}{(1-y)^2}-\frac{1+\frac{i}{\alpha}}{1-y}-\frac{1+\frac{i}{\alpha}}{1+y}.
\end{align*}
This removes the problems connected to the unboundedness of the
original domain and the oscillatory behavior of the solution near $\infty$.

\item \label{itm:approx} Next, we employ a Chebyshev pseudo-spectral method in order to
  obtain an approximate solution $(a_*,f_*)$ that is defined by
  \[ f_*(y,a):=P_*(y)+\frac{2i(a_*+a)}{1-i(a_*+a)}P_*'(1)-P_*(1), \]
  where the polynomial $P_*: \Q\to\Q$ and $a_*\in \Q$ are given in Definition \ref{def:gstar}. Observe the $a$-dependent correction that
  accounts for the fact that the parameter $\alpha=a_*+a$ determines
  the leading-order asymptotics. Furthermore, it is important to note
  that the origin of the coefficients in Appendix \ref{apx:tables} that
  define
  $P_*$ and the other auxiliary polynomials is completely
  irrelevant. Once these coefficients are given, our whole
  argument is rigorous and independent of numerics. The numerical
  method used to obtain the coefficients is described in Appendix \ref{apx:numerical}.

\item \label{itm:lin} Now we would like to correct $(a_*,f_*)$ to an actual solution,
  i.e., we need to find $(a,f)$ such that
$\mc R(a_*+a,f_*(\cdot,a)+f)=0$. 
  As usual, the idea is to apply a fixed point argument and hence we linearize at $(a_*,f_*)$ and write
  \[ \mc R(a_*+a,f_*(\cdot,a)+f)=\mc R(a_*+a,f_*(\cdot,a))+\mc
    L_a(f)+\mc N_a(f), \]
  where $\mc L_a$ is linear and $\mc
  N_a(f)$ contains all the
 terms that are quadratic or cubic in $f$. Note that we only linearize
 in $f$. Thus, our goal is to solve
 \[ \mc L_a(f)=-\mc R(a_*+a,f_*(\cdot,a))-\mc N_a(f). \]

\item \label{itm:approxlin} We
  construct an approximate fundamental system for $\mc
  L_a$ that is defined
  by the coefficients in Appendices \ref{apx:PL} and \ref{apx:PR}.
  These coefficients are obtained numerically by solving the
  linearized equation with the pseudo-spectral method described in
  Appendix \ref{apx:numerical}. Recall that in general, a linear operator can be
  reconstructed from its fundamental system. Applying this
  reconstruction to the \emph{approximate} fundamental system yields an
  \emph{approximate} linear operator $\widetilde{\mc L}_a$ with an
  \emph{exact} fundamental system. This reconstruction is defined in
  Definitions \ref{def:coeff} and \ref{def:LF} in the technical part.
 We
show that $\widetilde{\mc L}_a$ is close to $\mc L_a$ in a suitable
sense.
This is the content of Section \ref{sec:coeff}.
  While $\widetilde{\mc L}_a$ may not have a nontrivial
  kernel, it certainly has an eigenvalue very close to $0$ due to
  phase invariance and it is not feasible to invert it directly.

  \item \label{itm:linfunc} We
  employ the variation of constants formula
 and the explicit fundamental system for $\widetilde{\mc L}_a$ to construct
  a formal inverse $\mc H_a$ of $\widetilde{\mc L}_a$ in the ODE
  sense, see Section \ref{sec:inh}.
That is to say, $\widetilde{\mc L}_a\mc H_ag=g$ on $(-1,1)$ but
$\mc H_a$ is not
  an inverse of $\widetilde{\mc L}_a$ in the functional-analytic
  sense, i.e., $\mc H_a$ does not map into the right space. Concretely,
$\mc H_ag$ is singular at the left endpoint
$-1$, unless $\psi_a(g)=0$ for a suitable linear functional $\psi_a$
that is easily read off from the variation of constants formula.
More precisely, upon writing down the variation of constants formula
one identifies an integral term that necessarily needs to vanish if
$\mc H_ag$ is supposed to be regular at both endpoints. This
integral defines the functional $\psi_a$. The precise
definition of the functional is given in Section \ref{sec:reg}, in particular
Definitions \ref{def:MF}, \ref{def:alpha}, and \ref{def:psiF}.
From the spectral theory point of view, the kernel of $\psi_a$ is
precisely the complement of the eigenspace of the eigenvalue of
$\widetilde{\mc L}_a$ close to $0$ and the vanishing of $\psi_a(g)$ is
reminiscent of an orthogonality condition.

\item \label{itm:reg}
  Next, we add a regularization term of the
  form $\psi_a(g)f_0$, where
  $f_0$ is a suitable fixed function that
  is supported near the left endpoint $-1$ and
  singular at $-1$.
  With this, we define a new operator
  \[ \mc J_a(g):=\mc H_a(g)+\psi_a(g)f_0 \]
  which now has nice mapping properties but is of course no longer a
  formal inverse of $\widetilde{\mc L}_a$, unless $\psi_a(g)=0$.
  The precise construction of this new operator is in Section
  \ref{sec:reg}, in particular Definition \ref{def:J}.

\item \label{itm:contr} We then solve the equation
  \begin{equation}
    \label{eq:fpstrategy}
      f=\mc J_a\left (\widetilde{\mc L}_a f-\mc L_a f-\mc N_a(f)-\mc R(a_*+a,f_*(\cdot,a))\right)
    \end{equation}
    by applying the contraction mapping principle. The precise
    formulation of the fixed point problem is given in Lemma
    \ref{lem:solsys} and the system is solved in Section
    \ref{sec:end}.
    This yields a solution $f=f_a$ for
    any sufficiently small $a$. Note that the implementation of this
    fixed point argument requires
    quantitative bounds for all of the terms involved. The theoretical
    estimates for this purpose
    are derived in Section \ref{sec:quant} and the concrete bounds
    are obtained in Sections \ref{sec:op},
    \ref{sec:coeff}, and \ref{sec:psi} using computer assistance
    based on the algorithms described in Section \ref{sec:poly}.

  \item \label{itm:a} Finally, we show that
    \begin{equation}
      \label{eq:intropsia}
      a\mapsto \psi_a\left (\widetilde{\mc L}_a f_a-\mc L_a f_a-\mc
        N_a(f_a)-\mc R(a_*+a,f_*(\cdot,a))\right)
      \end{equation}
    changes sign. This is done by first showing that the dominant
    contribution comes from
    \[ \psi_a(\mc R(a_*+a,f_*(\cdot,a)) \]
    and, second, by explicitly evaluating a suitable approximation to this term
    at $a=\pm 10^{-10}$.
    This is carried out in Section \ref{sec:psi}.
    By continuity and the intermediate value theorem,
    the function \eqref{eq:intropsia} vanishes at
    some\footnote{We make no claims concerning the uniqueness of $a$.} value of
    $a$ in the interval $[-10^{-10},10^{-10}]$, see Proposition \ref{prop:psi}. For such a value of $a$, we have
    \begin{align*}
      \widetilde{\mc L}_af_a&=\widetilde{\mc L}_a\mc J_a\left (\widetilde{\mc L}_a f_a
                            -\mc L_a f_a-\mc N_a(f_a)-\mc R(a_*+a,f_*(\cdot,a))\right)  \\
      &=\widetilde{\mc L}_a f_a
                            -\mc L_a f_a-\mc N_a(f_a)-\mc R(a_*+a,f_*(\cdot,a))
    \end{align*}
    or, equivalently, $\mc R(a_*+a,f_*(\cdot,a)+f_a)=0$ and the
    existence of the desired solution follows.
    Technically, this is one of the most challenging parts of the
    proof because extracting the sign change requires very precise quantitative estimates for highly
    oscillatory integrals\footnote{The oscillatory behavior comes from the
    asymptotics of the fundamental system near the singular endpoint $1$.}.
  \end{enumerate}

  \subsection{Function spaces}
  \label{sec:funcspec}
  As a motivation for our choice of function spaces, it is
  instructive to
  consider the linearization around $0$, i.e.,
  \[ \mc L(\alpha,f)(y):=f''(y)+p_0(y,\alpha)f'(y)+q_0(y,\alpha)f(y). \]
 This is not the linear operator we have to deal with but to
 leading order the behavior
 near the singular endpoints $\pm 1$ is the
 same.
In order to determine the asymptotics of solutions to $\mc L(\alpha,f)=0$, we make the rough ansatz
$f(y)\sim e^{(1-y)^{-\sigma}}$. Since $p_0(y,\alpha)\sim (1-y)^{-3}$,
we must have $\sigma=2$. This motivates the refined ansatz
\[ f(y)\sim e^{a_2(1-y)^{-2}+a_1(1-y)^{-1}+a_0\log(1-y)} \]
and by comparing orders we determine the coefficients $a_2=-2i\alpha$, $a_1=2i\alpha$,
and $a_0=1-\frac{2i}{\alpha}$. Consequently, near $1$, we expect a
singular oscillatory solution $\widehat f_1(y,\alpha)\sim
(1-y)e^{i\varphi(y,\alpha)}$ with
\[
  \varphi(y,\alpha)=-\frac{2\alpha}{(1-y)^2}+\frac{2\alpha}{1-y}-\frac{2}{\alpha}\log(1-y). \]
Furthermore, by a standard power series ansatz we expect a
smooth solution $f_1(y,\alpha)\sim 1$.
Near $-1$ the situation is more classical because the singularity at
$y=-1$ is regular. By employing the Frobenius method, we find two
solutions $f_{-1}(y,\alpha)\sim 1$ and $\widehat f_{-1}(y,\alpha)\sim
(1+y)^{-1}$.
For the Wronskian we obtain the expression
\begin{align*}
  W(f_{-1}(\cdot,\alpha), f_1(\cdot,\alpha))(y)
  &=W(f_{-1}(\cdot,\alpha),f_1(\cdot,\alpha))(0) e^{-\int_0^y p_0(x,\alpha)dx} \\
  &=W(f_{-1}(\cdot,\alpha),f_1(\cdot,\alpha))(0)(1-y^2)^{-2}e^{i\varphi(y,\alpha)}=:W_*(y,\alpha).
\end{align*}
Note that all the local solutions extend to $(-1,1)$ because the
coefficients of the equation are smooth on $(-1,1)$.
Consequently, assuming that $W(f_{-1}(\cdot,\alpha),f_1(\cdot,\alpha))(0)\not=0$, the
variation of constants formula yields the inverse operator
\[ \mc L^{-1}(\alpha,g)(y)=f_{-1}(y,\alpha)\int_y^1
  \frac{f_{1}(x,\alpha)}{W_*(x,\alpha)}g(x)dx+f_1(y,\alpha)\int_{-1}^y \frac{f_{-1}(x,\alpha)}{W_*(x,\alpha)}g(x)dx.
\]
In order to obtain global estimates, we note that the function
$f_{-1}$ must be a linear combination of $f_{1}$ and $\widehat f_1$.
Analogously, $f_1$ is a linear
combination of $f_{-1}$ and $\widehat f_{-1}$. This yields the global estimates
\[ |f_{-1}(y,\alpha)|\lesssim 1,\qquad |f_{-1}'(y,\alpha)|\lesssim
  (1-y)^{-2}\]
and
\[ |f_1(y,\alpha)|\lesssim (1+y)^{-1},\qquad |f_1'(y,\alpha)|\lesssim (1+y)^{-2}. \]
Consequently, we infer the bounds
\begin{align*}
|\mc L^{-1}(\alpha,g)(y)|&\lesssim \int_y^1
                    (1+x)(1-x)^2|g(x)|dx+(1+y)^{-1}\int_{-1}^y
                    (1-x^2)^2|g(x)|dx \\
  &\lesssim \sup_{x\in (-1,1)}(1+x)(1-x)^2|g(x)|
\end{align*}
and
\begin{align*}
  |(1-y^2)\mc L^{-1}(\alpha,g)'(y)|
  &\lesssim (1-y)^{-1}\int_y^1
    (1+x)(1-x)^2|g(x)|dx +(1+y)^{-1} 
    \int_{-1}^y (1-x^2)^2 |g(x)|dx \\
  &\lesssim \sup_{x\in (-1,1)}(1+x)(1-x)^2|g(x)|.
\end{align*}
Motivated by this, we define the norms
\begin{align*}
  \|f\|_X&:=\sup_{y\in(-1,1)}(1-y^2)|f'(y)|+\|f\|_{L^\infty(-1,1)} \\
  \|f\|_Y&:=\sup_{y\in (-1,1)}(1+y)(1-y)^2|f(y)|
  \end{align*}
and obtain the bound
\[ \|\mc L^{-1}(\alpha,g)\|_X\lesssim \|g\|_Y. \]
As a consequence, the norm $\|\cdot\|_Y$
will be used to measure smallness of the source terms in the fixed
point argument.

These heuristic considerations suggest
to base the functional-analytic treatment on the spaces $X$ and $Y$ as
given in Definition \ref{def:XY} below.
Note, however, that $\mc R(\alpha,f)$ is
cubically singular at $1$ and hence does not belong to $Y$ in general. As a consequence, we are forced to include
a suitable correction of the approximate solution at the right
endpoint. Luckily, this correction can be determined explicitly
because the condition
\[ \lim_{y\to 1-}(1-y)^3\mc R(\alpha,f)(y)=0 \]
involves only terms that are
\emph{linear} in $f$.
Finally, the spaces $X$ and $Y$ are also compatible with the
nonlinear structure. Indeed, the required local Lipschitz bound
\[ \|\mc N_a(f)-\mc N_a(g)\|_Y\lesssim (\|f\|_X+\|g\|_X)\|f-g\|_X, \]
is easy to establish. 

In the true implementation we are facing the additional technical difficulty that
the linearized operator is not linear over $\C$
but only linear over $\R$. This is because the nonlinearity is not a
holomorphic function. As a consequence, we need to separate real and
imaginary parts and treat the
linearized equation as a coupled system. Fortunately, the terms that
spoil the linearity over $\C$ are
of higher order and the above asymptotic analysis
remains essentially valid. 
A more substantial difference to the simple heuristics here is the fact
that we will be forced to choose a fundamental
system that involves a function that is bad at both endpoints. This
is due to the nontrivial kernel of the linearized operator induced by the phase
invariance
and eventually leads to the vanishing condition on the
functional $\psi_a$ as described before.

  \subsection{Computer assistance}

  Our approach splits into a theoretical part and a concrete
  implementation. The latter is mildly
  computer-assisted and based on polynomial approximations. More precisely,
  we use polynomials with \emph{rational coefficients} that are
  provided explicitly in tables, see Appendix \ref{apx:tables}. At
  rational points these polynomials can be
  evaluated \emph{exactly} using elementary fraction
  arithmetic. Indeed, one simply inserts the rational point and
  computes the value of the polynomial using exact fraction
  arithmetic.
  This is one of the key features of our approach. In fact,
  evaluating polynomials with rational coefficients at rational points
  is the only
  component of our proof that requires computer assistance. We do not use any kind of
  symbolic computation, interval arithmetic or more sophisticated
  computer-assisted tools. In particular, all of our computations are
  exact and we never have to deal with
  round-off errors in floating-point arithmetic.
  Of course, the polynomials we use are initially obtained by some suitable
  numerical method (usually Chebyshev pseudo-spectral techniques, see
  Appendix \ref{apx:numerical}) but
  their origin is completely irrelevant to the argument and logically
  independent of the proof.

  In the technical implementation, we
  work with three different
  polynomial representations in parallel. First, we use the
  \emph{Lagrange representation} that is defined by evaluating the polynomial at a
  prescribed set of (rational) points. This representation is the easiest and most
  straightforward one to
  obtain, regardless in what form the polynomial is given. Second, we
  use the \emph{monomial representation} in order to
  compute derivatives of polynomials without symbolic computation. Lastly, we use the
  \emph{Chebyshev representation}, where the polynomial is expanded in a
  Chebyshev basis. The latter is crucial in order to compute
  what we call the \emph{T-norm}, essentially the $\ell^1$-norm of the
  Chebyshev expansion coefficients. This norm can be computed exactly and acts as a substitute
  for the $L^\infty$-norm. The $T$-norm together with explicit
  evaluation turns out to be sufficient to rigorously prove all the estimates we
  need. The linear maps that are used to switch between the
  three representations are provided explicitly and all of our
  computations 
can be verified straight away by using any programming language that
  supports fraction arithmetic. An implementation in
 \emph{Wolfram Mathematica} and data files
  containing the coefficients of Appendix \ref{apx:tables}
  are available on the arXiv page of this
  paper, \url{https://arxiv.org/abs/2406.16597}.

\subsection{Notation}
    To conclude the introduction, we finally explain some less
    standard notation that we prefer to use in the context of
    matrices.
    Let $A\in\R^{n\times n}$. We denote by $A_k$ the $k$-th column
    vector and by $A^j{}_k$ the element in the $j$-th row and the
    $k$-th column, e.g.
    \[ A=(A_1,A_2,A_3,A_4)=
      \begin{pmatrix}
        A^1{}_1 & A^1{}_2 & A^1{}_3 & A^1{}_4 \\
        A^2{}_1 & A^2{}_2 & A^2{}_3 & A^2{}_4 \\
        A^3{}_1 & A^3{}_2 & A^3{}_3 & A^3{}_4 \\
        A^4{}_1 & A^4{}_2 & A^4{}_3 & A^4{}_4 \\
      \end{pmatrix}
    \]
    for $n=4$. Accordingly, $\delta$ is the unit matrix with elements
    $\delta^j{}_k$ given by the Kronecker symbol and $\delta_k$ is the
    $k$-th unit vector. The determinant $\det$
    is the unique $n$-form on $\R^n$ that satisfies
    $\det(\delta_1,\delta_2,\dots,\delta_n)=1$ and as usual, we write
    \[ \det(A):=\det(A_1,A_2,\dots,A_n). \]
    The adjugate matrix $\adj(A)$ is defined by
    \[
      \adj(A)^j{}_k:=\det(A_1,\dots,A_{j-1},\delta_k,A_{j+1},\dots,A_n). \]
We have
    \begin{align*}
      \sum_{\ell=1}^n\adj(A)^j{}_\ell A^\ell{}_k
      &=\sum_{\ell=1}^n\det(A_1,\dots,A_{j-1},A^\ell{}_k\delta_\ell,A_{j+1},\dots,
        A_n) \\
      &=\det(A_1,\dots,A_{j-1},A_k,A_{j+1},\dots,A_n) \\
      &=\delta^j{}_k \det(A),
    \end{align*}
    which, if $\det(A)\not=0$, yields 
    Cramer's rule
    \[ A^{-1}=\frac{\adj(A)}{\det(A)}. \]

Finally, we remark that for the sake of readability we have occasionally
simplified the notation in this introductory section. In the technical
part we often use more complicated symbols that carry more
information. 
This is just as a warning in case the reader wants to quickly browse through the
technical part, looking for a specific object that is mentioned in the introduction.

\subsection{Structure of the technical part}

Unfortunately, due to the complexity of the problem, the
technical part cannot simply follow the strategy outlined in Section
\ref{sec:strategy} in a linear fashion. For instance, constructing the
different objects in play is largely independent of the concrete approximate
solution and much cleaner to describe in a more
abstract fashion. Therefore, we start with a general
theory and only later we ``insert'' the concrete approximate
solution. In the following, we give a detailed outline of the
structure of the rest of the paper and
relate each part to the strategy of the proof described in Section
\ref{sec:strategy}.

\begin{itemize}
\item In Section \ref{sec:theo}, we develop the theoretical framework to
construct a solution near a given approximate solution. This is a
purely theoretical part that makes no use of a specific polynomial
approximation. We define the linearized operator mentioned in Section
\ref{sec:strategy}, point (\ref{itm:lin}), the approximate linear operator,
Section \ref{sec:strategy}, point (\ref{itm:approxlin}), the linear
functional described in Section \ref{sec:strategy}, point
(\ref{itm:linfunc}), and the regularization procedure, Section
\ref{sec:strategy}, point (\ref{itm:reg}). Section \ref{sec:theo} thus lays the
theoretical foundation for the whole proof and defines all the key
objects in a general way, independent of a concrete approximate solution.

 \item In Section \ref{sec:quant}, we provide quantitative
estimates for the operator norms in play, again on a purely
theoretical basis without making reference to a particular
approximation. That is to say, we derive estimates that yield
computable bounds once one inserts a concrete approximation. Section
\ref{sec:quant} therefore provides the theoretical foundation for
obtaining the quantitative bounds used in the fixed point argument as
explained in Section \ref{sec:strategy}, point (\ref{itm:contr}).

\item In Section \ref{sec:poly}, we introduce our methods for
  estimating polynomials and rational functions. We give explicit,
  completely elementary,
  algorithms that do not rely on any sophisticated software packages
  and only require fraction arithmetic.

\item In Section
\ref{sec:approx}, we define the concrete approximations and prove
basic properties. This concerns the approximate solution as mentioned
in Section \ref{sec:strategy}, point (\ref{itm:approx}), and the
approximate linear operator, Section \ref{sec:strategy}, point (\ref{itm:approxlin}).

\item In Section \ref{sec:op}, we use the estimates from
Section \ref{sec:quant} to obtain explicit bounds on the operator
norms for the fixed point argument described in Section
\ref{sec:strategy}, point (\ref{itm:contr}).

\item In
Section \ref{sec:coeff}, we prove bounds on the approximate linear
operator that also enter the fixed point argument of Section
\ref{sec:strategy}, point (\ref{itm:contr}).

\item In Section \ref{sec:psi}, we prove the oscillatory estimates
  on the functional $\psi_a$ as explained in Section
  \ref{sec:strategy}, point (\ref{itm:a}).

\item  Finally, in Section \ref{sec:end}, we
  complete the proof of Theorem \ref{thm:quant} by executing the fixed
  point argument of Section \ref{sec:strategy}, points
  (\ref{itm:contr}) and (\ref{itm:a}).

\item  Appendix \ref{apx:tables}
contains the tables with the explicit coefficients of the polynomials
we use throughout. Appendix \ref{apx:numerical} gives a description of
the Chebyshev pseudo-spectral method we employed in order to obtain the
numerical approximations.

\end{itemize}

\section{Theoretical framework}
\label{sec:theo}

\noindent In this section we develop the theory for the construction of a
solution to $\mc R(\alpha,f)=0$ close to a given ``approximate
solution''.

\subsection{Function spaces and basic properties}

We define the basic function spaces $X$ and $Y$ that will be used for
the fixed point argument. The space $Y$ will carry the source terms on
the right-hand side and the solution will belong to the space $X$. The
motivation for the choice of the spaces $X$ and $Y$ is given in
Section \ref{sec:funcspec}.

\begin{definition}
  \label{def:XY}
  We define the norms
  \begin{align*}
    \|f\|_X&:=\sup_{y\in (-1,1)}(1-y^2)|f'(y)|+\|f\|_{L^\infty(-1,1)} \\
    \|f\|_Y&:=\sup_{y\in (-1,1)}(1+y)(1-y)^2|f(y)|.
  \end{align*}
 Furthermore, we let $X$ be the vector space of all $f\in
 C([-1,1])\cap C^1(-1,1)$ such that $y\mapsto (1-y^2)f'(y)$ extends to a
 continuous function on $[-1,1]$. Similarly, we let $Y$ be the vector
 space of all functions $f\in C(-1,1)\setminus\{0\}$ such that $y\mapsto
 (1+y)(1-y)^2f(y)$ extends to a bounded continuous function on $[-1,1]\setminus\{0\}$.
\end{definition}

\begin{remark}
  Note that $(X,\|\cdot\|_X)$ and $(Y,\|\cdot\|_Y)$ are Banach spaces.
\end{remark}

Next, we define the residual operator $\mc R$ that represents the
nonlinear Schr\"odinger equation we intend to solve, see Section
\ref{sec:strategy}, point (\ref{itm:compactify}).

\begin{definition}
  \label{def:R}
  For $(\alpha,f)\in \R\setminus\{0\}\times C^2(-1,1)$ and $y\in (-1,1)$, we set
  \[ \mc
    R(\alpha,f)(y):=f''(y)+p_0(y,\alpha)f'(y)+q_0(y,\alpha)f(y)+\frac{f(y)|f(y)|^2}{(1-y)^2}, \]
  where
  \begin{align*}
  p_0(y,\alpha)&:=\frac{4i\alpha}{(1-y)^3}-\frac{2i\alpha}{(1-y)^2}-\frac{2+\frac{2i}{\alpha}}{1-y}+\frac{2}{1+y} \\
  q_0(y,\alpha)&:=-\frac{2-2i\alpha}{(1-y)^3}-\frac{\frac{1}{\alpha^2}-\frac{i}{\alpha}}{(1-y)^2}-\frac{1+\frac{i}{\alpha}}{1-y}-\frac{1+\frac{i}{\alpha}}{1+y}.
\end{align*}
\end{definition}

The next definition specifies the correct asymptotics of the
solution. Any function that has these asymptotics is called
\emph{admissible}, cf.~Section \ref{sec:strategy}, point (\ref{itm:approx}).

\begin{definition}
Let $a_*\in \R\setminus\{0\}$ and $f_*\in C^2([-1,1]\times \R\setminus\{-a_*\})$. Then the pair $(a_*, f_*)$ is called an
\emph{admissible approximation} if\footnote{A prime denotes the
  derivative with respect to the first variable, i.e.,
  $f_*'(y,a):=\partial_y f_*(y,a)$.}
\[ f_*'(1,a)=\frac{1-i(a_*+a)}{2i(a_*+a)}f_*(1,a) \]
for all $a\in \R\setminus\{-a_*\}$. An admissible approximation
$(a_*,f_*)$ is
called \emph{trivial} if $f_*=0$.
\end{definition}

The expression $\mc R(\alpha,f)(y)$ has a cubic singularity at the endpoint $y=1$ and
hence does not in general map into $Y$, which only allows for a
quadratic singularity. However, if the approximation is admissible in the
above sense, the cubic singularity cancels out. This shows that
admissible approximations have the correct asymptotics, cf.~Section
\ref{sec:strategy}, point (\ref{itm:approx}).

\begin{lemma}
  \label{lem:R}
  Let $(a_*, f_*)$ be an admissible approximation.
Then we have $\mc R(a_*+a, f_*(\cdot,a))\in Y$ for all $a\in
\R\setminus\{-a_*\}$. Furthermore, the map
\[ a\mapsto \mc R(a_*+a, f_*(\cdot,a)): \R\setminus\{-a_*\}\to Y \]
is continuous.
\end{lemma}

\begin{proof}
  Fix $a\in \R\setminus\{-a_*\}$.
  To begin with, we only have $\mc
  R(a_*+a,f_*(\cdot,a))(y)=O((1+y)^{-1}(1-y)^{-3})$ but
  \[ \lim_{y\to 1}(1-y)^3\mc R(a_*+a,f_*(\cdot,a))(y)=4i(a_*+a)
    f_*'(1,a)-[2-2i(a_*+a)]f_*(1,a)=0 \]
  and $f_*(\cdot,a)\in C^2([-1,1])$ 
  show that in fact $\mc
  R(a_*+a,f_*(\cdot,a))(y)=O((1+y)^{-1}(1-y)^{-2})$.
Furthermore, the map
  \[ (y,a)\mapsto (1+y)(1-y)^2\mc
    R(a_*+a,f_*(\cdot,a))(y): (-1,1)\times \R\setminus\{-a_*\}\to \C \]
  extends continuously to $[-1,1]\times \R\setminus\{-a_*\}$ and is
  thus uniformly continuous with respect to $y$. 
\end{proof}

\subsection{Perturbative decomposition}

Next, we define the perturbative decomposition as outlined in Section
\ref{sec:strategy}, point (\ref{itm:lin}). This yields the linear
operator $\mc L$ and the nonlinear operator $\mc N$.

\begin{definition}
\label{def:LN}
  Let $(a_*,f_*)$ be an admissible approximation.
  For $(a,f)\in \R\setminus\{-a_*\}\times C^2(-1,1)$ and $y\in (-1,1)$, we set
  \begin{align*} \mc
    L(a_*,f_*,a,f)(y):=&f''(y)+p_0(y,a_*+a)f'(y)+q_0(y,a_*+a)f(y) \\
    &+\frac{2|f_*(y,a)|^2}{(1-y)^2}f(y)+\frac{f_*(y,a)^2}{(1-y)^2}\overline{f(y)}
  \end{align*}
  and
  \[ \mc
    N(a_*,f_*,a,f)(y):=\frac{\overline{f_*(y,a)}f(y)^2+2f_*(y,a)|f(y)|^2+f(y)|f(y)|^2}{(1-y)^2}. \]
\end{definition}

\begin{lemma}[Perturbative decomposition]
  \label{lem:pert}
  Let $(a_*,f_*)$ be an admissible approximation and
  $(a,f)\in \R\setminus\{-a_*\}\times C^2(-1,1)$. Then we have
  \[ \mc R(a_*+a, f_*(\cdot,a)+f)=\mc L(a_*,f_*,a,f)+\mc N(a_*,f_*, a,f)+\mc R(a_*+a,
    f_*(\cdot,a)). \]
\end{lemma}

\begin{proof}
  This follows immediately by inserting the definitions.
\end{proof}

\subsection{The admissible right fundamental matrix}

Next, we introduce the notion of an admissible fundamental system of the linear
operator. We split the domain in a right half and a left half for
$y\in [0,1)$ and $y\in (-1,0]$, respectively. The construction is done in the
first-order formulation and hence we obtain fundamental
\emph{matrices}.

In order to efficiently translate between the scalar and the first-order
formulation, the following map is important.

\begin{definition}
  \label{def:phi}
  We define $\phi: \C\to \R^{2\times 2}$ by
  \[ \phi(z):=
    \begin{pmatrix}
      \Re z & -\Im z \\
      \Im z & \Re z
    \end{pmatrix}. \]
\end{definition}

We work with polynomial approximations and for this purpose it is
useful to have an efficient notation to express properties of
functions  modulo polynomials.

\begin{definition}
  Let $f: \C\to \C$. We write
  $f(y)=\mc P(y)$ if $f$ is a polynomial.  
\end{definition}

We define the function $\varphi$ that appears as the phase in the
asymptotic behavior of the fundamental system, cf.~Section \ref{sec:funcspec}.

\begin{definition}
  For $\alpha\in \R\setminus\{0\}$ and $y\in (-1,1)$, we set
  \[
    \varphi(y,\alpha):=-\frac{2\alpha}{(1-y)^2}+\frac{2\alpha}{1-y}-\frac{2}{\alpha}\log(1-y). \]
\end{definition}

The following definition incorporates the correct asymptotics of the
fundamental system near the singular endpoint $y=1$. Since we are
dealing with a $4\times 4$ real system (coming from a scalar complex
second-order equation), each fundamental system consists of four
functions and we put them together in the fundamental matrix. 

\begin{definition}
  For $a_*\in \R\setminus\{0\}$ let $J\subset \R\setminus\{-a_*\}$ be compact. Furthermore, let
  $f_1,f_2,f_3,f_4: [0,1)\times J\to\C$ and set
  \[ F(y,a):=\begin{pmatrix}
      \Re f_1(y,a) & \Re f_2(y,a) & \Re f_3(y,a) & \Re f_4(y,a) \\
      \Im f_1(y,a) & \Im f_2(y,a) & \Im f_3(y,a) & \Im f_4(y,a) \\
      \Re f_1'(y,a) & \Re f_2'(y,a) & \Re f_3'(y,a) & \Re f_4'(y,a) \\
      \Im f_1'(y,a) & \Im f_2'(y,a) & \Im f_3'(y,a) & \Im f_4'(y,a) \\      
    \end{pmatrix}. \]
 Then $F$ is called an \emph{admissible right fundamental matrix on $J$}
  if $\det (F(y,a))\not= 0$ for all $(y,a)\in [0,1)\times J$ and
  \begin{align*}
    f_1(y,a)&=1+\frac{a_*+a+i}{2(a_*+a)}(1-y)+(1-y)^2\mc P(y) \\
    f_2(y,a)&=i+i\frac{a_*+a+i}{2(a_*+a)}(1-y)+(1-y)^2\mc P(y) \\
    f_3(y,a)&=e^{i\varphi(y,a_*+a)}(1-y)\left [1+\frac{2(a_*+a)-i}{2(a_*+a)}(1-y)+(1-y)^2\mc
              P(y)\right ] \\
    &\quad +e^{-i\varphi(y,a_*+a)}(1-y)^5\mc P(y) \\
    f_4(y,a)&=ie^{i\varphi(y,a_*+a)}(1-y)\left [1+\frac{2(a_*+a)-i}{2(a_*+a)}(1-y)+(1-y)^2\mc
              P(y)\right ] \\
    &\quad +ie^{-i\varphi(y,a_*+a)}(1-y)^5\mc P(y).
  \end{align*}
  The number $a_*$ is called the \emph{parameter of $F$}.
\end{definition}

\begin{remark}
  \label{rem:FRcont}
  Observe that $f_j: [0,1)\times J\to \C$ for $j\in \{1,2,3,4\}$
  extends continuously to the compact domain $[0,1]\times J$ and thus,
  $f_j$ is uniformly continuous.
\end{remark}

Thanks to the function $\phi$ in Definition \ref{def:phi}, the
fundamental matrix can be treated like a $2\times 2$ matrix, at least
to leading order. This gives an easy expression for its
inverse. Using this, we derive the crucial
asymptotics of the inverse of the fundamental matrix towards the
singular endpoint $y=1$.

\begin{lemma}
  \label{lem:FR-1}
  Let $F$ be an admissible right fundamental matrix on $J$ with
  parameter $a_*$ and for $j\in \{1,2,3,4\}$ set
  $f_j:=F^1{}_j+iF^2{}_j$. Then there exists a $\delta\in (0,1]$ such
  that
  \[ F(y,a)^{-1}=
    \begin{pmatrix}
      \phi(\frac{f_3'(y,a)}{W(y,a)}) & -\phi(\frac{f_3(y,a)}{W(y,a)}) \\
      -\phi(\frac{f_1'(y,a)}{W(y,a)}) & \phi(\frac{f_1(y,a)}{W(y,a)})
    \end{pmatrix}
    [1+O((1-y)^2a^0)]
  \]
  for all $(y,a)\in [1-\delta,1)\times J$, where
  \begin{align*}
    W(y,a):&=f_1(y,a)f_3'(y,a)-f_1'(y,a)f_3(y,a) \\
    &=-\frac{4i(a_*+a)}{(1-y)^2}e^{i\varphi(y,a_*+a)}\left
    [1+(1-y)+O((1-y)^2a^0)\right].
  \end{align*}
\end{lemma}

\begin{remark}
  More precisely, we mean that
  \[ F(y,a)^{-1}=
    \begin{pmatrix}
      \phi(\frac{f_3'(y,a)}{W(y,a)})[1+O((1-y)^2a^0)] & -\phi(\frac{f_3(y,a)}{W(y,a)})[1+O((1-y)^2a^0)] \\
      -\phi(\frac{f_1'(y,a)}{W(y,a)})[1+O((1-y)^2a^0)] & \phi(\frac{f_1(y,a)}{W(y,a)})[1+O((1-y)^2a^0)]
    \end{pmatrix}
  \]
  but here and in the following, we will use the sloppier but more
  compact notation of Lemma \ref{lem:FR-1}.
\end{remark}

\begin{proof}[Proof of Lemma \ref{lem:FR-1}]
  We commence with the statement on the Wronskian $W$. Note that
  $f'_1(y,a)=O(y^0a^0)$
  and
  \begin{align*}
    f_3'(y,a)&=i\varphi'(y,a_*+a)e^{i\varphi(y,a_*+a)}(1-y)\left
               [1+\frac{2(a_*+a)-i}{2(a_*+a)}(1-y)+O((1-y)^2a^0)\right
               ]
               \\
             &\quad +O(y^0a^0).
  \end{align*}
  Thus, since
  \begin{align*}
    \varphi'(y,a_*+a)&=-\frac{4(a_*+a)}{(1-y)^3}+\frac{2(a_*+a)}{(1-y)^2}+\frac{2}{(a_*+a)(1-y)} \\
  &=-\frac{4(a_*+a)}{(1-y)^3}\left
    [1-\tfrac12(1-y)+O((1-y)^2a^0)\right],
  \end{align*}
  we obtain
 \begin{align*}
    f_3'(y,a)&=-\frac{4i(a_*+a)}{(1-y)^2}e^{i\varphi(y,a_*+a)}\left
      [1-\tfrac12(1-y)+O((1-y)^2a^0) \right] \\
      &\quad\times\left[
        1+\frac{2(a_*+a)-i}{2(a_*+a)}(1-y)+O((1-y)^2a^0)\right] \\
    &=-\frac{4i(a_*+a)}{(1-y)^2}e^{i\varphi(y,a_*+a)}\left
      [1+\frac{a_*+a-i}{2(a_*+a)}(1-y)+O((1-y)^2a^0)\right].
  \end{align*}
  Consequently,
  \begin{align*}
    W(y,a)&=f_1(y,a)f_3'(y,a)+O((1-y)a^0) \\
    &=-\frac{4i(a_*+a)}{(1-y)^2}e^{i\varphi(y,a_*+a)}\left
      [1+\frac{a_*+a+i}{2(a_*+a)}(1-y)+O((1-y)^2a^0)\right] \\
    &\quad\times\left
      [1+\frac{a_*+a-i}{2(a_*+a)}(1-y)+O((1-y)^2a^0)\right] \\
    &=-\frac{4i(a_*+a)}{(1-y)^2}e^{i\varphi(y,a_*+a)}\left
      [1+(1-y)+O((1-y)^2a^0)\right],
  \end{align*}
  as claimed. Furthermore, it follows that we can find a $\delta\in (0,1]$ such that
  $W(y,a)\not=0$ for all $(y,a)\in [1-\delta,1)\times J$.
  
  Now observe that
  \begin{align*}
    f_2(y,a)&=if_1(y,a)+O((1-y)^2) \\
    f_2'(y,a)&=if_1'(y,a)+O(1-y) \\
    f_4(y,a)&=if_3(y,a)+O((1-y)^3a^0) \\
    f_4'(y,a)&=if_3'(y,a)+O(y^0a^0)
  \end{align*}
  and thus,
  \[ F(y,a)=
    \begin{pmatrix}
      \phi(f_1(y,a)) & \phi(f_3(y,a)) \\
      \phi(f_1'(y,a)) & \phi(f_3'(y,a))
    \end{pmatrix}+
    \begin{pmatrix}
      0 & O((1-y)^2) & 0 & O((1-y)^3a^0) \\
      0 & O((1-y)^2) & 0 & O((1-y)^3a^0) \\
      0 & O(1-y) & 0 & O(y^0a^0) \\
            0 & O(1-y) & 0 & O(y^0a^0) 
    \end{pmatrix}.
  \]
  Since $\phi(z_1z_2)=\phi(z_1)\phi(z_2)$ for all $z_1,z_2\in \C$, the first matrix can be
  formally inverted like a $2\times 2$-matrix and we obtain
  \begin{align*}
    F(y,a)&=\begin{pmatrix}
      \phi(f_1(y,a)) & \phi(f_3(y,a)) \\
      \phi(f_1'(y,a)) & \phi(f_3'(y,a))
            \end{pmatrix} \\
    &\quad\times
            \left [
            1+\begin{pmatrix}
      \phi(\frac{f_3'(y,a)}{W(y,a)}) & -\phi(\frac{f_3(y,a)}{W(y,a)}) \\
      -\phi(\frac{f_1'(y,a)}{W(y,a)}) & \phi(\frac{f_1(y,a)}{W(y,a)})
              \end{pmatrix}
            \begin{pmatrix}
      0 & O((1-y)^2) & 0 & O((1-y)^3a^0) \\
      0 & O((1-y)^2) & 0 & O((1-y)^3a^0) \\
      0 & O(1-y) & 0 & O(y^0a^0) \\
            0 & O(1-y) & 0 & O(y^0a^0) 
            \end{pmatrix}\right] \\
    &=\begin{pmatrix}
      \phi(f_1(y,a)) & \phi(f_3(y,a)) \\
      \phi(f_1'(y,a)) & \phi(f_3'(y,a))
      \end{pmatrix}
      \left [1+O((1-y)^2a^0)\right]
  \end{align*}
  for all $(y,a)\in [1-\delta,1)\times J$.
  This yields the stated form of $F^{-1}$.
\end{proof}

\subsection{The admissible left fundamental matrix}
Now we turn to the fundamental system in the left half, i.e., for
$y\in (-1,0]$. This is much simpler because the singularity at $y=-1$
is regular and hence easier to handle. In particular, the asymptotics
are independent of the parameter $a$.

\begin{definition}
  Let $f_1,f_2,f_3,f_4: (-1,0]\to \C$ and set
  \[ F:=
    \begin{pmatrix}
      \Re f_1 & \Re f_2 & \Re f_3 & \Re f_4 \\
      \Im f_1 & \Im f_2 & \Im f_3 & \Im f_4 \\
      \Re f_1' & \Re f_2' & \Re f_3' & \Re f_4' \\
      \Im f_1' & \Im f_2' & \Im f_3' & \Im f_4' \\
    \end{pmatrix}. \]
Then $F$ is called an \emph{admissible left fundamental matrix}
  if $\det(F(y))\not=0$ for all $y\in (-1,0]$ and
  \begin{align*}
    f_1(y)&=1+(1+y)\mc P(y) & f_2(y)&=i+(1+y)\mc P(y) \\
    f_3(y)&=\frac{1}{1+y}[1+(1+y)\mc P(y)] &
                                               f_4(y)&=\frac{i}{1+y}[1+(1+y)\mc P(y)].
  \end{align*}
\end{definition}

As before with the right fundamental matrix, we derive the asymptotics
towards $y=-1$ of the inverse of the left fundamental matrix. Again,
this relies on the map $\phi$ in Definition \ref{def:phi} that allows us to treat the fundamental
matrix like a $2\times 2$ matrix.

\begin{lemma}
  \label{lem:FL-1}
  Let $F$ be an admissible left fundamental
  matrix. Then we
  have
  \[ F(y)^{-1}=
    \begin{pmatrix}
      1 & O(1+y) & 1+y & O((1+y)^2)\\
      O(1+y) & 1 & O((1+y)^2) & 1+y \\
      O((1+y)^2) & O((1+y)^2) & -(1+y)^2 & O((1+y)^3) \\
      O((1+y)^2) & O((1+y)^2) & O((1+y)^3) &
      -(1+y)^2
    \end{pmatrix}[1+O(1+y)]. \]
\end{lemma}

\begin{proof}
  For $j\in \{1,2,3,4\}$, set $f_j:=F^1{}_j+iF^2{}_j$.
   By definition, we have \begin{align*}
                           f_1(y)&=1+O(1+y) &
                                              f_3(y)&=\frac{1}{1+y}[1+O(1+y)] \\
  f_1'(y)&=O(y^0) &
          f_3'(y)&=-\frac{1}{(1+y)^2}[1+O(1+y)]
                   \end{align*}
and thus,
\[ W(y):=f_1(y)f_3'(y)-f_1'(y)f_3(y)=-\frac{1}{(1+y)^2}[1+O(1+y)]. \]
  As in the proof of Lemma \ref{lem:FR-1}, we have
  \begin{align*}
    F(y)&=
    \begin{pmatrix}
      \phi(f_1(y)) & \phi(f_3(y)) \\
      \phi(f_1'(y)) & \phi(f_3'(y))
    \end{pmatrix} \\
    &\quad\times
    \left [1+\begin{pmatrix}
      \phi(\frac{f_3'(y)}{W(y)}) & -\phi(\frac{f_3(y)}{W(y)}) \\
      -\phi(\frac{f_1'(y)}{W(y)}) & \phi(\frac{f_1(y)}{W(y)})
    \end{pmatrix}
      \begin{pmatrix}
        0 & O(1+y) & 0 & O(1+y) \\
        0 & O(1+y) & 0 & O(1+y) \\
        0 & O(y^0) & 0 & O(y^0) \\
        0 & O(y^0) & 0 & O(y^0)
      \end{pmatrix}
      \right ],
      \end{align*}
  provided that $y$ is sufficiently close to $-1$
  and thus,
  \[ F(y)^{-1}=
    \begin{pmatrix}
      \phi(\frac{f_3'(y)}{W(y)}) & -\phi(\frac{f_3(y)}{W(y)}) \\
      -\phi(\frac{f_1'(y)}{W(y)}) & \phi(\frac{f_1(y)}{W(y)})
    \end{pmatrix}
    [1+O(1+y)]. \]
This implies the claim.
\end{proof}

\subsection{The global  fundamental matrix}

Next, we glue together the left fundamental matrix and the right
fundamental matrix at $y=0$ so that the resulting global fundamental
matrix is continuous.

\begin{definition}
  \label{def:F}
  Let $F_R$ be an admissible right
  fundamental matrix on $J$ and let $F_L$ be an admissible left fundamental matrix.
  Then
  \[
    F(y,a):=1_{(-1,0]}(y)F_L(y)F_L(0)^{-1}F_R(0,a)+1_{(0,1)}(y)F_R(y,a) \]
  is called an \emph{admissible fundamental matrix on $J$} if
  \[ [F_L(0)^{-1}F_R(0,a)]^3{}_1\not= 0 \]
  for all $a\in J$. Furthermore, $F_L$ and $F_R$ are called the
  \emph{left part of $F$} and the \emph{right part of $F$}, respectively.
  Finally, the \emph{parameter of $F$} is defined as the parameter of $F_R$.
\end{definition}

\begin{remark}
  The condition $[F_L(0)^{-1}F_R(0,a)]^3{}_1\not= 0$ ensures that the
  function $F^1{}_1+iF^2{}_1$ is singular at the left endpoint
  $-1$. This will be crucial.
\end{remark}

The connection coefficient that glues together the two fundamental
matrices is called $M_F(a)$.

\begin{definition}
  \label{def:MF}
  Let $F$ be an admissible fundamental matrix on $J$.
Then we define $M_F: J\to \R^{4\times 4}$ by
\[ M_F(a):=F_L(0)^{-1}F_R(0,a), \]
where $F_L$ is the left part of $F$ and $F_R$ is the right part of $F$.
\end{definition}

Every admissible fundamental matrix induces a scalar operator which is
the approximate linear operator described in Section
\ref{sec:strategy}, point (\ref{itm:approxlin}). The coefficients of this
operator are called $p_1$, $p_2$, $q_1$, and $q_2$ and defined as follows.

\begin{definition}
   \label{def:coeff}
  Let $F$ be an admissible fundamental matrix on $J$.
  Then the functions
  $(p_1,p_2,q_1,q_2)$, defined on $(-1,1)\setminus\{0\}\times J$ by
   \begin{align*}
        p_1&:=\tfrac12(A^3{}_3+A^4{}_4)+\tfrac{i}{2}(A^4{}_3-A^3{}_4),
    &
      p_2&:=\tfrac12(A^3{}_3-A^4{}_4)+\tfrac{i}{2}(A^4{}_3+A^3{}_4), \\
    q_1&:=\tfrac12(A^3{}_1+A^4{}_2)+\tfrac{i}{2}(A^4{}_1-A^3{}_2),
    &
      q_2&:=\tfrac12(A^3{}_1-A^4{}_2)+\tfrac{i}{2}(A^4{}_1+A^3{}_2),
   \end{align*}
   where we abbreviate $A(y,a)=-F'(y,a)F(y,a)^{-1}$,
are called the \emph{coefficients associated to $F$}.
\end{definition}

Now we are ready to establish the connection between the matrix
formulation and the scalar formulation. The following result explains
how the fundamental matrix induces the linear operator. In addition,
we derive the crucial asymptotics of the coefficients.

\begin{proposition}
  \label{prop:F}
  Let $F$ be an admissible fundamental matrix on $J$ with parameter $a_*$ 
 and let
  \[ f_j(y,a):=F(y,a)^1{}_j+iF(y,a)^2{}_j \]
  for $j\in \{1,2,3,4\}$, $y\in (-1,1)$, and $a\in J$. Then $f_j(\cdot,a)\in
  C^2((-1,1)\setminus\{0\})\cap C^1(-1,1)$ for any $a\in J$ and 
  \[
    f_j''(y,a)+p_1(y,a)f_j'(y,a)+p_2(y,a)\overline{f_j'(y,a)}+q_1(y,a)f_j(y,a)+q_2(y,a)\overline{f_j(y,a)}=0 \]
  for all $(y,a)\in (-1,1)\setminus\{0\}\times
  J$ and $j\in
  \{1,2,3,4\}$, where $(p_1,p_2,q_1,q_2)$ are the coefficients
  associated to $F$. Furthermore,
    \begin{align*}
    p_1(y,a)&=\frac{4i(a_*+a)}{(1-y)^3}-\frac{2i(a_*+a)}{(1-y)^2}+\frac{2}{1+y}+O((1-y)^{-1}a^0) \\
                                                         p_2(y,a)&=O((1-y)^{-1}a^0) \\
    q_1(y,a)&=-\frac{2-2i(a_*+a)}{(1-y)^3}+O((1-y)^{-2}a^0)+
              O((1+y)^{-1}) \\
      q_2(y,a)&=O((1-y)^{-2}a^0)+O((1+y)^{-1}).
  \end{align*}
\end{proposition}

\begin{proof}
  Let $a\in J$ and denote by $F_L$ and $F_R$ the left part and the
  right part of $F$, respectively.
  By definition, we have
  \[ F(y,a)=
    \begin{pmatrix}
      \Re f_1(y,a) & \Re f_2(y,a) & \Re f_3(y,a) & \Re f_4(y,a) \\
      \Im f_1(y,a) & \Im f_2(y,a) & \Im f_3(y,a) & \Im f_4(y,a) \\
      \Re f_1'(y,a) & \Re f_2'(y,a) & \Re f_3'(y,a) & \Re f_4'(y,a) \\
      \Im f_1'(y,a) & \Im f_2'(y,a) & \Im f_3'(y,a) & \Im f_4'(y,a) 
    \end{pmatrix} \]
  for all $y\in (-1,1)\setminus\{0\}$ and $F(\cdot,a)\in
  C((-1,1),\R^{4\times 4})$. Moreover, $F_L\in
  C^2((-1,0),\R^{4\times 4})$
  and $F(\cdot,a)\in C^1((0,1),\R^{4\times 4})$. Consequently, $f_j(\cdot,a)\in
  C^2((-1,1)\setminus\{0\})\cap C^1(-1,1)$ for all $j\in
  \{1,2,3,4\}$. Let $A(y,a):=-F'(y,a)F(y,a)^{-1}$. Then we have
  \[ F'(y,a)+A(y,a)F(y,a)=0. \]
  Since $F'(y,a)^j{}=F(y,a)^{j+2}$ for $j\in \{1,2\}$, it follows that
  \[ A(y,a)^j{}_k=-\sum_{\ell=1}^4 F'(y,a)^j{}_\ell
    F^{-1}(y,a)^\ell{}_k=-\sum_{\ell=1}^4F(y,a)^{j+2}{}_\ell F^{-1}(y,a)^\ell{}_k=-\delta^{j+2}{}_k \]
and thus, $A$ is of the form
  \[ A=
    \begin{pmatrix}
      0 & 0 & -1 & 0 \\
      0 & 0 & 0 & -1 \\
      A^3{}_1 & A^3{}_2 & A^3{}_3 & A^3{}_4 \\
      A^4{}_1 & A^4{}_2 & A^4{}_3 & A^4{}_4
    \end{pmatrix}. \]
  Consequently, $F'(y,a)+A(y,a)F(y,a)=0$ is equivalent to
\[
  f_j''(y,a)+p_1(y,a)f_j'(y,a)+p_2(y,a)\overline{f_j'(y,a)}+q_1(y,a)f_j(y,a)+q_2(y,a)\overline{f_j(y,a)}=0 \]
for $j\in \{1,2,3,4\}$,
where
  \begin{align*}
        p_1&=\tfrac12(A^3{}_3+A^4{}_4)+\tfrac{i}{2}(A^4{}_3-A^3{}_4),
    &
      p_2&=\tfrac12(A^3{}_3-A^4{}_4)+\tfrac{i}{2}(A^4{}_3+A^3{}_4), \\
    q_1&=\tfrac12(A^3{}_1+A^4{}_2)+\tfrac{i}{2}(A^4{}_1-A^3{}_2),
    &
      q_2&=\tfrac12(A^3{}_1-A^4{}_2)+\tfrac{i}{2}(A^4{}_1+A^3{}_2).
  \end{align*}
  
  If $y\in (-1,0)$, we have
  \begin{align*}
    A(y,a)&=-F'(y,a)F(y,a)^{-1}=-F_L'(y)F_L(0)^{-1}F_R(0,a)F_R(0,a)^{-1}F_L(0)F_L(y)^{-1} \\
            &=-F_L'(y)F_L(y)^{-1}
  \end{align*}
  and thus, by definition and Lemma \ref{lem:FL-1},
  \begin{align*}
    A(y,a)
    &=
      -\begin{pmatrix}
        O(y^0) & O(y^0) & -(1+y)^{-2}[1+O(1+y)] & O((1+y)^{-1}) \\
        O(y^0) & O(y^0) & O((1+y)^{-1}) & -(1+y)^{-2}[1+O(1+y)] \\
        O(y^0) & O(y^0) & 2(1+y)^{-3}[1+O(1+y)] & O((1+y)^{-2}) \\
        O(y^0) & O(y^0) & O((1+y)^{-2}) & 2(1+y)^{-3}[1+O(1+y)] 
       \end{pmatrix} \\
    &\qquad\times
      \begin{pmatrix}
      1 & O(1+y) & 1+y & O((1+y)^2)\\
      O(1+y) & 1 & O((1+y)^2) & 1+y \\
      O((1+y)^2) & O((1+y)^2) & -(1+y)^2 & O((1+y)^3) \\
      O((1+y)^2) & O((1+y)^2) & O((1+y)^3) &
      -(1+y)^2
    \end{pmatrix}[1+O(1+y)] \\
    &=
      \begin{pmatrix}
        0 & 0 & -1 & 0 \\
        0 & 0 & 0 & -1 \\
                O((1+y)^{-1}) & O((1+y)^{-1}) & 2(1+y)^{-1}[1+O(1+y)] & O(y^0) \\
                O((1+y)^{-1}) & O((1+y)^{-1}) & O(y^0) & 2(1+y)^{-1}[1+O(1+y)] 
      \end{pmatrix},
  \end{align*}
  which yields the claimed asymptotics of the coefficients on the
  left-hand side.

  In the case $y\in (0,1)$, we have
  $A(y,a)=-F_R'(y,a)F_R(y,a)^{-1}$ and
  by definition,
  \begin{align*}
    f_2(y,a)&=if_1(y,a)+O((1-y)^2) \\
    f_2'(y,a)&=if_1'(y,a)+O(1-y) \\
    f_2''(y,a)&=if_1''(y,a)+O(y^0),
  \end{align*}
  as well as
  \begin{align*}
    f_4(y,a)&=if_3(y,a)+O((1-y)^3a^0) \\
    f_4'(y,a)&=if_3'(y,a)+O(y^0a^0) \\
    f_4''(y,a)&=if_3''(y,a)+O((1-y)^{-3}a^0).
  \end{align*}
  Consequently,
  \[ F_R'(y,a)=
        \begin{pmatrix}
            \phi(f_1'(y,a)) & \phi(f_3'(y,a)) \\
            \phi(f_1''(y,a)) & \phi(f_3''(y,a))
          \end{pmatrix}
          +
          \begin{pmatrix}
            0 & O(1-y) & 0 & O(y^0a^0) \\
            0 & O(1-y) & 0 & O(y^0a^0) \\
            0 & O(y^0) & 0 & O((1-y)^{-3}a^0) \\
            0 & O(y^0) & 0 & O((1-y)^{-3}a^0) 
        \end{pmatrix} \]
and, in view of Lemma \ref{lem:FR-1},
\begin{align*}
  A(y,a)&=-F'(y,a)F(y,a)^{-1} \\
  &=
    F_R'(y,a)
          \begin{pmatrix}
            -\phi(\frac{f_3'(y,a)}{W(y,a)}) &
                                              \phi(\frac{f_3(y,a)}{W(y,a)})
            \\
            \phi(\frac{f_1'(y,a)}{W(y,a)}) & -\phi(\frac{f_1(y,a)}{W(y,a)})
          \end{pmatrix}
    [1+O((1-y)^2a^0)] \\
  &=
    \begin{pmatrix}
      \phi(0) & -\phi(1) \\
      \phi(\frac{f_1'(y,a)f_3''(y,a)-f_1''(y,a)f_3'(y,a)}{W(y,a)}) & -\phi(\frac{f_1(y,a)f_3''(y,a)-f_1''(y,a)f_3(y,a)}{W(y,a)})
    \end{pmatrix}
    [1+O((1-y)^2a^0)] \\
  &\quad +
    \begin{pmatrix}
      O((1-y)a^0) & O((1-y)a^0) & O((1-y)^2a^0) & O((1-y)^2a^0) \\
      O((1-y)a^0) & O((1-y)a^0) & O((1-y)^2a^0) & O((1-y)^2a^0) \\
      O((1-y)^{-1}a^0) & O((1-y)^{-1}a^0) & O((1-y)^{-1}a^0) & O((1-y)^{-1}a^0) \\
            O((1-y)^{-1}a^0) & O((1-y)^{-1}a^0) & O((1-y)^{-1}a^0) & O((1-y)^{-1}a^0) 
    \end{pmatrix}.
\end{align*} 
In other words,
\begin{align*}
  A(y,a)&=
  \begin{pmatrix}
    0 & 0 & -1 & 0 \\
    0 & 0 & 0 & -1 \\
    \Re h_1(y,a) & -\Im h_1(y,a) & \Re h_2(y,a) & -\Im h_2(y,a) \\
    \Im h_1(y,a) & \Re h_1(y,a) & \Im h_2(y,a) & \Re h_2(y,a)
  \end{pmatrix}
  [1+O((1-y)^2a^0)] \\
  &\quad +\begin{pmatrix}
      0 & 0 & 0 & 0 \\
           0 & 0 & 0 & 0 \\
      O((1-y)^{-1}a^0) & O((1-y)^{-1}a^0) & O((1-y)^{-1}a^0) & O((1-y)^{-1}a^0) \\
            O((1-y)^{-1}a^0) & O((1-y)^{-1}a^0) & O((1-y)^{-1}a^0) & O((1-y)^{-1}a^0) 
    \end{pmatrix}
\end{align*}
with
\begin{align*}
  h_1(y,a)&:=\frac{f_1'(y,a)f_3''(y,a)-f_1''(y,a)f_3'(y,a)}{W(y,a)}\\
  h_2(y,a)&:=-\frac{f_1(y,a)f_3''(y,a)-f_1''(y,a)f_3(y,a)}{W(y,a)}.
\end{align*}
From Lemma \ref{lem:FR-1} we recall that
\begin{align*}
  W(y,a)&=-\frac{4i(a_*+a)}{(1-y)^2}e^{i\varphi(y,a_*+a)}\left
          [1+(1-y)+O((1-y)^2a^0)\right]
\end{align*}
and by definition, we have $f_3'(y,a)=O((1-y)^{-2}a^0)$ as well as
\begin{align*}
  f_3''(y,a)=&-\varphi'(y,a)^2e^{i\varphi(y,a_*+a)}(1-y)\left
               [1+\frac{2(a_*+a)-i}{2(a_*+a)}(1-y)+O((1-y)^2a^0) \right ] \\
  &+O((1-y)^{-3}a^0).
\end{align*}
Note that
\begin{align*}
  \varphi'(y,a)^2&=\left
  [-\frac{4(a_*+a)}{(1-y)^3}+\frac{2(a_*+a)}{(1-y)^2}+\frac{2}{(a_*+a)(1-y)}\right]^2 \\
  &=\frac{16(a_*+a)^2}{(1-y)^6}-\frac{16(a_*+a)^2}{(1-y)^5}+O((1-y)^{-4}a^0)
  \\
  &=\frac{16(a_*+a)^2}{(1-y)^6}\left [1-(1-y)+O((1-y)^2a^0)\right]
\end{align*}
and thus,
\begin{align*} f_3''(y,a)&=-\frac{16(a_*+a)^2}{(1-y)^5}e^{i\varphi(y,a_*+a)}\left
    [1-(1-y)+O((1-y)^2a^0)\right] \\
  &\quad \times\left
    [1+\frac{2(a_*+a)-i}{2(a_*+a)}(1-y)
    +O((1-y)^2a^0)\right] \\
  &=-\frac{16(a_*+a)^2}{(1-y)^5}e^{i\varphi(y,a_*+a)}\left
    [1-\frac{i}{2(a_*+a)}(1-y)+O((1-y)^2a^0)\right].
\end{align*}
This yields
\begin{align*}
  \frac{f_3''(y,a)}{W(y,a)}&=-\frac{4i(a_*+a)}{(1-y)^3}\left  [1-\frac{i}{2(a_*+a)}(1-y)
                             +O((1-y)^2a^0)\right] 
  \left
    [1-(1-y)+O((1-y)^2a^0)\right] \\
  &=-\frac{4i(a_*+a)}{(1-y)^3}\left [1-\frac{2(a_*+a)+i}{2(a_*+a)}(1-y)+O((1-y)^2a^0)\right]
\end{align*}
and by recalling that
$f_1(y,a)=1+\frac{a_*+a+i}{2(a_*+a)}(1-y)+(1-y)^2\mc P(y)$, we
obtain
\begin{align*}
  h_1(y,a)&=f_1'(y,a)\frac{f_3''(y,a)}{W(y,a)}+O((1-y)^0a^0)=\left
            [-\frac{a_*+a+i}{2(a_*+a)}+O(1-y)\right]
  \frac{f_3''(y,a)}{W(y,a)}+O((1-y)^0a^0) \\
            &=-\frac{2-2i(a_*+a)}{(1-y)^3}[1+O((1-y)a^0)]
\end{align*}
as well as
\begin{align*}
  h_2(y,a)&=-f_1(y,a)\frac{f_3''(y,a)}{W(y,a)}+O((1-y)^3a^0) \\
          &=\frac{4i(a_*+a)}{(1-y)^3}\left[1+\frac{a_*+a+i}{2(a_*+a)}(1-y)+O((1-y)^2a^0)\right] \\
  &\quad\times\left [1-\frac{2(a_*+a)+i}{2(a_*+a)}(1-y)+O((1-y)^2a^0)\right]\\
  &=\frac{4i(a_*+a)}{(1-y)^3}\left [1-\tfrac12(1-y)+O((1-y)^2a^0)\right].
\end{align*}
Consequently,
\begin{align*}
  p_1(y,a)&=\tfrac12(A(y,a)^3{}_3+A(y,a)^4{}_4)+\tfrac{i}{2}(A(y,a)^4{}_3-A(y,a)^3{}_4) \\
            &=h_2(y,a)[1+O((1-y)^2a^0)]+O((1-y)^{-1}a^0) \\
          &=\frac{4i(a_*+a)}{(1-y)^3}-\frac{2i(a_*+a)}{(1-y)^2}+O((1-y)^{-1}a^0)
\end{align*}
and
\begin{align*}
  p_2(y,a)&=\tfrac12(A(y,a)^3{}_3-A(y,a)^4{}_4)+\tfrac{i}{2}(A(y,a)^4{}_3+A(y,a)^3{}_4)
  \\
  &=h_2(y,a)O((1-y)^2a^0)+O((1-y)^{-1}a^0) \\
  &=O((1-y)^{-1}a^0),
\end{align*}
as well as
\begin{align*}
  q_1(y,a)&=\tfrac12(A(y,a)^3{}_1+A(y,a)^4{}_2)+\tfrac{i}{2}(A(y,a)^4{}_1-A(y,a)^3{}_2) \\
          &=h_1(y,a)[1+O((1-y)^2a^0)]+O((1-y)^{-1}a^0) \\
  &=-\frac{2-2i(a_*+a)}{(1-y)^3}+O((1-y)^{-2}a^0)
\end{align*}
and
\begin{align*}
  q_2(y,a)&=\tfrac12(A(y,a)^3{}_1-A(y,a)^4{}_2)+\tfrac{i}{2}(A(y,a)^4{}_1+A(y,a)^3{}_2)
  \\
  &=h_1(y,a)O((1-y)^2a^0)+O((1-y)^{-1}a^0) \\
  &=O((1-y)^{-1}a^0).
\end{align*}
\end{proof}

Finally, we are able to define the approximate linear operator
associated to the fundamental matrix $F$. This is the precise
definition of the approximate operator mentioned in Section
\ref{sec:strategy}, point (\ref{itm:approxlin}).

\begin{definition}
  \label{def:LF}
  Let $F$ be an admissible fundamental matrix on $J$
  and let
  $(p_1,p_2,q_1,q_2)$ be the associated coefficients. Furthermore, let
  $I\subset (-1,1)$ be open. Then, for $f\in C^2(I)$ and $(y,a)\in
  I\times J$, we set
\[
\mc L_F(a,f)(y):=f''(y)+p_1(y,a)f'(y)+p_2(y,a)\overline{f'(y)}
+q_1(y,a)f(y)+q_2(y,a)\overline{f(y)}. \]
\end{definition}

\subsection{The inhomogeneous equation}
\label{sec:inh}
Our next goal is to construct the formal inverse of the approximate linear
operator, i.e., we intend to invert the map $f\mapsto \mc
L_F(a,f)$ as outlined in Section \ref{sec:strategy}, point (\ref{itm:linfunc}). This amounts to solving the inhomogeneous equation
$\mc L_F(f,a)=g$. The solution is provided by the variation of
constants formula. We first consider the integrand which involves the
inverse of the fundamental matrix.

\begin{definition}
  \label{def:alpha}
  Let $F$ be an admissible fundamental matrix on $J$. Then, for
  $(y,a)\in (-1,1)\times J$, $k\in \{1,2,3,4\}$, and $g\in Y$, we set
  \[ \alpha_F^k(a,g)(y):=F^{-1}(y,a)^k{}_3\Re
    g(y)+F^{-1}(y,a)^k{}_4\Im g(y). \]
\end{definition}

\begin{lemma}
  \label{lem:boundalpha}
  Let $F$ be an admissible fundamental matrix on $J$. Then we have
  the bound
  \[ \left \|\alpha_F^k(a,g)\right\|_{L^\infty(-1,1)}\lesssim
    \|g\|_Y \]
  for all $(a,g)\in J\times Y$ and $k\in \{1,2,3,4\}$. 
\end{lemma}

\begin{proof}
  From Lemma \ref{lem:FR-1} we have
  \[ \left |F^{-1}(y,a)^k{}_3\right |+\left | F^{-1}(y,a)^k{}_4\right
    |\lesssim (1-y)^2 \]
  for all $(y,a)\in [0,1)\times J$.
  Furthermore, if $y\in (-1,0]$,
  \[ F^{-1}(y,a)^k{}_\ell=\sum_{m=1}^4 M_F^{-1}(a)^k{}_m F_L^{-1}(y)^m{}_\ell \]
  and
  Lemma \ref{lem:FL-1} yields
  \[ \left |F^{-1}(y,a)^k{}_3\right |+\left | F^{-1}(y,a)^k{}_4\right
    |\lesssim 1+y \]
  for all $(y,a)\in (-1,0]\times J$.
\end{proof}

Now we can write down the variation of constants formula and thereby
define the inverse of the linear operator.

\begin{definition}
  \label{def:L-1}
  Let $F$ be an admissible fundamental matrix on $J$.
Then, for $(y,a)\in (-1,1)\times J$ and $g\in Y$, we set 
  \begin{align*}
  \mc L^{-1}_F(a,g)(y):=&\sum_{k=1}^2 \left
      [F(y,a)^1{}_k+iF(y,a)^2{}_k\right]\int_{-1}^y \alpha_F^k(a,g)(x)dx \\
    &-\sum_{k=3}^4 \left
      [F(y,a)^1{}_k+iF(y,a)^2{}_k\right]\int_y^1 \alpha_F^k(a,g)(x)dx \\
                        &+[F(y,a)^1{}_1+iF(y,a)^2{}_1] 
    \sum_{k=3}^4\frac{M_F(a)^3{}_k}{M_F(a)^3{}_1}\int_{-1}^1 \alpha_F^k(a,g)(x)dx.
  \end{align*}
\end{definition}

\begin{remark}
  The first two terms in the definition of $\mc L_F^{-1}$ are
  straightforward from the variation of constants formula. However,
  the last term requires some
explanation. Without this term we would have to satisfy two real
vanishing conditions
in order to ensure that $\mc L_F^{-1}(a,g)$ is regular at $y=-1$. However, we only
have the single real parameter $a$ to vary. Thus, we need to cancel
one of the singular terms near $y=-1$ by some other means. This is
provided by the last term in the definition of $\mc L_F^{-1}$, see below. Indeed,
the condition $M_F(a)^3{}_1\not=0$ ensures that the function $F^1{}_1+iF^2{}_1$
(which is regular at $y=1$)
is singular at $y=-1$ and can therefore be
used to cancel one of the singular terms. The remaining singular term
will then give rise to a single real vanishing condition that will be
satisfied by
varying the parameter $a$. This will determine the value of $a$.

At this point the reader might wonder why it is not possible to cancel both
singular contributions by such a construction, without getting a
condition on $a$. This would require the function $F^1{}_2+iF^2{}_2$
to be singular at $y=-1$. However, there is no condition in Definition
\ref{def:F} that would ensure this. In other words, $\mc
L_F$ is allowed to have a one-dimensional kernel and this is necessary
because of the phase invariance of the problem.
\end{remark}

We now show that the operator defined in Definition \ref{def:L-1} is
really an inverse to $\mc L_F$. The proof consists of a straightforward
verification of the variation of constants formula.

    \begin{lemma}
      \label{lem:formalinverse}
  Let $F$ be an admissible fundamental matrix on $J$ and $g\in Y$. 
  Then $\mc L_F^{-1}(a,g)\in C^2((-1,1)\setminus\{0\})\cap C^1(-1,1)$ and
  \[
    \mc L_F\left (a,\mc L_F^{-1}(a,g)\right )(y)=g(y) \]
  for all $(y,a)\in (-1,1)\setminus\{0\}\times J$.
\end{lemma}

\begin{proof}
  For $j\in \{1,2,3,4\}$, we set
  \[ f^j(y,a):=\sum_{k=1}^4F(y,a)^j{}_k\int_{y_k}^y \left [F^{-1}(x,a)^k{}_3
      \Re g(x)+F^{-1}(x,a)^k{}_4\Im g(x)\right ]dx+c(a,g)F(y,a)^j{}_1, \]
  where $(y_1,y_2,y_3,y_4)=(-1,-1,1,1)$ and
  \[ c(a,g):=\sum_{k=3}^4\frac{M_F(a)^3{}_k}{M_F(a)^3{}_1}\int_{-1}^1
                          \alpha_F^k(a,g)(x)dx. \]
  Then we have
  \begin{equation}
    \label{eq:proofinvLF}
    \begin{split}
    \partial_y f^j(y,a)
    &=\sum_{k=1}^4\partial_y F(y,a)^j{}_k\int_{y_k}^y \left [F^{-1}(x,a)^k{}_3
      \Re g(x)+F^{-1}(x,a)^k{}_4\Im g(x)\right ]dx \\
      &\quad +\delta^j{}_3 \Re
        g(y)+\delta^j{}_4 \Im g(y)+c(a,g)\partial_y F(y,a)^j{}_1 \\
   &=-\sum_{k=1}^4A(y,a)^j{}_k f^k(y,a)+\delta^j{}_3 \Re
        g(y)+\delta^j{}_4 \Im g(y)
   \end{split}
  \end{equation}
  with $A(y,a)=-\partial_yF(y,a)F(y,a)^{-1}$. Note that $\partial_y
  f^j(y,a)=f^{j+2}(y,a)$ for $j\in \{1,2\}$ and thus, in vector
  notation, Eq.~\eqref{eq:proofinvLF} reads
  \[
    \partial_y \begin{pmatrix}
      f^1(y,a) \\
      f^2(y,a) \\
      \partial_y f^1(y,a) \\
      \partial_y f^2(y,a)
    \end{pmatrix}
    +A(y,a)\begin{pmatrix}
      f^1(y,a) \\
      f^2(y,a) \\
      \partial_y f^1(y,a) \\
      \partial_y f^2(y,a)
    \end{pmatrix}=
    \begin{pmatrix}
      0 \\ 0 \\ \Re g(y) \\ \Im g(y)
    \end{pmatrix}.
  \]
  Consequently, in view of the definition of the coefficients
  $(p_1,p_2,q_1,q_2)$ associated to $F$,
  \[ \mc L_F^{-1}(a,g)(y)=f^1(y,a)+if^2(y,a) \]
  satisfies the stated equation.
\end{proof}

\subsection{Regularization of $\mc L_F^{-1}$}
\label{sec:reg}
The problem with the inverse $\mc L_F^{-1}$ is that it does not map
from $Y$ to $X$ because of a singularity at the endpoint $-1$.
We now correct $\mc L_F^{-1}$ with an additional term that cancels
this singularity. This procedure is outlined in Section
\ref{sec:strategy}, point (\ref{itm:reg}).

\begin{definition}
  \label{def:psiF}
  Let $F$ be an admissible fundamental matrix on $J$. Then we define
  $\psi_F: J\times Y\to \R$
  by
  \[ \psi_F(a,g):=\sum_{k=3}^4\left
      [M_F(a)^4{}_1\frac{M_F(a)^3{}_k}{M_F(a)^3{}_1}-M_F(a)^4{}_k\right
    ]\int_{-1}^1 \alpha_F^k(a,g)(x)dx. \]
\end{definition}

\begin{lemma}
  \label{lem:boundpsiF}
  Let $F$ be an admissible fundamental matrix on $J$. 
    Then we have the bound
    \[ |\psi_F(a,g)|\lesssim \|g\|_Y \]
    for all $g\in Y$ and $a\in J$. Furthermore, for every $g\in Y$,
    the map
    \[ a\mapsto \psi_F(a,g): J\to \R \]
    is continuous.
    \end{lemma}

  \begin{proof}
 The bound follows directly from Lemma \ref{lem:boundalpha}. For the
 continuity statement it suffices to note that, given $\epsilon>0$,
 there exist $\eta,\delta>0$ such that
 \[ \int_{1-\eta}^1\left| \alpha_F^k(g,a)-\alpha_F^k(g,b)\right|<
   \frac{\epsilon}{2} \]
 for all $a,b\in J$,
 as well as
 \[  \int_{-1}^{1-\eta}\left |\alpha_F^k(a,g)-\alpha_F^k(b,g)\right |
   <\frac{\epsilon}{2}, \]
 provided that $|a-b|<\delta$.
  \end{proof}

Now we can give the definition of the operator $\mc J_F$ that is used as a
replacement for the inverse $\mc L^{-1}_F$, see Section
\ref{sec:strategy}, item (\ref{itm:reg}).
  
  \begin{definition}
    \label{def:J}
    Let $F$ 
    be an admissible fundamental matrix on $J$ and $(a,g)\in
    J\times Y$.
Furthermore, let $\chi: \R\to [0,1]$ be a smooth cut-off that
satisfies $\chi(y)=1$ for $y\leq -\frac12$, $\chi(y)=0$ for $y\geq 0$,
and $|\chi'(y)|\leq 3$ for all $y\in \R$.
    Then we set
    \[ \mc J_F(a,g)(y):=\mc L_F^{-1}(a,g)(y)-\chi(y)\left
        [F_L(y)^1{}_4+iF_L(y)^2{}_4\right]\psi_F(a,g) \]
    for $y\in (-1,1)$, where $F_L$ denotes the left part of $F$.
  \end{definition}

  \begin{remark}
    Note that $\mc J_F$ is not an inverse of $\mc L_F$, unless
    $\psi_F(a,g)=0$.
  \end{remark}

  The point is that $\mc J_F$ maps into the right space.
  
  \begin{proposition}
    \label{prop:JF}
    Let $F$ be an admissible fundamental matrix on $J$.
    Then we have the bound
    \[ \left \|\mc J_F(a,g)\right \|_X\lesssim \|g\|_Y \]
    for all $(a,g)\in J\times Y$. 
  \end{proposition}

  \begin{proof}
    From Lemma \ref{lem:boundalpha} we obtain
    \begin{align*}
      |\mc J_F(a,g)(y)|&\lesssim \sum_{k=1}^2
      \left (|F(y,a)^1{}_k|+|F(y,a)^2{}_k|\right )\int_{-1}^y \|g\|_Y
      dx \\
      &\quad +\sum_{k=3}^4 \left (|F(y,a)^1{}_k|+|F(y,a)^2{}_k|\right
      )\int_y^1 \|g\|_Ydx+\left (|F(y,a)^1{}_1|+|F(y,a)^2{}_1|\right
        )\|g\|_Y \\
      &\lesssim \|g\|_Y
    \end{align*}
    and
    \begin{align*}
    |\mc J_F(a,g)'(y)|&\lesssim \sum_{k=1}^2
      \left (|F(y,a)^3{}_k|+|F(y,a)^4{}_k|\right )\int_{-1}^y \|g\|_Y
      dx \\
      &\quad +\sum_{k=3}^4 \left (|F(y,a)^3{}_k|+|F(y,a)^4{}_k|\right
      )\int_y^1 \|g\|_Ydx+\left (|F(y,a)^3{}_1|+|F(y,a)^4{}_1|\right
        )\|g\|_Y \\
      &\lesssim \frac{\|g\|_Y}{1-y}
    \end{align*}
    for all $(y,a)\in [0,1)\times J$.

    In the case $y\in (-1,0]$, we have $F(y,a)=F_L(y)M_F(a)$, where
    $F_L$ denotes the left part of $F$.
    Consequently, if $y\in (-1,-\frac12]$,
    \begin{align*}
      \mc J_F(a,g)(y)
      &=\mc L_F^{-1}(a,g)(y)-\left
        [F_L(y)^1{}_4+iF_L(y)^2{}_4\right ]\psi_F(a,g) \\
        &=\sum_{k=1}^2 \sum_{\ell=1}^4\left
      [F_L(y)^1{}_\ell +iF_L(y)^2{}_\ell
          \right] M_F(a)^\ell{}_k
      \int_{-1}^y \alpha_F^k(a,g)
         \\
    &\quad -\sum_{k=3}^4 \sum_{\ell=1}^4\left
      [F_L(y)^1{}_\ell +iF_L(y)^2{}_\ell
      \right]M_F(a)^\ell{}_k\int_y^1 \alpha_F^k(a,g) \\
                        &\quad +\sum_{k=3}^4\sum_{\ell=1}^4\left [F_L(y)^1{}_\ell
                          +iF_L(y)^2{}_\ell\right ] M_F(a)^\ell{}_1 
    \frac{M_F(a)^3{}_k}{M_F(a)^3{}_1}\int_{-1}^1
                          \alpha_F^k(a,g) \\
      &\quad -\sum_{k=3}^4\left
      [M_F(a)^4{}_1\frac{M_F(a)^3{}_k}{M_F(a)^3{}_1}-M_F(a)^4{}_k\right
        ]
        \left
        [F_L(y)^1{}_4+iF_L(y)^2{}_4\right ]\int_{-1}^1 \alpha_F^k(a,g).
    \end{align*}
    Since $|F_L(y)^j{}_k|\lesssim (1+y)^{-1}$ for $j\in \{1,2\}$ and
    $|F_L(y)^j{}_k|\lesssim 1$ for $k\in \{1,2\}$, we obtain
    \begin{align*}
      \mc L_F^{-1}(a,g)(y)
      &=O((1+y)^{-1}a^0)\sum_{k=1}^2\int_{-1}^y \alpha_F^k(a,g)+O(y^0a^0)\sum_{k=3}^4\int_y^1
        \alpha_F^k(a,g) \\
      &\quad +\sum_{k=3}^4 \left [F_L(y)^1{}_3
        +iF_L(y)^2{}_3 \right ]M_F(a)^3{}_k\left [\int_1^y
        \alpha_F^k(a,g)+\int_{-1}^1 \alpha_F^k(a,g)\right ] \\
      &\quad +\sum_{k=3}^4 \left [F_L(y)^1{}_4+iF_L(y)^2{}_4\right
        ]M_F(a)^4{}_k \left [\int_1^y \alpha_F^k(a,g)+\int_{-1}^1 \alpha_F^k(a,g)\right]
    \end{align*}
    and $\|\alpha_F^k(a,g)\|_{L^\infty(-1,1)}\lesssim \|g\|_Y$ yields
    \[ \left |\mc J_F(a,g)(y)\right |\lesssim
      \|g\|_Y \]
    for all $(y,a)\in (-1,-\frac12]\times J$. Analogously, we
    find
    \[ \left |\mc J_F(a,g)'(y)\right |\lesssim
      \frac{\|g\|_Y}{1+y} \]
    for all $(y,a)\in (-1,-\frac12]\times J$. Finally, in the domain
    $y\in [-\frac12,0]$, the desired bound is obvious.
  \end{proof}

  Next, we turn to the continuity of $\mc J_F(a,g)$ which requires a bit
of thought because of the oscillatory terms near $y=1$. We start with
a simple lemma on a particular type of oscillatory integral that
provides the basis for the treatment.

  \begin{lemma}
    \label{lem:osc}
    Let $J\subset\R\setminus\{0\}$ be compact. Then we have the bound
    \[ \left |\int_y^1
        e^{-i\varphi(x,\alpha)}dx\right|\lesssim (1-y)^3 \]
    for all $y\in [0,1)$ and all $\alpha\in J$.
  \end{lemma}

  \begin{proof}
An integration by parts yields
    \[ \int_y^1 e^{-i\varphi(x,\alpha)}dx
      =\int_y^1 \frac{i}{\varphi'(x,\alpha)}\partial_x
      e^{-i\varphi(x,\alpha)}dx=-\frac{i}{\varphi'(y,\alpha)}e^{-i\varphi(y,\alpha)}
      -\int_y^1 e^{-i\varphi(x,\alpha)}\partial_x\frac{i}{\varphi'(x,\alpha)}dx
    \]
    and the stated bound follows.
  \end{proof}

  \begin{lemma}
    \label{lem:oscf}
    Let $J\subset \R\setminus\{0\}$ be compact
    and $f\in C([0,1])$.
    Define $\psi: [0,1)\times J\to \C$ by
    \[ \psi(y,\alpha):=\frac{e^{i\varphi(y,\alpha)}}{1-y}\int_y^1
        e^{-i\varphi(x,\alpha)}f(x)dx. \]
      Then $\psi$ extends to a continuous function on $[0,1]\times J$.
  \end{lemma}

  \begin{proof}
    Let $\epsilon>0$ be arbitrary. Then there exists a $\delta>0$ such that
    $|f(x)-f(1)|<\epsilon$ for all $x\in [1-\delta,1)$.
    In view of Lemma \ref{lem:osc},
    \begin{align*}
      \left |\frac{1}{1-y}\int_y^1 e^{-i\varphi(x,\alpha)}f(x)dx\right|
      &\leq \left |\frac{1}{1-y}\int_y^1
        e^{-i\varphi(x,\alpha)}f(1)dx\right|
        +\left |\frac{1}{1-y}\int_y^1
        e^{-i\varphi(x,\alpha)}[f(x)-f(1)]dx\right| \\
      &\lesssim (1-y)^2+\epsilon
    \end{align*}
    for all $y\in [1-\delta,1)$ and all $\alpha\in J$. As a
    consequence, $\lim_{y\to 1-}\psi(y,\alpha)=0$ and we extend $\psi$
    to $[0,1]\times J$ by setting $\psi(1,\alpha):=0$. By
    construction, $\psi(\cdot,\alpha): [0,1]\to \C$ is continuous and
    the continuity is uniform with respect to $\alpha\in
    J$. Furthermore, $\psi(y,\cdot): J\to\C$ is continuous for every
    $y\in [0,1]$. By putting all of this together, we arrive at the claim.
  \end{proof}

  \begin{lemma}
    \label{lem:contJF}
    Let $F$ be an admissible fundamental system on $J$ and let
    $(a,g)\in J\times Y$. Then $\mc J_F(a,g)\in X$ and the map
    \[ a\mapsto \mc J_F(a,g): J\to X \]
    is continuous.
  \end{lemma}

  \begin{proof}
    In view of the proof of Proposition \ref{prop:JF} it is evident
    that $(y,a)\mapsto \mc J_F(a,g)(y)$ and
    $(y,a)\mapsto (1-y^2)\mc J_F(a,g)'(y)$ extend continuously to
    $[-1,1-\delta]\times J$ for any $\delta\in (0,1]$.
    Consequently, the only terms that require an argument are the oscillatory ones
    near $y=1$. Inspection of Definitions \ref{def:alpha} and
    \ref{def:L-1} in conjunction with Lemma \ref{lem:FR-1} shows that
    the most delicate term is 
    \[ \sum_{k=3}^4 \left
      [F(y,a)^3{}_k+iF(y,a)^4{}_k\right]\int_y^1
    \alpha_F^k(a,g)(x)dx, \]
  which occurs in the expression for $\mc
    J_F(a,g)'$. However, this term is exactly of the type that is
    handled in Lemma \ref{lem:oscf} and we conclude that
    the map $(y,a)\mapsto (1-y^2)\mc J_F(a,g)'(y)$ 
  extends continuously to $[-1,1]\times J$.
  \end{proof}

  \subsection{The solution system}

Now we have all components in place in order to solve the equation
$\mc R(a_*+a, f_*(\cdot,a)+f)=0$. We collect all the terms that we
will put on the right-hand side of the fixed point problem as outlined
in Section \ref{sec:strategy}, point (\ref{itm:contr}).
  
\begin{definition}
  \label{def:GF}
    Let $(a_*,f_*)$ be an admissible approximation and 
    let $F$ be an admissible fundamental matrix on $J$ with parameter
    $a_*$. Furthermore, let
    $(p_1,p_2,q_1,q_2)$ be the associated coefficients. Then, for
    $(a,f)\in J\times X$ and $y\in (-1,1)$, we set
    \begin{align*}
      \mc
      G_F(a_*, f_*, a,f)(y):=&[p_1(y,a)-p_0(y,a_*+a)]f'(y)+p_2(y,a)\overline{f'(y)}
      \\
      &+\left
        [q_1(y,a)-q_0(y,a_*+a)-\frac{2|f_*(y,a)|^2}{(1-y)^2}\right
                     ]f(y) \\
      &+\left
        [q_2(y,a)-\frac{f_*(y,a)^2}{(1-y)^2}\right]\overline{f(y)}
        -\mc N(a_*,f_*,a,f)(y)-\mc R(a_*+a,f_*(\cdot,a))(y).
      \end{align*}
    \end{definition}

      The operator $\mc G_F$ represents the full ``right-hand side''
      of the problem, taking into account also the error we make by
      replacing the linear operator $\mc L$ by $\mc L_F$,
      cf.~Eq.~\eqref{eq:fpstrategy} and Lemma \ref{lem:solsys} below.
      The point is
      that $\mc G_F$ maps into $Y$ and by choosing ``good''
      approximations, the right-hand side can be made arbitrarily
      small and hence treated perturbatively.

      \begin{lemma}
        \label{lem:GF}
    Let $(a_*, f_*)$ be an admissible approximation and let $F$ be an
    admissible fundamental matrix on $J$ with parameter $a_*$. Then $\mc G_F(a_*, f_*, a,f)\in Y$
    for all $(a,f)\in
    J\times X$.
  \end{lemma}

  \begin{proof}
    Let $(a,f)\in J\times X$.
    We show that each of the terms that constitute $\mc
    G_F(a_*,f_*,a,f)$ belongs to $Y$.
    From Lemma \ref{lem:R} we already know that $\mc
    R(a_*+a,f_*(\cdot,a))\in Y$. For the nonlinearity
    note that $f\in X$ implies $f\in L^\infty(-1,1)$
    and thus, $\mc N(a_*, f_*, a, f)(y)=O((1-y)^{-2})$. This shows that
    $\mc N(a_*,f_*,a,f)\in Y$. Next, recall from Definition
    \ref{def:R} that
    \begin{align*}
  p_0(y,a_*+a)&=\frac{4i(a_*+a)}{(1-y)^3}-\frac{2i(a_*+a)}{(1-y)^2}-\frac{2+\frac{2i}{(a_*+a)}}{1-y}+\frac{2}{1+y} \\
      q_0(y,a_*+a)&=-\frac{2-2i(a_*+a)}{(1-y)^3}-\frac{\frac{1}{(a_*+a)^2}-\frac{i}{(a_*+a)}}{(1-y)^2}
                    -\frac{1+\frac{i}{a_*+a}}{1-y}-\frac{1+\frac{i}{a_*+a}}{1+y}
\end{align*}
and from Proposition \ref{prop:F} we have
\begin{align*}
      p_1(y,a)&=\frac{4i(a_*+a)}{(1-y)^3}-\frac{2i(a_*+a)}{(1-y)^{2}}+\frac{2}{1+y}+O((1-y)^{-1}) \\
      q_1(y,a)&=-\frac{2-2i(a_*+a)}{(1-y)^3}+O((1-y)^{-2})+O((1+y)^{-1})
      \end{align*}
   as well as
    \begin{align*}
      p_2(y,a)&=O((1-y)^{-1}) \\
      q_2(y,a)&=O((1-y)^{-2})+O((1+y)^{-1}).
    \end{align*}
    Consequently,
    \begin{align*}
      p_1(y,a)-p_0(y,a_*+a)&=O((1-y)^{-1}) \\
      q_1(y,a)-q_0(y,a_*+a)&=O((1-y)^{-2})+O((1+y)^{-1})
    \end{align*}
    and since $f\in X$ implies that $f'(y)=O((1-y^2)^{-1})$, we obtain
    \begin{align*}
      [p_1(y,a)-p_0(y,a_*+a)]f'(y)&=O((1+y)^{-1}(1-y)^{-2}) \\
      p_2(y,a)\overline{f'(y)}&=O((1+y)^{-1}(1-y)^{-2})
    \end{align*}
    as well as
    \begin{align*}
     \left
        [q_1(y,a)-q_0(y,a_*+a)-\frac{2|f_*(y,a)|^2}{(1-y)^2}\right
      ]f(y)&=O((1+y)^{-1}(1-y)^{-2}) \\
      \left
        [q_2(y,a)-\frac{f_*(y,a)^2}{(1-y)^2}\right]\overline{f(y)}&=O((1+y)^{-1}(1-y)^{-2}),
    \end{align*}
    which completes the proof.
  \end{proof}

  \begin{lemma}
    \label{lem:contGF}
  Let $(a_*, f_*)$ be an admissible approximation and let $F$ be an
  admissible fundamental matrix on $J$ with parameter $a_*$.
  Then, for every $f\in X$, the map
  \[ a\mapsto \mc G_F(a_*,f_*,a,f): J\to Y \]
  is continuous.
\end{lemma}

\begin{proof}
  This follows immediately from the explicit form of $\mc G_F$ and
  Lemma \ref{lem:R}.
\end{proof}

We state the system we actually solve and show that it is equivalent
to the original problem, cf.~Eq.~\eqref{eq:fpstrategy}.
  
  \begin{lemma}
    \label{lem:solsys}
  Let $(a_*, f_*)$ be an admissible approximation and let $F$ be an
  admissible fundamental matrix on $J$ with parameter $a_*$. 
    Suppose that $(a,f)\in
    J\times X$ satisfies
    \begin{equation}
      \label{eq:solsys}
     \begin{cases}
       f(y)=\mc J_F\left (a, \mc G_F(a_*, f_*,a,f)\right )(y) \\
       \psi_F\left (a,\mc G_F(a_*,f_*,a,f)\right )=0
     \end{cases}
   \end{equation}
   for all $y\in (-1,1)$. Then $f\in C^2(-1,1)$ and
   \[ \mc R(a_*+a, f_*(\cdot,a)+f)(y)=0 \]
   for all $y\in (-1,1)$.
 \end{lemma}

 \begin{proof}
   Since $\psi_F(a,\mc G_F(a_*, f_*,a,f))=0$, we have $\mc J_F(a,\mc
   G_F(a_*,f_*,a,f))=\mc L_F^{-1}(a,\mc G_F(a_*,f_*,a,f))$. By
   Lemma \ref{lem:GF} and Lemma \ref{lem:formalinverse}, $f\in 
   C^2((-1,1)\setminus\{0\})\cap C^1(-1,1)$ and by
   Eq.~\eqref{eq:solsys} together with Lemma \ref{lem:formalinverse},
   \begin{align*}
     \mc L_F(a,f)(y)=\mc L_F\left (a, \mc L_F^{-1}\left (a,\mc
     G_F(a_*,f_*,a,f)\right)\right)(y)=\mc G_F(a_*,f_*,a,f)(y)
   \end{align*}
   for all $y\in (-1,1)\setminus\{0\}$.
Furthermore, by
   the perturbative
   decomposition (Lemma \ref{lem:pert}), we obtain
   \begin{align*}
     \mc R(a_*+a, f_*(\cdot,a)+f)&=\mc L(a_*, f_*, a, f)+\mc N(a_*,
                                   f_*, a, f)+\mc R(a_*+a, f_*(\cdot,a)) \\
     &=\mc L_F(a,f)+\mc L(a_*, f_*, a, f)-\mc L_F(a,f)+\mc N(a_*,
                                   f_*, a, f)+\mc R(a_*+a,
       f_*(\cdot,a)) \\
     &=\mc G_F(a_*,f_*,a,f)+\mc L(a_*, f_*, a, f)-\mc L_F(a,f)+\mc N(a_*,
       f_*, a, f) \\
     &\quad +\mc R(a_*+a,
       f_*(\cdot,a)) \\
     &=0
   \end{align*}
   on $(-1,1)\setminus\{0\}$, where we have used that
   \[ \mc G_F(a_*,f_*,a,f)=\mc L_F(a,f)-\mc L(a_*,f_*,a,f)-\mc
     N(a_*,f_*,a,f)-\mc R(a_*+a,f_*(\cdot,a)), \]
   see Definitions \ref{def:GF}, \ref{def:LF}, and \ref{def:LN}. Since $f\in C^1(-1,1)$ and $\mc R(a_*+a,
   f_*(\cdot,a)+f)=0$ is a second-order ODE with coefficients and
   forcing term in
   $C(-1,1)$, it follows that $f\in C^2(-1,1)$ and $\mc R(a_*+a,
   f_*(\cdot,a)+f)=0$ on all of $(-1,1)$.
 \end{proof}

 \begin{remark}
   Lemma \ref{lem:solsys} provides us with the means to practically
   solve the equation $\mc R(a_*+a, f_*(\cdot,a)+f)=0$ since
   Eq.~\eqref{eq:solsys} only involves quantities that can be
   constructed explicitly.
 \end{remark}

 \section{Quantitative bounds}
 \label{sec:quant}
 
In order to really solve Eq.~\eqref{eq:solsys}, it will be necessary
to quantify all the bounds involved. In this section we
develop the corresponding theoretical foundation. The following estimates
together with the concrete approximations are then used to
obtain the quantitative bounds for the fixed point argument described
in Section \ref{sec:strategy}, point (\ref{itm:contr}).

We start with estimates for the functional $\psi_F$.

  \begin{definition}
    Let $F$ be an admissible fundamental matrix on $J$.
    In view of Lemma \ref{lem:boundalpha}, we define
    \[ C_{\alpha}^k(F):=\sup\left \{\left \|\alpha_F^k(a,g)\right
      \|_{L^\infty(-1,1)}: g\in Y, \|g\|_Y\leq 1, a\in J \right \} \]
  for $k\in \{1,2,3,4\}$.
  \end{definition}

  \begin{lemma}
    \label{lem:alphakF}
    Let $F$ be an admissible fundamental matrix on $J$.
    Then we have the bound
    \[ C_\alpha^k(F)\leq
      \sup_{(y,a)\in
        (-1,1)\times J}\frac{|F^{-1}(y,a)^k{}_3|+|F^{-1}(y,a)^k{}_4|}{(1+y)(1-y)^2} \]
    for $k\in\{1,2,3,4\}$.
  \end{lemma}

  \begin{proof}
    Evidently, 
  \[ \left |\alpha_F^k(a,g)(y)\right |\leq \left (\left
        |F^{-1}(y,a)^k{}_3\right|+\left |F^{-1}(y,a)^k{}_4\right
      |\right )|g(y)| \]
  and the stated bound follows.
  \end{proof}

  \begin{definition}
    Let $F$ be an admissible fundamental matrix on $J$. Then we define
    \[ C_{\psi}(F):=\sup\left \{|\psi_F(a,g)|: g\in Y, \|g\|_Y\leq
        1, a\in J\right \}. \]
  \end{definition}

  \begin{lemma}
    \label{lem:quantCpsi}
    Let $F$ be an admissible fundamental matrix on $J$. Then we have the bound
    \[ C_{\psi}(F)\leq 2 \max_{a\in J}\sum_{k=3}^4 \left |
        M_F(a)^4{}_1\frac{M_F(a)^3{}_k}{M_F(a)^3{}_1}-M_F(a)^4{}_k\right
      | C_{\alpha}^k(F). \]
  \end{lemma}

  \begin{proof}
    By definition, we have
    \begin{align*}
      |\psi_F(a,g)|&\leq \sum_{k=3}^4 \left
      |M_F(a)^4{}_1\frac{M_F(a)^3{}_k}{M_F(a)^3{}_1}-M_F(a)^4{}_k\right|\int_{-1}^1
      \left |\alpha_F^k(a,g)(x)\right|dx
    \end{align*}
    and the claim follows.
  \end{proof}

  Next, we turn to the operator $\mc J_F$. We define the auxiliary
  quantities $\beta_{F,L}^n$, $\beta_{F,M}^n$, and $\beta_{F,L}^n$
  that will be used to formulate the estimate for $\mc J_F$.

  \begin{definition}
    \label{def:beta}
    Let $F$ be an admissible fundamental matrix on $J$ with left part
    $F_L$. For $n\in \{0,1\}$, we define $\beta_{F,L}^n:
    (-1,-\frac12)\times J\to \R$, $\beta_{F,M}^n: [-\frac12,0]\times
    J\to \R$, and $\beta_{R,M}^n: (0,1)\times J\to \R$ by
    \begin{align*}
      \beta_{F,L}^n(y,a)
      &:=
        (1+y)(1-y^2)^n\sum_{k=1}^2 \sum_{m=1}^2\left |\sum_{\ell=1}^4 F_L(y)^{m+2n}{}_\ell
        M_F(a)^\ell{}_k\right |C_\alpha^k(F) \\
      &\quad +(1+y)(1-y^2)^n \\
      &\qquad \times \sum_{k=3}^4\sum_{m=1}^2 \left |\sum_{\ell=1}^3
        F_L(y)^{m+2n}{}_\ell
        M_F(a)^\ell{}_1\frac{M_F(a)^3{}_k}{M_F(a)^3{}_1}+
        F_L(y)^{m+2n}{}_4 M_F(a)^4{}_k\right |C_\alpha^k(F)
     \\
      &\quad +(1-y)(1-y^2)^n \\
      &\qquad\times \sum_{k=3}^4\sum_{m=1}^2 \left |\sum_{\ell=1}^2F_L(y)^{m+2n}{}_\ell
        \left
        (M_F(a)^\ell{}_1\frac{M_F(a)^3{}_k}{M_F(a)^3{}_1}-M_F(a)^\ell{}_k\right
        )\right | C_\alpha^k(F),
    \end{align*}
    \begin{align*}
      \beta_{F,M}^n(y,a)
      &:=(1+y)(1-y^2)^n\sum_{k=1}^2\sum_{m=1}^2\left |\sum_{\ell=1}^4F_L(y)^{m+2n}{}_\ell
        M_F(a)^\ell{}_k\right |C_\alpha^k(F) \\
      &\quad +(1+y)(1-y^2)^n\sum_{k=3}^4\sum_{m=1}^2\left
        |\sum_{\ell=1}^4F_L(y)^{m+2n}{}_\ell
        M_F(a)^\ell{}_1\frac{M_F(a)^3{}_k}{M_F(a)^3{}_1}\right
        |C_\alpha^k(F) \\
      &\quad +(1-y)(1-y^2)^n \\
      &\qquad\times\sum_{k=3}^4\sum_{m=1}^2\left
        |\sum_{\ell\in\{1,2,4\}}F_L(y)^{m+2n}{}_\ell\left
        (M_F(a)^\ell{}_1\frac{M_F(a)^3{}_k}{M_F(a)^3{}_1}-M_F(a)^\ell{}_k\right)\right
        |C_\alpha^k(F) \\
      &\quad +(1-y^2)^n\sum_{m=1}^2 \left |F_L(y)^{m+2n}{}_4\right
        |C_\psi(F)+3n(1-y^2)^n\sum_{m=1}^2
       \left | F_L(y)^m{}_4 \right |C_\psi(F),
    \end{align*}
    and
     \begin{align*}
      \beta_{F,R}^n(y,a)
      &:=(1+y)(1-y^2)^n\sum_{k=1}^2 \sum_{m=1}^2\left
        |F(y,a)^{m+2n}{}_k\right |C_\alpha^k(F) \\
      &\quad +(1-y)(1-y^2)^n\sum_{k=3}^4 \sum_{m=1}^2\left
        |\frac{M_F(a)^3{}_k}{M_F(a)^3{}_1}F(y,a)^{m+2n}{}_1-F(y,a)^{m+2n}{}_k\right
        |C_\alpha^k(F) \\
      &\quad +(1+y)(1-y^2)^{n}\sum_{k=3}^4\sum_{m=1}^2\left |F(y,a)^{m+2n}{}_1
        \frac{M_F(a)^3{}_k}{M_F(a)^3{}_1}\right |C_\alpha^k(F).
     \end{align*}
  \end{definition}

\begin{definition}
   Let $F$ be an admissible fundamental matrix on $J$.
    We define
    \[ C_{\mc J}(F):=\sup\left \{\|\mc J_F(a,g)\|_X: g\in Y,
        \|g\|_Y\leq 1, a\in J\right \}. \]
  \end{definition}

  \begin{proposition}
    \label{prop:quantCJ}
    Let $F$ be an admissible fundamental matrix on $J$.
   Then we have the bound
   \[ C_{\mc J}(F)\leq \max\left\{\sum_{n=0}^1\sup_{
         (-1,-\frac12)\times J} \beta_{F, L}^n, \sum_{n=0}^1
       \sup_{(-\frac12, 0)\times J} \beta_{F,M}^n,
       \sum_{n=0}^1\sup_{(0,1)\times J}
       \beta_{F, R}^n\right \}. \]
  \end{proposition}

  \begin{proof}
    Let $g\in Y$ with $\|g\|_Y\leq 1$ and $a\in J$.
    By definition, we have
    \begin{align*}
      \mc L_F^{-1}(a,g)(y)
      &=\sum_{k=1}^2 \left [ F(y,a)^1{}_k
        +iF(y,a)^2{}_k\right ]\int_{-1}^y \alpha_F^k(a,g) \\
      &\quad -\sum_{k=3}^4 \left [ F(y,a)^1{}_k
        +iF(y,a)^2{}_k\right ]\int_y^1 \alpha_F^k(a,g)\\
      &\quad +\sum_{k=3}^4 \left
        [F(y,a)^1{}_1+iF(y,a)^2{}_1\right]\frac{M_F(a)^3{}_k}{M_F(a)^3{}_1}\left[
        \int_{-1}^y \alpha_F^k(a,g)+\int_y^1\alpha_F^k(a,g)\right ]
    \end{align*}
    and
    \[ \mc J_F(a,g)(y)=\mc L_F^{-1}(a,g)(y)-\chi(y)\left
        [F_L(y)^1{}_4+iF_L(y)^2{}_4\right]\psi_F(a,g), \]
    with
    \[ \psi_F(a,g)=\sum_{k=3}^4\left
      [M_F(a)^4{}_1\frac{M_F(a)^3{}_k}{M_F(a)^3{}_1}-M_F(a)^4{}_k\right
    ]\int_{-1}^1 \alpha_F^k(a,g)(x)dx. \]
    Consequently, $|\mc J_F(a,g)(y)|\leq \beta_{F,R}^0(y,a)$ for all
    $(y,a)\in (0,1)\times J$. For the derivative, we note that
    \begin{align*}
      \mc L_F^{-1}(a,g)'(y)&=\sum_{k=1}^2 \left [ F(y,a)^3{}_k
        +iF(y,a)^4{}_k\right ]\int_{-1}^y \alpha_F^k(a,g) \\
      &\quad -\sum_{k=3}^4 \left [ F(y,a)^3{}_k
        +iF(y,a)^4{}_k\right ]\int_y^1 \alpha_F^k(a,g)\\
      &\quad +\sum_{k=3}^4 \left
        [F(y,a)^3{}_1+iF(y,a)^4{}_1\right]\frac{M_F(a)^3{}_k}{M_F(a)^3{}_1}\left[
        \int_{-1}^y \alpha_F^k(a,g)+\int_y^1\alpha_F^k(a,g)\right ]
    \end{align*}
    since $\partial_y F(y,a)^j{}_k=F(y,a)^{j+2}{}_k$ for all $(y,a)\in
    (-1,1)\times J$, $j\in \{1,2\}$, $k\in \{1,2,3,4\}$, and the evaluation
    terms vanish because
    \begin{align*}
      \sum_{k=1}^4 &\left [ F(y,a)^1{}_k
                     +iF(y,a)^2{}_k\right ]\alpha_F^k(a,g)(y) \\
      &=\sum_{k=1}^4 \left [ F(y,a)^1{}_k
                     +iF(y,a)^2{}_k\right ]\left [F^{-1}(y,a)^k{}_3\Re
        g(y)+F^{-1}(y,a)^k{}_4\Im g(y)\right] \\
      &=0.
      \end{align*}
    Accordingly, we obtain $(1-y^2)|\mc
    J_F(a,g)'(y)|\leq \beta_{F,R}^1(y,a)$ for all $(y,a)\in (0,1)\times J$. 

    In the case $y\in (-1,0]$, we have
    $F(y,a)^j{}_k=\sum_{\ell=1}^4F_L(y)^j{}_\ell M_F(a)^\ell{}_k$ and by observing that
    \begin{align*}
      &\sum_{k=3}^4\sum_{\ell=1}^4\left [F_L(y)^1{}_\ell
                          +iF_L(y)^2{}_\ell\right ] M_F(a)^\ell{}_1 
    \frac{M_F(a)^3{}_k}{M_F(a)^3{}_1}\int_y^1
      \alpha_F^k(a,g) \\
     &=\sum_{k=3}^4\sum_{\ell\in\{1,2,4\}} \left [F_L(y)^1{}_\ell
                          +iF_L(y)^2{}_\ell\right ] M_F(a)^\ell{}_1 
    \frac{M_F(a)^3{}_k}{M_F(a)^3{}_1}\int_y^1
       \alpha_F^k(a,g) \\
      &\quad +\sum_{k=3}^4\left [F_L(y)^1{}_3
                          +iF_L(y)^2{}_3\right ] 
    M_F(a)^3{}_k\int_y^1
        \alpha_F^k(a,g),
    \end{align*}
    we find
    \begin{align*}
      \mc L_F^{-1}(a,g)(y)
      &=\sum_{k=1}^2 \sum_{\ell=1}^4 \left
      [F_L(y)^1{}_\ell +iF_L(y)^2{}_\ell
          \right] M_F(a)^\ell{}_k
      \int_{-1}^y \alpha_F^k(a,g)
      \\
      &\quad +\sum_{k=3}^4\sum_{\ell=1}^4\left [F_L(y)^1{}_\ell
                          +iF_L(y)^2{}_\ell\right ] M_F(a)^\ell{}_1 
    \frac{M_F(a)^3{}_k}{M_F(a)^3{}_1}\int_{-1}^y
        \alpha_F^k(a,g) \\
      &\quad -\sum_{k=3}^4 \sum_{\ell\in\{1,2,4\}}\left
      [F_L(y)^1{}_\ell +iF_L(y)^2{}_\ell
        \right]M_F(a)^\ell{}_k\int_y^1 \alpha_F^k(a,g) \\
      &\quad +\sum_{k=3}^4\sum_{\ell\in\{1,2,4\}} \left [F_L(y)^1{}_\ell
                          +iF_L(y)^2{}_\ell\right ] M_F(a)^\ell{}_1 
    \frac{M_F(a)^3{}_k}{M_F(a)^3{}_1}\int_y^1
       \alpha_F^k(a,g).
    \end{align*}
    By recalling that $|\chi'|\leq 3$, we therefore obtain $|\mc
    J_F(a,g)(y)|\leq \beta^0_{F,M}(y,a)$ and $|(1-y^2)\mc
    J_F(a,g)'(y)|\leq \beta^1_{F,M}(y,a)$ for all $(y,a)\in
    [-\frac12,0]\times J$.

    In the case $y\in (-1,-\frac12]$, we further compute
    \begin{align*}
      \mc J_F(a,g)(y)
      &=\mc L_F^{-1}(a,g)(y)-\sum_{k=3}^4\left
        [F_L(y)^1{}_4+iF_L(y)^2{}_4\right ]\left [
      M_F(a)^4{}_1\frac{M_F(a)^3{}_k}{M_F(a)^3{}_1}-M_F(a)^4{}_k\right]
      \\
      &\quad \times \left [\int_{-1}^y \alpha_F^k(a,g)+\int_y^1 \alpha_F^k(a,g)\right]
      \\
      &=\sum_{k=1}^2 \sum_{\ell=1}^4\left
      [F_L(y)^1{}_\ell +iF_L(y)^2{}_\ell
          \right] M_F(a)^\ell{}_k
      \int_{-1}^y \alpha_F^k(a,g)
      \\
      &\quad +\sum_{k=3}^4\sum_{\ell=1}^4\left [F_L(y)^1{}_\ell
                          +iF_L(y)^2{}_\ell\right ] M_F(a)^\ell{}_1 
    \frac{M_F(a)^3{}_k}{M_F(a)^3{}_1}\int_{-1}^y
        \alpha_F^k(a,g) \\
&\quad      -\sum_{k=3}^4\left
        [F_L(y)^1{}_4+iF_L(y)^2{}_4\right ]\left [
      M_F(a)^4{}_1\frac{M_F(a)^3{}_k}{M_F(a)^3{}_1}-M_F(a)^4{}_k\right]\int_{-1}^y
                             \alpha_F^k(a,g) \\
      &\quad -\sum_{k=3}^4 \sum_{\ell=1}^2\left
      [F_L(y)^1{}_\ell +iF_L(y)^2{}_\ell
        \right]M_F(a)^\ell{}_k\int_y^1 \alpha_F^k(a,g) \\
      &\quad +\sum_{k=3}^4\sum_{\ell=1}^2 \left [F_L(y)^1{}_\ell
                          +iF_L(y)^2{}_\ell\right ] M_F(a)^\ell{}_1 
    \frac{M_F(a)^3{}_k}{M_F(a)^3{}_1}\int_y^1
       \alpha_F^k(a,g),
    \end{align*}
    which simplifies to
    \begin{align*}
     \mc J_F(a,g)(y)
      &=\sum_{k=1}^2 \sum_{\ell=1}^4\left
      [F_L(y)^1{}_\ell +iF_L(y)^2{}_\ell
          \right] M_F(a)^\ell{}_k
      \int_{-1}^y \alpha_F^k(a,g)
      \\
      &\quad +\sum_{k=3}^4\sum_{\ell=1}^3\left [F_L(y)^1{}_\ell
                          +iF_L(y)^2{}_\ell\right ] M_F(a)^\ell{}_1 
    \frac{M_F(a)^3{}_k}{M_F(a)^3{}_1}\int_{-1}^y
        \alpha_F^k(a,g) \\
&\quad      +\sum_{k=3}^4\left
        [F_L(y)^1{}_4+iF_L(y)^2{}_4\right ]M_F(a)^4{}_k\int_{-1}^y
                             \alpha_F^k(a,g) \\
      &\quad -\sum_{k=3}^4 \sum_{\ell=1}^2\left
      [F_L(y)^1{}_\ell +iF_L(y)^2{}_\ell
        \right]M_F(a)^\ell{}_k\int_y^1 \alpha_F^k(a,g) \\
      &\quad +\sum_{k=3}^4\sum_{\ell=1}^2 \left [F_L(y)^1{}_\ell
                          +iF_L(y)^2{}_\ell\right ] M_F(a)^\ell{}_1 
    \frac{M_F(a)^3{}_k}{M_F(a)^3{}_1}\int_y^1
       \alpha_F^k(a,g)
    \end{align*}
    and this yields the bounds $|\mc J_F(a,g)(y)|\leq
    \beta^0_{F,L}(y,a)$ and $|(1-y^2)\mc J_F(a,g)'(y)|\leq
    \beta^1_{F,L}(y,a)$ for all $(y,a)\in (-1,-\frac12]\times J$.
\end{proof}

\section{Tools for polynomials and rational functions}
\label{sec:poly}

\noindent In order to apply the theory developed so far to the
concrete problem at hand, we need
good approximations to the desired solution and the
fundamental matrix of the linear operator. These will be provided by
explicit polynomials with rational coefficients. For estimating these
polynomials, we employ some elementary methods that we describe now. 

  \subsection{The $T$-norm} One of the most frequent problems we will face is to
  rigorously estimate the $L^\infty$-norm of explicitly given but complicated
  polynomials of high degree. The key tool for this is an auxiliary norm
  that is based on the Chebyshev expansion.

     \begin{definition}
       Let $p:\C^n\to\C$ be a polynomial. For $j\in [1,n]\cap\Z$, we denote by $\deg_j p$ the
       degree of $p$ in the $j$-th variable and set $\deg p:=(\deg_1
       p,\deg_2 p,\dots,\deg_n p)$. Furthermore, we say that $\deg
       p\leq (d_1,d_2,\dots,d_n)\in \N_0^n$ if $\deg_j p\leq d_j$ for
       every $j\in [1,n]\cap\Z$ and analogously for $\leq$ replaced
       by $\geq$.
     \end{definition}

   \begin{definition}[$T$-norm on the standard cube]
     Let $p: \C^n \to \C$ be a polynomial with $\deg p\geq 0$ and set
     $d_j:=\deg_jp$ for $j\in [1,n]\cap\Z$. Furthermore, let
     $c_{j_1 j_2\dots j_n}\in \C$ be the
     unique coefficients that
     satisfy
     \[
       p(x^1,\dots,x^n)=\sum_{j_1=0}^{d_1}\cdots
       \sum_{j_n=0}^{d_n}c_{j_1\dots j_n}
       T_{j_1}(x^1)\dots T_{j_n}(x^n) \]
     for all $(x^1,\dots,x^n)\in \C^n$.
     Then we define the \emph{T-norm of $p$ on $[-1,1]^n$} by
     \[ \|p\|_{T([-1,1]^n)}:=\sum_{j_1=0}^{d_1}\cdots
       \sum_{j_n=0}^{d_n}\left |c_{j_1\dots j_n}\right | \]
     and $\|0\|_{T([-1,1]^n)}:=0$.
   \end{definition}

The $T$-norm on an arbitrary rectangular domain is now simply defined by
composition with suitable affine-linear functions.

 \begin{definition}
      For $a,b\in\R$ with $a<b$, we define $\phi_{a,b}: [-1,1]\to [a,b]$ by
  \[ \phi_{a,b}(x):=\frac{b-a}{2}x+\frac{a+b}{2}. \]
\end{definition}

   \begin{definition}[$T$-norm on a rectangular domain]
     Let $n\in\N$ and for $j\in [1,n]\cap\Z$, let $a_j,b_j\in\R$ with
     $a_j<b_j$. Let $p: \C^n\to\C$ be a polynomial.
     Then we define the \emph{$T$-norm of $p$ on
       $[a_1,b_1]\times\dots\times[a_n,b_n]$} by
     \[ \|p\|_{T([a_1,b_1]\times\dots\times [a_n,b_n])}:=\left \|(x^1,\dots,x^n)\mapsto
         p(\phi_{a_1,b_1}(x^1),\dots,\phi_{a_n,b_n}(x^n))\right\|_{T([-1,1]^n)}. \]
   \end{definition}

One crucial feature is that the $T$-norm controls the $L^\infty$-norm.
   
   \begin{lemma}
\label{lem:LinfT}
     Let $p: \C^n\to \C$ be a polynomial and for $j\in [1,n]\cap\Z$,
     let $a_j,b_j\in\R$ with
     $a_j<b_j$. Then we have
     \[ \|p\|_{L^\infty([a_1,b_1]\times\dots\times[a_n,b_n])}\leq
       \|p\|_{T([a_1,b_1]\times\dots\times [a_n,b_n])}. \]
   \end{lemma}

   \begin{proof}
     It suffices to prove the statement for $n=1$ and we write $a=a_1$
     and $b=b_1$. The case $p=0$ is trivial and thus, we assume $\deg p\geq
     0$.
     Let $(c_j)_{j=0}^{\deg p}\subset \C$ be
     the unique coefficients that satisfy
     \[ (p\circ\phi_{a,b})(x)=\sum_{j=0}^{\deg p} c_jT_j(x). \]
     Then we have
     \[ \|p\|_{L^\infty([a,b])}=\|p\circ
       \phi_{a,b}\|_{L^\infty([-1,1])}\leq \sum_{j=0}^{\deg p}
       |c_j|=\|p\circ\phi_{a,b}\|_{T([-1,1])}=\|p\|_{T([a,b])}, \]
     since $\|T_n\|_{L^\infty([-1,1])}\leq 1$ for all $n\in\N_0$,
   \end{proof}

   \begin{remark}
     One of the key observations in the following is that the $T$-norm not only controls the
     $L^\infty$-norm but is in fact reasonably close to it. As a
     consequence, in many cases one can obtain a sufficiently precise
     estimate on the $L^\infty$-norm of a polynomial by simply
     computing its $T$-norm.
   \end{remark}

   \subsection{Computation of the $T$-norm}

Next, we derive an algorithm to explicitly compute the $T$-norm of a
given polynomial. The first ingredient is the representation of
monomials in terms of Chebyshev polynomials.
   
    \begin{lemma}
   \label{lem:monCheb}
   Let $n\in \N_0$ and $x\in \C$. Then we have
   \begin{align*}
     x^n=\frac{1}{2^{n-1}}\sum_{k=0}^{\frac{n-1}{2}}\binom{n}{\frac{n-1}{2}-k}T_{2k+1}(x)
   \end{align*}
   if $n$ is odd and
   \begin{align*}
     x^n=\frac{1}{2^n}\binom{n}{\frac{n}{2}}+\frac{1}{2^{n-1}}\sum_{k=1}^{\frac{n}{2}}\binom{n}{\frac{n}{2}-k}T_{2k}(x)
   \end{align*}
   if $n$ is even.
 \end{lemma}

 \begin{proof}
   These formulas are of course well-known and can easily be derived by
   combining the defining property $T_n(\cos(z))=\cos(nz)$, $z\in\C$, with the binomial
   theorem applied to
   \[ \cos(z)^n=\frac{1}{2^n}\left (e^{iz}+e^{-iz}\right)^n. \]
 \end{proof}

 Lemma \ref{lem:monCheb} motivates the following definition.

 \begin{definition}
   Let $j,k\in\N_0$. If $j>k$, we set $C^j{}_k=0$. If $j\leq k$, we define
   \begin{align*}
     C^j{}_k&:=\frac{1}{2^{k-1}}\binom{k}{\frac{k-1}{2}-\frac{j-1}{2}}
     &
     &\mbox{if $k$ odd and $j$ odd}
     \\
     C^0{}_k&:=\frac{1}{2^k}\binom{k}{\frac{k}{2}}
     &
     &\mbox{if $k$ even} \\
     C^j{}_k&:=\frac{1}{2^{k-1}}\binom{k}{\frac{k}{2}-\frac{j}{2}}
     &
     &\mbox{if $k$ even, $j$ even, and $j\geq 2$} \\
     C^j{}_k&:=0
     &
     &\mbox{otherwise}.
         \end{align*}
 \end{definition}

 According to Lemma \ref{lem:monCheb}, we have
 \[ x^n=\sum_{k=0}^n C^k{}_n T_k(x) \]
and this allows one to go from the monomial representation to
the Chebyshev representation. In practice, however, a polynomial $p$ is
typically given in \emph{Lagrange representation}, i.e., by evaluation
at particular points. In order to capture this, one uses the
\emph{Lagrange basis}.

\begin{definition}
  Let $d\in \N$ and $\xi=(\xi^0,\xi^1,\dots,\xi^d)\in
  \C^{d+1}$ with $\xi^j\not=\xi^k$ if $j\not= k$. Then, for $j\in [0,d]\cap \Z$, we define $\ell^\xi_j:
  \C\to\C$ by
  \[ \ell_j^\xi(x):=\prod_{k\in [0,d]\cap\Z\setminus\{j\}}\frac{x-\xi^k}{\xi^j-\xi^k}. \]
\end{definition}

Note that $\ell_j^\xi(\xi^k)=\delta_j{}^k$ and thus, the corresponding
Lagrange representation of a polynomial $p$ of degree $d$ reads
\[ p(x)=\sum_{j=0}^d p(\xi^j)\ell^\xi_j(x). \]
In order to obtain the monomial representation, we need to
express the polynomials $\ell_j^\xi$ in terms of monomials.
This is a classical problem in polynomial approximation that amounts to
inverting the Vandermonde matrix. 
To keep the technical prerequisites at a bare minimum, we
prefer
to state an explicit recursion formula that can be
implemented straight away without the need for solving a linear system.

\begin{definition}
  Let $n\in \N$ and $\xi=(\xi^1,\xi^2,\dots,\xi^n)\in \C^n$. For $j\in
  [0,n]\cap \Z$ and $k\in [0,j]\cap \Z$, we
  define $a_{j,k}(\xi)\in \C$ recursively by
  \begin{align*}
    a_{j,0}(\xi)&:=\prod_{\ell=1}^j \xi^\ell \\
    a_{j,k}(\xi)&:=a_{j-1,k-1}(\xi)+\xi^ja_{j-1,k}(\xi), \qquad k\in
                                                            [1,j-1]\cap
                                                            \Z \\
    a_{j,j}(\xi)&:=1.
\end{align*}
\end{definition}

\begin{lemma}
  \label{lem:expandprod}
  Let $n\in \N$, $\xi=(\xi^1,\xi^2,\dots,\xi^n)\in \C^n$, and
  $x\in \C$. Then
  we have
  \[ \prod_{j=1}^n (x+\xi^j)=\sum_{k=0}^n a_{n,k}(\xi)x^k. \]
\end{lemma}

\begin{proof}
  This follows immediately by a simple induction.
\end{proof}

\begin{definition}
  Let $d\in \N$ and $\xi=(\xi^0,\xi^1,\dots,\xi^d)\in\C^{d+1}$.
  For $j\in [0,d]\cap\Z$, we define $\xi_{(j)}\in \C^d$ by
  \[ \xi_{(j)}:=(\xi^1,\dots, \xi^{j-1},\xi^{j+1},\dots,\xi^d), \]
  i.e., $\xi_{(j)}$ is obtained from $\xi$ by removing the $j$-th element.
\end{definition}

\begin{definition}
  Let $d\in \N$ and $\xi=(\xi^0,\xi^1,\dots,\xi^d)\in
  \C^{d+1}$ with $\xi^j\not= \xi^k$ if $j\not= k$. For $j,k\in [0,d]\cap\Z$, we define $L(\xi)^j{}_k\in\C$
  by
  \begin{align*}
    L(\xi)^j{}_k:=\frac{a_{d,j}(-\xi_{(k)})}{\prod_{\ell\in [0,d]\cap\Z\setminus\{k\}}(\xi^k-\xi^\ell).}
  \end{align*}
\end{definition}

\begin{lemma}
  \label{lem:Legmon}
  Let $d\in \N$,
  $\xi=(\xi_0,\xi_1,\dots,\xi_d)\in\C^{d+1}$ with $\xi^j\not=\xi^k$ if
  $j\not= k$, $x\in \C$, and $j\in [0,d]\cap\Z$. Then we have
  \[ \ell_j^\xi(x)=\sum_{k=0}^d L(\xi)^k{}_j x^k. \]
\end{lemma}

\begin{proof}
  This follows immediately by inserting the definitions and using
  Lemma \ref{lem:expandprod}.
\end{proof}

\begin{definition}
  Let $d\in \N$ and $\xi=(\xi^0,\xi^1,\dots,\xi^d)\in \C^{d+1}$ with
  $\xi^j\not= \xi^k$ if $j\not= k$. Then,
  for $j,k\in [0,d]\cap\Z$,
  we define $T(\xi)^j{}_k\in\C$ by
  \[ T(\xi)^j{}_k:=\sum_{\ell=0}^d C^j{}_\ell L(\xi)^\ell{}_k. \]
\end{definition}

With these preparations we can now calculate the $T$-norm directly
from the Lagrange representation.

\begin{lemma}
  Let $d\in \N$, $\xi\in (\xi^0,\xi^1,\dots,\xi^d)\in\C^{d+1}$ with
  $\xi^j\not= \xi^k$ if $j\not= k$. Let $p: \C\to\C$ be a polynomial
  of degree $d$. 
  Then we have
  \[ \|p\|_{T([-1,1])}=\sum_{j=0}^d\left |\sum_{k=0}^d
        T(\xi)^j{}_kp(\xi^k)\right|
\]
\end{lemma}

\begin{proof}
  By Lemmas \ref{lem:Legmon} and \ref{lem:monCheb}, we have
  \begin{align*}
    \sum_{j=0}^d p(\xi^j)\ell_j^\xi(x)
    &=\sum_{j=0}^d
      p(\xi^j)\sum_{k=0}^d L(\xi)^k{}_j x^k
      =\sum_{j=0}^d
      p(\xi^j)\sum_{k=0}^dL(\xi)^k{}_j\sum_{\ell=0}^dC^\ell{}_kT_\ell(x)
    \\
    &=\sum_{j=0}^d \sum_{k=0}^d T(\xi)^j{}_kp(\xi^k)T_j(x)
  \end{align*}
  and the claim follows.
\end{proof}

\begin{remark}
One crucial point is that if the polynomial $p$ has rational
coefficients and $a,b\in\Q$ then $\|p\|_{T([a,b])}$ can be computed by performing
a finite number of elementary operations in $\Q$. Consequently,
$\|p\|_{T([a,b])}$ can be
computed \emph{exactly}.
\end{remark}

The generalization to polynomials of more than one variable is
immediate and at most notationally challenging.

\begin{lemma}
  Let $n\in\N$ and for $j\in [1,n]\cap\Z$,
  let $d_j\in \N$ and
  $\xi_j=(\xi_j^0,\xi_j^1,\dots,\xi_j^{d_j})\in\C^{d_j+1}$ with
  $\xi_j^k\not= \xi_j^\ell$ if $k\not= \ell$.
  Suppose that $p: \C^n\to\C$ is a polynomial of degree $d_j$ in the
  $j$-th variable.
  Then we have
  \begin{align*}
    \|p\|_{T([-1,1]^n)}=&\sum_{j_1=0}^{d_1}\cdots\sum_{j_n=0}^{d_n}
    \left |
        \sum_{k_1=0}^{d_1}\cdots\sum_{k_n=0}^{d_n}T(\xi_1)^{j_1}{}_{k_1}
        \dots
                          T(\xi_n)^{j_n}{}_{k_n}p(\xi_1^{k_1},\dots,\xi_n^{k_n})\right|.
  \end{align*}
\end{lemma}

\subsection{Derivatives of polynomials}

With the tools developed so far it is also straightforward to compute
derivatives of polynomials.

\begin{lemma}
  Let $d\in\N_0$ and
  $\xi=(\xi^0,\xi^1,\dots,\xi^d)\in\C^{d+1}$ with $\xi^j\not=\xi^k$ if
  $j\not= k$.
  Furthermore, let $p: \C\to\C$ be a polynomial of degree $d$
  and $m\in [0,d]\cap\Z$. Then we have
  \[ p^{(m)}(x)=
  \sum_{j=0}^{d-m}\frac{(j+m)!}{j!}
  \sum_{k=0}^d
  L(\xi)^{j+m}{}_k
  p(\xi^k)x^j. \]
\end{lemma}

\begin{proof}
  This is straightforward from the representation
  \[ p(x)=\sum_{j=0}^d \sum_{k=0}^dL(\xi)^j{}_k p(\xi^k)x^j. \]
\end{proof}

\subsection{The approximate maximum}
Finally, we provide a simple method to rigorously estimate the minima and
maxima of rational functions by evaluating them on a discrete grid.

  \begin{lemma}
    \label{lem:eval}
    Let $n\in\N$ and for $j\in \{1,2\dots,n\}$, let
    $M_j\in\N$. Furthermore, let $a_j,b_j\in\R$ with $a_j<b_j$ and set
    \[ \Lambda:=\left \{\left (a_1+\frac{b_1-a_1}{M_1}k_1,
             a_2+\frac{b_2-a_2}{M_2}k_2,\dots,
             a_n+\frac{b_n-a_n}{M_n}k_n\right): k_j\in [0,M_j]\cap \Z,
           j\in [1,n]\cap\Z\right\} \]
       as well as
       $\Omega:=[a_1,b_1]\times[a_2,b_2]\times\dots\times[a_n,b_n]$.
       Let $f: \Omega\to\R$ be continuously
       differentiable. Then we have the bound
       \begin{align*}
         \max_\Omega f\leq \max_\Lambda f+\sum_{j=1}^n
         \frac{b_j-a_j}{M_j}\|\partial_jf\|_{L^\infty(\Omega)}.
         \end{align*}
     \end{lemma}

     \begin{proof}
The statement is a consequence of the mean value theorem.
     \end{proof}

   \begin{definition}
     \label{def:asup}
     Let $n\in\N$. For $j\in \{1,2,\dots,n\}$ let $a_j,b_j\in\R$ with
     $a_j< b_j$ and set $\Omega:=[a_1,b_1]\times
     [a_2,b_2]\times\dots\times [a_n,b_n]$. Furthermore, let $p,q:
     \R^n\to\R$ be polynomials with $\min_\Omega |q|>0$, $\epsilon>0$, and set
     \[ M_j:=\min\left \{m\in\N: m\geq
         \frac{n}{\epsilon}(b_j-a_j)\|\partial_j pq-p\partial_jq\|_{T(\Omega)}\right\} \]
       for $j\in \{1,2,\dots,n\}$, as well as
       \[ \Lambda:=\left \{\left (a_1+\frac{b_1-a_1}{M_1}k_1,
             a_2+\frac{b_2-a_2}{M_2}k_2,\dots,
             a_n+\frac{b_n-a_n}{M_n}k_n\right): k_j\in [0,M_j]\cap \Z,
           j\in [0,n]\cap\Z\right\}. \]
       Then we define
       \[ \max_{\Omega,\epsilon}\frac{p}{q}:=\max_\Lambda \frac{p}{q},\qquad
         \min_{\Omega,\epsilon}\frac{p}{q}:=\min_{\Lambda}\frac{p}{q},\qquad
         \left \|\frac{p}{q}\right
         \|_{L^\infty(\Omega,\epsilon)}:=\max_\Lambda \left |\frac{p}{q}\right|. \]
     \end{definition}

     \begin{lemma}
       \label{lem:amax}
       Let $p,q: \R^n\to\R$ be polynomials, $\epsilon>0$, and
       $a_j,b_j\in\R$ with $a_j<b_j$ for $j\in
       [1,n]\cap\Z$. Furthermore, set $\Omega:=[a_1,b_1]\times
       [a_2,b_2]\times\dots\times[a_n,b_n]$ and assume that $\min_\Omega|q|>0$. Then we have
       \[ \left |\min_{\Omega,\epsilon}\frac{p}{q}-\min_\Omega \frac{p}{q}\right|\leq
         \frac{\epsilon}{\min_\Omega q^2},\qquad \left |\max_{\Omega,\epsilon}\frac{p}{q}-\max_\Omega \frac{p}{q}\right|\leq
         \frac{\epsilon}{\min_\Omega q^2}, \]
       and
       \[ \left
           |\left
             \|\frac{p}{q}\right\|_{L^\infty(\Omega,\epsilon)}-\left
             \|\frac{p}{q}\right\|_{L^\infty(\Omega)}\right|
         \leq \frac{\epsilon}{\min_\Omega q^2}. \]
     \end{lemma}

     \begin{proof}
       Let $M_j$ and $\Lambda$ be as in Definition \ref{def:asup}.
       Obviously, we have $\max_\Omega \frac{p}{q}\geq
       \max_{\Omega,\epsilon}\frac{p}{q}$ and thus,
       by Lemmas \ref{lem:eval} and \ref{lem:LinfT}, we obtain
        \begin{align*}
   \left |\max_{\Omega,\epsilon}\frac{p}{q}-\max_\Omega \frac{p}{q}\right |
   &=\max_\Omega \frac{p}{q}-\max_\Lambda \frac{p}{q} 
   \leq
     \sum_{j=1}^n\frac{b_j-a_j}{M_j}\left \|\partial_j
     \frac{p}{q}\right \|_{L^\infty(\Omega)}
     =\sum_{j=1}^n\frac{b_j-a_j}{M_j}\left \|
     \frac{\partial_jpq-p\partial_jq}{q^2}\right \|_{L^\infty(\Omega)} \\
   &\leq \frac{1}{\min_\Omega q^2}\sum_{j=1}^n\frac{b_j-a_j}{M_j}\|\partial_jpq-p\partial_jq\|_{T(\Omega)} \\
          &\leq \frac{\epsilon}{\min_\Omega q^2}.
        \end{align*}
        The remaining estimates are a direct consequence of this because
        $\min \frac{p}{q}=-\max(-\frac{p}{q})$ and $\|\frac{p}{q}\|_{L^\infty}=\max\{|\min \frac{p}{q}|,|\max
        \frac{p}{q}|\}$.
      \end{proof}

      \begin{remark}
        Again, the crucial point is that for polynomials $p,q$ with
        rational coefficients and rational endpoints $a_j,b_j$,
        $\max_{\Omega,\epsilon}\frac{p}{q}$ can be computed \emph{exactly}.
      \end{remark}

      \begin{remark}
        We remark that the algorithms described in this section are very
        inefficient in practice. If one is willing to use
        interval arithmetic one has a much more sophisticated toolbox
        of modern numerical methods at hand that allows for a much
        faster implementation.
      \end{remark}
      
      \section{Approximations}
      \label{sec:approx}
      \noindent In this section we now start working with the concrete
      approximation and prove the necessary estimates using computer
      assistance. We provide in
       typewriter style the name of the accompanying \emph{Wolfram Mathematica}
       notebook that contains the implementation of the
       computer-assisted part of the respective item. The notebooks
       are available on the arXiv page of this paper, \url{https://arxiv.org/abs/2406.16597}.
      
\subsection{The approximate solution}
We define the approximate solution and rigorously estimate the
remainder $\mc R(a_*+a,f_*(\cdot,a))$.

   \begin{definition}[{\tt Definition\_6.1.nb}]
     \label{def:Pstar}
     Let
     \[ P_*(y):=\sum_{n=0}^{50} c_n(P_*) T_n(y) \]
     with $(c_n(P_*))_{n=0}^{50}\subset \C$ given in Appendix \ref{apx:Pstar}. We set
     $a_*:=\frac{772201763088846}{841768781900003}$ and
     define $f_*: [-1,1]\times \R\setminus\{-a_*\}$ by
     \[ f_*(y,a):=P_*(y)+\frac{2i(a_*+a)}{1-i(a_*+a)}P_*'(1)-P_*(1). \]
   \end{definition}

   \begin{lemma}
     The pair $(a_*,f_*)$ is a nontrivial admissible approximation.
   \end{lemma}

   \begin{proof}
     Evidently, $(a_*,f_*)\in \R\setminus\{0\}\times C^2([-1,1]\times
     \R\setminus\{-a_*\})$ and
     \[ f_*'(1,a)=\frac{1-i(a_*+a)}{2i(a_*+a)}f_*(1,a) \]
     for all $a\in \R\setminus\{-a_*\}$.
   \end{proof}

      \begin{definition}
     For brevity, we set
     \[ J_*:=[-10^{-10},10^{-10}]. \]
   \end{definition}

   Next, we provide the crucial result that measures the quality of the
   approximate solution by estimating the error when one inserts the
   approximation into the equation. Note that if $(a_*,f_*)$ were a
   true solution, we would have $\mc R(a_*,f_*(\cdot,0))=0$.
   
   \begin{proposition}[{\tt Proposition\_6.4.nb}]
  \label{prop:admapr}
  We have
  \[ \|\mc R(a_*+a, f_*(\cdot, a))\|_Y\leq 4\cdot 10^{-9} \]
  for all $a\in J_*$. 
\end{proposition}

\begin{proof}
  By definition, we have
  \[
    \mc
    R(a_*+a,f_*(\cdot,a))(y)=f_*''(y,a)+p_0(y,a_*+a)f_*'(y,a)+q_0(y,a_*+a)f_*(y,a)+\frac{f_*(y,a)^2\overline{f_*(y,a)}}{(1-y)^2} \]
  with
\begin{align*}
  p_0(y,a_*+a)&=\frac{4i(a_*+a)}{(1-y)^3}-\frac{2i(a_*+a)}{(1-y)^2}-\frac{2+\frac{2i}{a_*+a}}{1-y}+\frac{2}{1+y} \\
  q_0(y,a_*+a)&=-\frac{2-2i(a_*+a)}{(1-y)^3}-\frac{\frac{1}{(a_*+a)^2}-\frac{i}{a_*+a}}{(1-y)^2}-\frac{1
                +\frac{i}{a_*+a}}{1-y}-\frac{1+\frac{i}{a_*+a}}{1+y}.
\end{align*}
We set
\[ R(y,a):=(a_*+a)^2[1-i(a_*+a)]^2[1+i(a_*+a)](1+y)(1-y)^2\mc
  R(a_*+a,f_*(\cdot,a))(y). \]
By construction and Lemma \ref{lem:R}, $R$ is a polynomial of degree
at most $(151, 6)$. Furthermore, we have
\[ \left |(a_*+a)^2[1-i(a_*+a)]^2[1+i(a_*+a)]\right|\geq 2 \]
for all $a\in J_*$ and this yields
\begin{align*}
  \|\mc R(a_*+a,f_*(\cdot,a))\|_Y
  &=\sup_{y\in (-1,1)}(1+y)(1-y)^2\left |\mc
    R(a_*+a,f_*(\cdot,a))(y)\right|\\
    &\leq \tfrac{1}{2}\sqrt{\|\Re R\|_{T([-1,1]\times J_*)}^2+
      \|\Im R\|_{T([-1,1]\times J_*)}^2}. 
\end{align*}
We compute
\[ \|\Re R\|_{T([-1,1]\times J_*)}^2+
      \|\Im R\|_{T([-1,1]\times J_*)}^2\leq 4\left (4\cdot
      10^{-9}\right)^2, \]
  which implies the claimed bound.
  \end{proof}

\subsection{The approximate left fundamental matrix}

We turn to the fundamental systems and start with the left-hand
side. Using the polynomials from Appendix \ref{apx:PL}, we define the
left fundamental matrix and prove that it is admissible.

\begin{definition}[{\tt Definition\_6.5.nb}]
  \label{def:PL}
  Let
   \[ P_{L,j}(y):=\sum_{n=0}^{16}c_n(P_{L,j})T_n(2y+1),\qquad j\in
     \{1,2,3,4\} \]
   with $(c_n(P_{L,j}))_{n=0}^{16}\subset \C$ given in Appendix \ref{apx:PL}.
   For $y\in (-1,0]$ we set
    \begin{align*}
    f_{L,1}(y)&:=1+(1+y)P_{L,1}(y) & f_{L,2}(y)&:=i+(1+y)P_{L,2}(y) \\
    f_{L,3}(y)&:=\frac{1}{1+y}[1+(1+y)^2P_{L,3}(y)] &
                                               f_{L,4}(y)&:=\frac{i}{1+y}[1+(1+y)^2P_{L,4}(y)]
    \end{align*}
    and
    \[ F_L:=
    \begin{pmatrix}
      \Re f_{L,1} & \Re f_{L,2} & \Re f_{L,3} & \Re f_{L,4} \\
      \Im f_{L,1} & \Im f_{L,2} & \Im f_{L,3} & \Im f_{L,4} \\
      \Re f_{L,1}' & \Re f_{L,2}' & \Re f_{L,3}' & \Re f_{L,4}' \\
      \Im f_{L,1}' & \Im f_{L,2}' & \Im f_{L,3}' & \Im f_{L,4}' \\
    \end{pmatrix}. \]
\end{definition}

\begin{definition}[{\tt Definition\_6.6.nb}]
  We set $P_{d,L}(y):=(1+y)^4\det(F_L(y))$.
\end{definition}

\begin{remark}
  \label{rem:PdL}
  Note that $\deg P_{d,L}\leq 72$.
\end{remark}

\begin{lemma}[{\tt Lemma\_6.8.nb}]
  \label{lem:FLadm}
  We have
  \[ \min_{[-1,0]}P_{d,L}\geq \tfrac{99}{100}. \]
  In particular, $F_L$ is an admissible left fundamental matrix.
\end{lemma}

\begin{proof}
   It suffices to compute
  \[ \min_{[-1,0],\frac{1}{100}}P_{d,L}=1. \]
\end{proof}

\subsection{The approximate right fundamental matrix}

Now we turn to the more involved right-hand side. Again, using the
polynomials from Appendix \ref{apx:PR} we define the right fundamental
matrix and show that it is admissible.

\begin{definition}[Polynomials for fundamental system, {\tt Definition\_6.9.nb}]
  \label{def:PR}
  Let 
  \[ P_{R,j}(y):=\sum_{n=0}^{30}c_n(P_{R,j})T_n(2y-1),\qquad j\in \{1,2,3,4\} \]
  and
  \[ Q_{R,j}(y):=\sum_{n=0}^{26} c_n(Q_{R,j})T_n(2y-1),\qquad j\in \{3,4\}, \]
  where the coefficients $(c_n(P_{R,j}))_{n=0}^{30}\subset \C$ and
  $(c_n(Q_{R,j}))_{n=0}^{26}\subset \C$ are given in
  Appendix \ref{apx:PR}. For $(y,a)\in [0,1)\times \R\setminus\{-a_*\}$ we set
   \begin{align*}
    f_{R,1}(y,a)&:=1+\frac{a_*+a+i}{2(a_*+a)}(1-y)+(1-y)^2P_{R,1}(y)
     \\
     f_{R,2}(y,a)&:=i+i\frac{a_*+a+i}{2(a_*+a)}(1-y)+(1-y)^2P_{R,2}(y) \\
    f_{R,3}(y,a)&:=e^{i\varphi(y,a_*+a)}(1-y)\left [1+\frac{2(a_*+a)-i}{2(a_*+a)}(1-y)+(1-y)^2
              P_{R,3}(y)\right ] \\
    &\quad +e^{-i\varphi(y,a_*+a)}(1-y)^5Q_{R,3}(y) \\
    f_{R,4}(y,a)&:=ie^{i\varphi(y,a_*+a)}(1-y)\left [1+
              \frac{2(a_*+a)-i}{2(a_*+a)}(1-y)+(1-y)^2
              P_{R,4}(y)\right ] \\
    &\quad +ie^{-i\varphi(y,a_*+a)}(1-y)^5Q_{R,4}(y)
   \end{align*}
and define
 \[ F_R(y,a):=\begin{pmatrix}
      \Re f_{R,1}(y,a) & \Re f_{R,2}(y,a) & \Re f_{R,3}(y,a) & \Re f_{R,4}(y,a) \\
      \Im f_{R,1}(y,a) & \Im f_{R,2}(y,a) & \Im f_{R,3}(y,a) & \Im f_{R,4}(y,a) \\
      \Re f_{R,1}'(y,a) & \Re f_{R,2}'(y,a) &
      \Re f_{R,3}'(y,a) & \Re f_{R,4}'(y,a) \\
      \Im f_{R,1}'(y,a) & \Im f_{R,2}'(y,a) &
      \Im f_{R,3}'(y,a) & \Im f_{R,4}'(y,a) 
    \end{pmatrix}. \]
\end{definition}

Due to the appearance of the exponential in the definition of the
fundamental system, we need a number of auxiliary polynomials in order
to separate the polynomial parts from the transcendental ones. The
philosophy is to
carry along the transcendental functions explicitly and leave the treatment of
the polynomials to the computer.

\begin{definition}[Polynomials for fundamental matrix, {\tt Definition\_6.10.nb}]
  For $(y,a)\in [0,1]\times \R$ and $j\in \{3,4\}$, we set
  \begin{align*}
    P_{C,j}(y,a)&:=(a_*+a)(1-y)\left [2-y+(1-y)^2\Re
                  P_{R,j}(y)+(1-y)^4\Re Q_{R,j}(y)\right] \\
    P_{S,j}(y,a)&:=(1-y)^2\left[\tfrac12-(a_*+a)(1-y)\Im
                  P_{R,j}(y)+(a_*+a)(1-y)^3\Im Q_{R,j}(y)\right]
  \end{align*}
  as well as
  \begin{align*}
    Q_{C,j}(y,a)&:=-(1-y)^2\left [\tfrac12-(a_*+a)(1-y)\Im
             P_{R,j}(y)-(a_*+a)(1-y)^3\Im Q_{R,j}(y)\right] \\
    Q_{S,j}(y,a)&:=(a_*+a)(1-y)\left[2-y+(1-y)^2\Re
                  P_{R,j}(y)-(1-y)^4\Re Q_{R,j}(y)\right].
  \end{align*}
\end{definition}

\begin{remark}
  \label{rem:degPCj}
  Note that $\deg P_{X,j}\leq (33,1)$ and $\deg Q_{X,j}\leq (33,1)$
  for $X\in \{C,S\}$ and $j\in \{3,4\}$.
\end{remark}

\begin{definition}(Polynomials for derivatives of fundamental system,
  {\tt Definition\_6.12.nb})
  \label{def:PCj1}
  For $(y,a)\in [0,1)\times \R$ and $j\in \{3,4\}$, we set
    \begin{align*}
    P_{C,j}^{(1)}(y,a)&:=(a_*+a)(1-y)^2\left
                             [
                             \varphi'(y,a_*+a)P_{S,j}(y,a)+P_{C,j}'(y,a)\right] \\
    P_{S,j}^{(1)}(y,a)&:=(a_*+a)(1-y)^2\left [-
                             \varphi'(y,a_*+a)P_{C,j}(y,a)+P_{S,j}'(y,a)\right]
    \end{align*}
    as well as
    \begin{align*}
         Q_{C,j}^{(1)}(y,a)&:=(a_*+a)(1-y)^2\left
                             [
                             \varphi'(y,a_*+a)Q_{S,j}(y,a)+Q_{C,j}'(y,a)\right] \\
  Q_{S,j}^{(1)}(y,a)&:=(a_*+a)(1-y)^2\left [-
                             \varphi'(y,a_*+a)Q_{C,j}(y,a)+Q_{S,j}'(y,a)\right].
  \end{align*}
\end{definition}

\begin{remark}[{\tt Remark\_6.13.nb}]
  \label{rem:degPCj1}
  By explicit evaluation it follows that
  \[ P_{C,j}(1,a)=P_{S,j}(1,a)=Q_{C,j}(1,a)=Q_{S,j}(1,a)=0 \]
  for $j\in \{3,4\}$ and $a\in \{0,1\}$ and thus, the functions
  defined in Definition \ref{def:PCj1} are polynomials of degree at most $(34,3)$. 
\end{remark}

\begin{definition}[Polynomials for decomposition of fundamental
  matrix, {\tt Definition\_6.14.nb}]
    \label{def:PC}
  For $(y,a)\in [0,1)\times J_*$, we set
  \begin{align*}
    P_{0}(y,a):=
    \begin{pmatrix}
      (a_*+a)\Re f_{R,1}(y,a) & (a_*+a)\Re f_{R,2}(y,a) & 0 & 0 \\
      (a_*+a)\Im f_{R,1}(y,a) & (a_*+a)\Im f_{R,2}(y,a) & 0 & 0 \\
      (a_*+a)\Re f_{R,1}'(y,a) & (a_*+a)\Re f_{R,2}'(y,a) & 0 & 0 \\
      (a_*+a)\Im f_{R,1}'(y,a) & (a_*+a)\Im f_{R,2}'(y,a) & 0 & 0 
    \end{pmatrix},
  \end{align*}
  \begin{align*}
    P_{C}(y,a):=
    \begin{pmatrix}
      0 & 0 & (a_*+a)(1-y)^2 P_{C,3}(y,a) &
                                            -(a_*+a)(1-y)^2Q_{C,4}(y,a)
      \\
      0 & 0 & (a_*+a)(1-y)^2 Q_{C,3}(y,a) &
                                            (a_*+a)(1-y)^2P_{C,4}(y,a)
      \\
      0 & 0 & P_{C,3}^{(1)}(y,a) & -Q_{C,4}^{(1)}(y,a) \\
      0 & 0 & Q_{C,3}^{(1)}(y,a) & P_{C,4}^{(1)}(y,a)
    \end{pmatrix},
  \end{align*}
  and
  \begin{align*}
    P_{S}(y,a):=
    \begin{pmatrix}
      0 & 0 & (a_*+a)(1-y)^2 P_{S,3}(y,a) &
                                            -(a_*+a)(1-y)^2Q_{S,4}(y,a)
      \\
      0 & 0 & (a_*+a)(1-y)^2 Q_{S,3}(y,a) & (a_*+a)(1-y)^2P_{S,4}(y,a)
      \\
      0 & 0 & P_{S,3}^{(1)}(y,a) & -Q_{S,4}^{(1)}(y,a) \\
      0 & 0 & Q_{S,3}^{(1)}(y,a) & P_{S,4}^{(1)}(y,a)
    \end{pmatrix}
  \end{align*}
\end{definition}

\begin{remark}
  \label{rem:degP0}
  Note that $\deg (P_X)^j{}_k\leq (35,3)$
   for $X\in \{0,C,S\}$ and $j,k\in \{1,2,3,4\}$.
\end{remark}

\begin{lemma}
  \label{lem:FRdecomp}
  We have the decomposition
  \[
    F_R(y,a)=\frac{1}{a_*+a}P_{0}(y,a)+\frac{\cos(\varphi(y,a_*+a))}{(a_*+a)^2(1-y)^2}P_{C}(y,a)
    +\frac{\sin(\varphi(y,a_*+a))}{(a_*+a)^2(1-y)^2}P_{S}(y,a). \]
\end{lemma}

\begin{proof}
    By definition, we have
  \begin{align*}
    \Re
    f_{R,3}(y,a)&=\frac{\cos(\varphi(y,a_*+a))}{a_*+a}P_{C,3}(y,a)+\frac{\sin(\varphi(y,a_*+a))}{a_*+a}P_{S,3}(y,a)
    \\
    \Im
    f_{R,3}(y,a)&=\frac{\cos(\varphi(y,a_*+a))}{a_*+a}Q_{C,3}(y,a)+\frac{\sin(\varphi(y,a_*+a))}{a_*+a}Q_{S,3}(y,a) 
  \end{align*}
  and
  \begin{align*}
    \Re
    f_{R,4}(y,a)&=-\frac{\cos(\varphi(y,a_*+a))}{a_*+a}Q_{C,4}(y,a)-\frac{\sin(\varphi(y,a_*+a))}{a_*+a}Q_{S,4}(y,a) \\
    \Im
    f_{R,4}(y,a)&=\frac{\cos(\varphi(y,a_*+a))}{a_*+a}P_{C,4}(y,a)+\frac{\sin(\varphi(y,a_*+a))}{a_*+a}P_{S,4}(y,a).
  \end{align*}
  Furthermore,
  \begin{align*}
    \Re
    f_{R,3}'(y,a)&=\frac{\cos(\varphi(y,a_*+a))}{(a_*+a)^2(1-y)^2}
                   P_{C,3}^{(1)}(y,a)+\frac{\sin(\varphi(y,a_*+a))}{(a_*+a)^2(1-y)^2}
                   P_{S,3}^{(1)}(y,a) \\
 \Im f_{R,3}'(y,a)&=\frac{\cos(\varphi(y,a_*+a))}{(a_*+a)^2(1-y)^2}
                   Q_{C,3}^{(1)}(y,a)+\frac{\sin(\varphi(y,a_*+a))}{(a_*+a)^2(1-y)^2}
                   Q_{S,3}^{(1)}(y,a)
  \end{align*}
  as well as
  \begin{align*}
    \Re
    f_{R,4}'(y,a)&=-\frac{\cos(\varphi(y,a_*+a))}{(a_*+a)^2(1-y)^2}
                   Q_{C,4}^{(1)}(y,a)-\frac{\sin(\varphi(y,a_*+a))}{(a_*+a)^2(1-y)^2}
                   Q_{S,4}^{(1)}(y,a) \\
    \Im f_{R,4}'(y,a)&=\frac{\cos(\varphi(y,a_*+a))}{(a_*+a)^2(1-y)^2}
                   P_{C,4}^{(1)}(y,a)+\frac{\sin(\varphi(y,a_*+a))}{(a_*+a)^2(1-y)^2}
                   P_{S,4}^{(1)}(y,a)
  \end{align*}
  and the stated decomposition follows.
\end{proof}

\begin{definition}[{\tt Definition\_6.17.nb}]
  We set
  \[ P_{d,R}:=\det\left (P_{0}+P_{C}\right). \]
\end{definition}

\begin{remark}
  \label{rem:degPdR}
  Note that $\deg P_{d,R}\leq (140,12)$.
\end{remark}

\begin{lemma}[{\tt Lemma\_6.19.nb}]
  \label{lem:FRadm}
  We have
    \[ (a_*+a)^6(1-y)^4\det(F_R(y,a))=P_{d,R}(y,a)+\varepsilon_{d,R}(y,a) \]
  for all $(y,a)\in [0,1)\times J_*$, where
\[   \sup_{(y,a)\in [0,1)\times J_*}\left |\frac{\varepsilon_{d,R}(y,a)}{(1-y)^2}\right|\leq
  10^{-17} \]
and 
\[ \min_{[0,1]\times J_*} P_{d,R}\geq 8. \]
  In particular, the matrix $F_R$ is an admissible right fundamental matrix on
  $J_*$ with parameter $a_*$. 
\end{lemma}

\begin{proof}
By Lemma \ref{lem:FRdecomp} and multilinearity, we obtain the decomposition
  \begin{align*}
    (a_*+a)^6(1-y)^4 \det(F_R(y,a))
    &=
      \cos(\varphi(y,a_*+a))^2 P_{d,R}(y,a) \\
    &\quad +\cos(\varphi(y,a_*+a))\sin(\varphi(y,a_*+a))Q_2(y,a) \\
    &\quad +\sin(\varphi(y,a_*+a))^2 Q_3(y,a) \\
    &=P_{d,R}(y,a)+\cos(\varphi(y,a_*+a))\sin(\varphi(y,a_*+a))Q_2(y,a)
    \\
    &\quad +\sin(\varphi(y,a_*+a))^2 \left [Q_3(y,a)-P_{d,R}(y,a)\right],
  \end{align*}
  where we abbreviate
  \[
    Q_2:=\tfrac12\det\left
      (P_{0}+P_{C}+P_{S}\right)-\tfrac12\det\left (P_{0}+P_{C}-P_{S}\right) \]
       and
       $Q_3:=\det(P_{0}+P_{S})$.
       Note that $\deg P_{d,R}\leq (140,12)$ and
  $\deg Q_j\leq (140, 12)$ for $j\in \{2,3\}$.
       By explicit evaluation, we find
       \[ \partial_1^j
         Q_3(1,a)-\partial_1^j P_{d,R}(1,a)=\partial_1^j Q_2(1,a)=0 \]
       for $j\in \{0,1\}$ and $a\in \{0,1,2,\dots,12\}$.
       Consequently,
       \[ (y,a)\mapsto \frac{Q_3(y,a)-P_{d,R}(y,a)}{(1-y)^2},\qquad
         (y,a)\mapsto \frac{Q_2(y,a)}{(1-y)^2} \]
       are polynomials and computing their $T$-norms on $[0,1]\times
       J_*$ yields the claimed bound on $\varepsilon_{d,R}$.
       Furthermore, we compute
       \[ \min_{[0,1]\times J_*,\frac{1}{100}}P_{d,R}\geq 8.02,\]
  which finishes the proof.
\end{proof}

\subsection{The global fundamental matrix}

Now we can finally glue together the left fundamental matrix and the
right fundamental matrix.

\begin{definition}[{\tt Definition\_6.20.nb}]
  For $(y,a)\in (-1,1)\times \R\setminus \{-a_*\}$ we set
  \[ F(y,a):=1_{(-1,0]}(y)F_L(y)M_F(a)+1_{(0,1)}(y)F_R(y,a) \]
  with $M_F(a)=F_L(0)^{-1}F_R(0,a)$.
\end{definition}

\begin{remark}
  \label{rem:MF}
  Note that by Remark \ref{rem:degP0}, Lemma \ref{lem:FRdecomp}, and $\varphi(0,a_*+a)=0$ for
  all $a\in J_*$, it follows that $(a_*+a)^2 M_F(a)$ is a
  matrix with polynomial entries of degree at most $4$ in $a$.
\end{remark}

We prove that the fundamental matrix is admissible.

\begin{lemma}[{\tt Lemma\_6.22.nb}]
  \label{lem:Fadm}
  The matrix $F$ is an admissible fundamental matrix on
  $J_*$ with parameter $a_*$.
\end{lemma}

\begin{proof}
  Let $F_L$ and $F_R$ denote the left and right part of $F$,
  respectively.
  By Lemmas \ref{lem:FLadm} and \ref{lem:FRadm}, we only have to show that
  \[ 0\not=[F_L(0)^{-1}F_R(0,a)]^3{}_1=\sum_{\ell=1}^4 [F_L(0)^{-1}]^3{}_\ell
    F_R(0,a)^\ell{}_1 \] for
  all $a\in J_*$.
  Note that $\varphi(0,a_*+a)=0$ for all
  $a\in\R\setminus\{-a_*\}$ and thus, by Lemma \ref{lem:FRdecomp},
\[ F_R(0,a)=\frac{P_{0}(0,a)}{a_*+a}+\frac{P_{C}(0,a)}{(a_*+a)^2}. \]
This shows that $a\mapsto (a_*+a)^2 F_R(0,a)^\ell{}_1$ is a polynomial
of degree at most
  $4$ and so is
  $P(a):=(a_*+a)^2[F_L(0)^{-1}F_R(0,a)]^3{}_1$. We compute
  \[ \max_{J_*,\frac{1}{10}}P\leq -\tfrac15. \]
\end{proof}

\section{Explicit bounds on the operator norms}
\label{sec:op}
\noindent In this section we use the estimates from Section
\ref{sec:quant} and the concrete approximations from Section
\ref{sec:approx} to obtain quantitative bounds on the functional
$\psi_F$ and the operator $\mc J_F$. These bounds are then used in the
final fixed point argument as outlined in Section \ref{sec:strategy},
point (\ref{itm:contr}).
\subsection{The operator norm of $\alpha_F^k$}
We start with bounds on the integrands $\alpha_F^k$, see Definition
\ref{def:alpha}.
These involve the
inverse of the fundamental matrix and again, we need a number of
auxiliary polynomials to represent the inverse in a way that becomes
treatable by the computer. We employ Cramer's rule to obtain the
inverse and thus we need suitable bounds on the adjunct matrix and the determinant.

\begin{definition}[Polynomials for $\adj(F_R(y,a))$, {\tt Definition\_7.1.nb}]
  \label{def:QC}
  For $j\in \{1,2\}$ and $k\in \{1,2,3,4\}$, we set
  \begin{align*}
    (Q_{C})^j{}_k&:=0 & (Q_{C})^{j+2}{}_k&:=\adj\left (P_{0}+P_{C}\right)^{j+2}{}_k \\
    (Q_{S})^j{}_k&:=0 & (Q_{S})^{j+2}{}_k&:=\adj\left(P_{0}+P_{S}\right)^{j+2}{}_k
  \end{align*}
  as well as
  \begin{align*}
    (Q_{CC})^j{}_k&:=\adj\left (P_{0}+P_{C}\right )^j{}_k & (Q_{CC})^{j+2}{}_k&:=0 \\
    (Q_{CS})^j{}_k&:=\tfrac12\adj\left (P_{0}+P_{C}+P_{S}\right)^j{}_k
              -\tfrac12\adj\left(P_{0}+P_{C}-P_{S}\right)^j{}_k &
                                                                    (Q_{CS})^{j+2}{}_k&:=0 \\
    (Q_{SS})^j{}_k&:=\adj\left (P_{0}+P_{S}\right )^j{}_k & (Q_{SS})^{j+2}{}_k&:=0.
  \end{align*}
\end{definition}

\begin{remark}
  \label{rem:degQC}
  Note that $\deg (Q_X)^j{}_k\leq (105,9)$ and $\deg
  (Q_{XY})^j{}_k\leq (105,9)$ for $X\in \{C,S\}$, $XY\in \{CC, CS,
  SS\}$, and $j,k\in \{1,2,3,4\}$.
\end{remark}

\begin{lemma}
  \label{lem:adjFRdecomp}
  We have the decomposition
  \begin{align*}
    \adj(F_R(y,a))=&\frac{\cos(\varphi(y,a_*+a))}{(a_*+a)^4(1-y)^2}Q_{C}(y,a)
    +\frac{\sin(\varphi(y,a_*+a))}{(a_*+a)^4(1-y)^2}Q_{S}(y,a) \\
                   &+\frac{\cos(\varphi(y,a_*+a))^2}{(a_*+a)^5(1-y)^4}Q_{CC}(y,a)                     
                     +\frac{\sin(\varphi(y,a_*+a))^2}{(a_*+a)^5(1-y)^4}Q_{SS}(y,a) \\
    &+\frac{\cos(\varphi(y,a_*+a))\sin(\varphi(y,a_*+a))}{(a_*+a)^5(1-y)^4}Q_{CS}(y,a).
  \end{align*}
\end{lemma}

\begin{proof}
  By definition, Lemma \ref{lem:FRdecomp}, and multilinearity, we have
  \begin{align*}
    \adj&(F_R(y,a))^1{}_k\\
    &=\det\left (\delta_k,
      \frac{P_{0}(y,a)_2}{a_*+a},
      \frac{\cos(\varphi(y,a_*+a))}{(a_*+a)^2(1-y)^2}P_{C}(y,a)_3,
      \frac{\cos(\varphi(y,a_*+a))}{(a_*+a)^2(1-y)^2}P_{C}(y,a)_4
      \right) \\
    &\quad +\det\left (\delta_k,
      \frac{P_{0}(y,a)_2}{a_*+a},
      \frac{\cos(\varphi(y,a_*+a))}{(a_*+a)^2(1-y)^2}P_{C}(y,a)_3,
      \frac{\sin(\varphi(y,a_*+a))}{(a_*+a)^2(1-y)^2}P_{S}(y,a)_4
      \right) \\
    &\quad +\det\left (\delta_k,
      \frac{P_{0}(y,a)_2}{a_*+a},
      \frac{\sin(\varphi(y,a_*+a))}{(a_*+a)^2(1-y)^2}P_{S}(y,a)_3,
      \frac{\cos(\varphi(y,a_*+a))}{(a_*+a)^2(1-y)^2}P_{C}(y,a)_4
      \right) \\
    &\quad +\det\left (\delta_k,
      \frac{P_{0}(y,a)_2}{a_*+a},
      \frac{\sin(\varphi(y,a_*+a))}{(a_*+a)^2(1-y)^2}P_{S}(y,a)_3,
      \frac{\sin(\varphi(y,a_*+a))}{(a_*+a)^2(1-y)^2}P_{S}(y,a)_4
      \right)
  \end{align*}
  and accordingly for $\adj(F_R(y,a))^2{}_k$.
  Furthermore,
  \begin{align*}
    \adj(F_R(y,a))^3{}_k=&\det\left (\frac{P_{0}(y,a)_1}{a_*+a},
                          \frac{P_{0}(y,a)_2}{a_*+a},\delta_k,
                          \frac{\cos(\varphi(y,a_*+a))}{(a_*+a)^2(1-y)^2}P_{C}(y,a)_4\right
                          )\\
    &+\det\left (\frac{P_{0}(y,a)_1}{a_*+a},
                          \frac{P_{0}(y,a)_2}{a_*+a},\delta_k,
                          \frac{\sin(\varphi(y,a_*+a))}{(a_*+a)^2(1-y)^2}P_{S}(y,a)_4\right
                          )
  \end{align*}
  and accordingly for $\adj(F_R(y,a))^4{}_k$. Thus, the claim follows
  straightforwardly by multilinearity and the observation that
  \begin{align*}
&\tfrac12\adj\left (P_{0}+P_{C}+P_{S}\right)^1{}_k
  -\tfrac12\adj\left(P_{0}+P_{C}-P_{S}\right)^1{}_k \\
&=\tfrac12\det\left(\delta_k,(P_0)_2,(P_C)_3+(P_S)_3,(P_C)_4+(P_S)_4\right) \\
  &\quad
    -\tfrac12\det\left(\delta_k,(P_0)_2,(P_C)_3-(P_S)_3,(P_C)_4-(P_S)_4\right)
    \\
&=\det\left(\delta_k,(P_0)_2,(P_C)_3,(P_S)_4\right)
  +\det\left(\delta_k,(P_0)_2,(P_S)_3,(P_C)_4\right).
  \end{align*}
\end{proof}

    \begin{proposition}[{\tt Proposition\_7.4.nb}]
      \label{prop:numCalpha}
  We have the bounds
  \[ C_\alpha^1(F)\leq 0.2,\qquad C_\alpha^2(F)\leq 0.51,\qquad
    C_\alpha^3(F)\leq 0.47,\qquad C_\alpha^4(F)\leq 0.47. \]
\end{proposition}

\begin{proof}
  From Lemma \ref{lem:alphakF}, we have the bound
  \[ C_\alpha^k(F)\leq \sup_{(y,a)\in
      (-1,1)\times J_*}\frac{|F^{-1}(y,a)^k{}_3|+|F^{-1}(y,a)^k{}_4|}{(1+y)(1-y)^2}. \]
 As always, we denote by $F_L$ and $F_R$ the left and right part of
 $F$, respectively. We distinguish between $y\in (-1,0]$ and $y\in
 (0,1)$ and start with the latter case.
 By Lemma \ref{lem:FRadm}, we have
 \begin{align*}
   \frac{F^{-1}(y,a)}{(1+y)(1-y)^2}
   &=\frac{\adj(F_R(y,a))}{(1+y)(1-y)^2\det(F_R(y,a))}
     =\frac{(a_*+a)^6(1-y)^2\adj(F_R(y,a))}{(1+y)[P_{d,R}(y,a)+\varepsilon_{d,R}(y,a)]}
 \end{align*}
 and Lemma \ref{lem:adjFRdecomp} yields the bound
 \begin{equation}
   \label{eq:prfCalphaadj}
   \begin{split}
     (a_*+a)^6(1-y)^2|\adj(F_R(y,a))^k{}_\ell| 
              &\leq
                (a_*+a)^2|Q_C(y,a)^k{}_\ell|+(a_*+a)^2|Q_S(y,a)^k{}_\ell| \\
                &\quad
                  +\left|\frac{(a_*+a)Q_{CC}(y,a)^k{}_\ell}{(1-y)^2}\right|
                  +\left|\frac{(a_*+a)Q_{CS}(y,a)^k{}_\ell}{(1-y)^2}\right|
     \\
              &\quad
                +\left|\frac{(a_*+a)[Q_{SS}(y,a)^k{}_\ell-Q_{CC}(y,a)^k{}_\ell]}{(1-y)^2}\right|
   \end{split}
   \end{equation}
   and by explicit evaluation, we find
   \[ \partial_1^j Q_{XY}(1,a)^k{}_\ell=0 \]
   for $k\in \{1,2\}$, $\ell\in \{3,4\}$, $j\in \{0,1\}$,
   $a\in \{0,1,2,\dots,9\}$, and $XY\in\{CC,CS,SS\}$.
   Consequently, since $\deg (Q_{XY})^k{}_\ell\leq (105,9)$ by Remark
   \ref{rem:degQC}, we see that
   \[ (y,a)\mapsto \frac{(a_*+a)Q_{XY}(y,a)^k{}_\ell}{(1-y)^2} \]
   are polynomials of degree at most $(105,10)$.
   We compute
   \begin{align*}
     \left \|(y,a)\mapsto \frac{(a_*+a)Q_{CS}(y,a)^k{}_\ell}{(1-y)^2}
     \right\|_{T([0,1]\times J_*)}&\leq 10^{-18} \\
     \left \|(y,a)\mapsto \frac{(a_*+a)[Q_{SS}(y,a)^k{}_\ell-Q_{CC}(y,a)^k{}_\ell]}{(1-y)^2}
     \right\|_{T([0,1]\times J_*)}&\leq 10^{-18} 
   \end{align*}
   and thus, since $(Q_{XY})^k{}_\ell=0$ for $XY\in \{CC, CS, SS\}$
   and $k\in \{3,4\}$, the last two terms in Eq.~\eqref{eq:prfCalphaadj} are
   negligible. Next, we compute
   \begin{align*}
     \max_{[0,1]\times J_*,1}\left
     |(y,a)\mapsto\frac{(a_*+a)\frac{
     Q_{CC}(y,a)^1{}_3}{(1-y)^2}}{(1+y)P_{d,R}(y,a)}\right|&\leq 0.045,
     &
      \max_{[0,1]\times J_*,1}\left
     |(y,a)\mapsto \frac{(a_*+a)\frac{
     Q_{CC}(y,a)^1{}_4}{(1-y)^2}}{(1+y)P_{d,R}(y,a)}\right|&\leq 0.12,
     \\
     \max_{[0,1]\times J_*,1}\left
     |(y,a)\mapsto \frac{(a_*+a)\frac{
     Q_{CC}(y,a)^2{}_3}{(1-y)^2}}{(1+y)P_{d,R}(y,a)}\right|&\leq 0.19,
     &
      \max_{[0,1]\times J_*,1}\left
     |(y,a)\mapsto \frac{(a_*+a)\frac{
     Q_{CC}(y,a)^2{}_4}{(1-y)^2}}{(1+y)P_{d,R}(y,a)}\right|&\leq 0.28.
   \end{align*}
   Furthermore, $(Q_C)^k{}_\ell=(Q_S)^k{}_\ell=0$ for $k\in \{1,2\}$ and
   \begin{align*}
     \max_{[0,1]\times J_*,1}\left
     |(y,a)\mapsto \frac{(a_*+a)^2
     Q_C(y,a)^3{}_3}{(1+y)P_{d,R}(y,a)}\right|&\leq 0.08,
     &
      \max_{[0,1]\times J_*,1}\left
     |(y,a)\mapsto \frac{(a_*+a)^2
       Q_C(y,a)^3{}_4}{(1+y)P_{d,R}(y,a)}\right|&\leq 0.15, \\
     \max_{[0,1]\times J_*,1}\left
     |(y,a)\mapsto \frac{(a_*+a)^2
     Q_C(y,a)^4{}_3}{(1+y)P_{d,R}(y,a)}\right|&\leq 0.15,
     &
      \max_{[0,1]\times J_*,1}\left
     |(y,a)\mapsto \frac{(a_*+a)^2
     Q_C(y,a)^4{}_4}{(1+y)P_{d,R}(y,a)}\right|&\leq 0.022,
   \end{align*}
   as well as
      \begin{align*}
     \max_{[0,1]\times J_*,1}\left
     |(y,a)\mapsto \frac{(a_*+a)^2
     Q_S(y,a)^3{}_3}{(1+y)P_{d,R}(y,a)}\right|&\leq 0.15,
     &
      \max_{[0,1]\times J_*,1}\left
     |(y,a)\mapsto \frac{(a_*+a)^2
       Q_S(y,a)^3{}_4}{(1+y)P_{d,R}(y,a)}\right|&\leq 0.022, \\
     \max_{[0,1]\times J_*,1}\left
     |(y,a)\mapsto \frac{(a_*+a)^2
     Q_S(y,a)^4{}_3}{(1+y)P_{d,R}(y,a)}\right|&\leq 0.08,
     &
      \max_{[0,1]\times J_*,1}\left
     |(y,a)\mapsto \frac{(a_*+a)^2
     Q_S(y,a)^4{}_4}{(1+y)P_{d,R}(y,a)}\right|&\leq 0.15.
   \end{align*}
   With $\|\varepsilon_{d,R}\|_{L^\infty([0,1)\times J_*)}\leq 10^{-17}$
   and $\min_{[0,1]\times J_*} P_{d,R}\geq 8$
from Lemma \ref{lem:FRadm},
this yields the bounds
\begin{align*}
  \sup_{(y,a)\in
      (0,1)\times
  J_*}\frac{|F^{-1}(y,a)^1{}_3|+|F^{-1}(y,a)^1{}_4|}{(1+y)(1-y)^2}
  &\leq 0.2 \\
    \sup_{(y,a)\in
      (0,1)\times
  J_*}\frac{|F^{-1}(y,a)^2{}_3|+|F^{-1}(y,a)^2{}_4|}{(1+y)(1-y)^2}
  &\leq 0.51
\end{align*}
and
\begin{align*}
  \sup_{(y,a)\in
      (0,1)\times
  J_*}\frac{|F^{-1}(y,a)^3{}_3|+|F^{-1}(y,a)^3{}_4|}{(1+y)(1-y)^2}
  &\leq 0.47 \\
    \sup_{(y,a)\in
      (0,1)\times
  J_*}\frac{|F^{-1}(y,a)^4{}_3|+|F^{-1}(y,a)^4{}_4|}{(1+y)(1-y)^2}
  &\leq 0.47.
\end{align*}

Next, we turn to the left-hand side $y\in (-1,0]$. In this case, we
have
\begin{align*}
  \frac{F^{-1}(y,a)}{(1+y)(1-y)^2}
  &=\frac{M_F(a)^{-1}F_L(y)^{-1}}{(1+y)(1-y)^2}
    =M_F(a)^{-1}\frac{(1+y)^3\adj(F_L(y))}{(1-y)^2P_{d,L}(y)}
\end{align*}
and thus,
\[ \frac{F^{-1}(y,a)^k{}_\ell}{(1+y)(1-y)^2}=\sum_{j=1}^4
  \frac{\adj(M_F(a))^k{}_j
    (1+y)^3\adj(F_L(y))^j{}_\ell}{\det(M_F(a))(1-y)^2P_{d,L}(y)}. \]
Note that for
$k\in \{1,2,3,4\}$ and $\ell\in \{3,4\}$, $y\mapsto
(1+y)^3\adj(F_L(y))^k{}_\ell$ are polynomials of degree at most 54 and
$\deg(y\mapsto (1-y)^2 P_{d,L}(y))\leq 74$. Furthermore, by Remark \ref{rem:MF},
$(a_*+a)^2M_F(a)$ is a matrix with
polynomial entries of degree at most $4$. Consequently,
$a\mapsto \det((a_*+a)^2M_F(a))$ is a polynomial of degree at most $16$ and
with this information, we compute
\begin{align*}
\max_{[-1,0]\times J_*,1}\left
  |(y,a)\mapsto\frac{F^{-1}(y,a)^1{}_3}{(1+y)(1-y)^2}\right|
  &\leq 0.07,
  &
\max_{[-1,0]\times J_*,1}\left
  |(y,a)\mapsto\frac{F^{-1}(y,a)^1{}_4}{(1+y)(1-y)^2}\right|
  &\leq 0.1,   \\
  \max_{[-1,0]\times J_*,1}\left
  |(y,a)\mapsto\frac{F^{-1}(y,a)^2{}_3}{(1+y)(1-y)^2}\right|
  &\leq 0.16,
  &
\max_{[-1,0]\times J_*,1}\left
  |(y,a)\mapsto\frac{F^{-1}(y,a)^2{}_4}{(1+y)(1-y)^2}\right|
  &\leq 0.34,
\end{align*}
as well as
\begin{align*}
\max_{[-1,0]\times J_*,1}\left
  |(y,a)\mapsto\frac{F^{-1}(y,a)^3{}_3}{(1+y)(1-y)^2}\right|
  &\leq 0.09,
  &
\max_{[-1,0]\times J_*,1}\left
  |(y,a)\mapsto\frac{F^{-1}(y,a)^3{}_4}{(1+y)(1-y)^2}\right|
  &\leq 0.03,   \\
  \max_{[-1,0]\times J_*,1}\left
  |(y,a)\mapsto\frac{F^{-1}(y,a)^4{}_3}{(1+y)(1-y)^2}\right|
  &\leq 0.09,
  &
\max_{[-1,0]\times J_*,1}\left
  |(y,a)\mapsto\frac{F^{-1}(y,a)^4{}_4}{(1+y)(1-y)^2}\right|
  &\leq 0.05.
\end{align*}
Furthermore, we compute
\[ \min_{[-1,0]\times J_*,1}\left ((y,a)\mapsto
    \det(M_F(a))(1-y)^2P_{d,L}(y)\right)\geq 50\]
and the claimed bounds follow.
\end{proof}

\subsection{The operator norm of $\psi_F$}
Next, we obtain a concrete numerical bound on the norm of the functional $\psi_F$.

\begin{lemma}[{\tt Lemma\_7.5.nb}]
  \label{lem:numCpsi}
  We have $C_\psi(F)\leq 13$.
\end{lemma}

\begin{proof}
  By Lemma \ref{lem:quantCpsi}, we have
  \begin{align*} C_{\psi}(F)
    &\leq 2\max_{a\in J}\sum_{k=3}^4 \left |
      M_F(a)^4{}_1\frac{M_F(a)^3{}_k}{M_F(a)^3{}_1}-M_F(a)^4{}_k\right
      | C_{\alpha}^k(F)
    \\
    &=2\max_{a\in J}\sum_{k=3}^4 \left|\frac{
      (a_*+a)^4M_F(a)^4{}_1
      M_F(a)^3{}_k-(a_*+a)^4M_F(a)^3{}_1 M_F(a)^4{}_k}{(a_*+a)^4M_F(a)^3{}_1}\right
      |C_\alpha^k(F)
  \end{align*}
    and from Remark \ref{rem:MF} we recall
    that $(a_*+a)^2M_F(a)$ has polynomial entries of
    degree at most $4$.
    We compute
    \begin{align*}
      \max_{J_*,\frac{1}{1000}}\left |a\mapsto \frac{(a_*+a)^4M_F(a)^4{}_1
      M_F(a)^3{}_3-(a_*+a)^4M_F(a)^3{}_1 M_F(a)^4{}_3}{(a_*+a)^4M_F(a)^3{}_1}\right
      |&\leq 9.9
      \\
       \max_{J_*,\frac{1}{1000}}\left |a\mapsto \frac{(a_*+a)^4M_F(a)^4{}_1
      M_F(a)^3{}_4-(a_*+a)^4M_F(a)^3{}_1 M_F(a)^4{}_4}{(a_*+a)^4M_F(a)^3{}_1}\right
      |&\leq 3.5
    \end{align*}
    as well as
    \[ \min_{J_*,10^{-4}}\left(a\mapsto \left
          |(a_*+a)^4M_F(a)^3{}_1\right |\right)\geq 0.2008 \]
    and the claimed bound follows.
\end{proof}

\subsection{The operator norm of $\mc J_F$}
Finally, we turn to the operator norm of $\mc J_F$. This involves the
auxiliary quantities $\beta_{F,L}^n$, $\beta_{F,M}^n$, and
$\beta_{F,L}^n$.

\begin{proposition}[{\tt Proposition\_7.6.nb}]
  \label{prop:numCJ}
  We have
  $C_{\mc J}(F)\leq 232$.
\end{proposition}

\begin{proof}
  By Definition \ref{def:beta} and Proposition \ref{prop:numCalpha}, we have
  \[ \beta_{F,L}^n(y,a)\leq \sum_{k=1}^2\sum_{m=1}^2
    \frac{|P^{nm}_{k}(y,a)|}{(a_*+a)^2}+\sum_{k=3}^4\sum_{m=1}^2
    \frac{|Q^{nm}_{k}(y,a)|}{|(a_*+a)^4M_F(a)^3{}_1|}+\sum_{k=3}^4\sum_{m=1}^2 \frac{|R^{nm}_{k}(y,a)|}{|(a_*+a)^4M_F(a)^3{}_1|} \]
  with
  \begin{align*}
    P^{nm}_{k}(y,a)&:=c_k(1+y)(1-y^2)^n(a_*+a)^2[F_L(y)M_F(a)]^{m+2n}{}_k
    \\
    Q^{nm}_k(y,a)&:=c_k(1+y)(1-y^2)^n(a_*+a)^4 \\
    &\quad\times \left (\sum_{\ell=1}^3 
                   F_L(y)^{m+2n}{}_\ell M_F(a)^\ell{}_1 M_F(a)^3{}_k
                   +F_L(y)^{m+2n}{}_4M_F(a)^4{}_k M_F(a)^3{}_1\right )
    \\
    R^{nm}_k(y,a)&:=c_k(1-y)(1-y^2)^n(a_*+a)^4 \\
    &\quad\times \sum_{\ell=1}^2
                   F_L(y)^{m+2n}{}_\ell  \left (M_F(a)^\ell{}_1 M_F(a)^3{}_k
                   -M_F(a)^\ell{}_k M_F(a)^3{}_1\right )
  \end{align*}
  and
  $(c_k)=(\frac{1}{5},\frac{51}{100},\frac{47}{100},\frac{47}{100})$.
By Remark \ref{rem:MF}, $X^{nm}_k$ for $X\in \{P,Q,R\}$ are
polynomials with $\deg X^{nm}_k\leq (20,12)$.
  Furthermore, from the proof of Lemma \ref{lem:numCpsi} we recall the
  estimate $|(a_*+a)^4M_F(a)^3{}_1|\geq \frac{1}{5}$ and with
  $(a_*+a)^2\geq\frac{84}{100}$, we obtain
  \begin{align*} \sum_{n=0}^1\sup_{(-1,-\frac12)\times J_*)}\beta_{F,L}^n
    &\leq \tfrac{100}{84}\sum_{n=0}^1\sum_{k=1}^2\sum_{m=1}^2
   \|P^{nm}_{k}\|_{T([-1,-\frac12]\times J_*)}+5\sum_{n=0}^1\sum_{k=3}^4\sum_{m=1}^2
    \|Q^{nm}_{k}\|_{T([-1,-\frac12]\times
      J_*)} \\
    &\quad +5\sum_{n=0}^1\sum_{k=3}^4\sum_{m=1}^2 
      \|R^{nm}_{k}\|_{T([-1,-\frac12]\times J_*)} \\
    &\leq 54.
      \end{align*}

  Next, we have
  \begin{align*}
    \beta_{F,M}^n(y,a)
    &\leq
      \sum_{k=1}^2\sum_{m=1}^2
    \frac{|P^{nm}_{k}(y,a)|}{(a_*+a)^2}+\sum_{k=3}^4\sum_{m=1}^2
      \frac{|Q^{nm}_{k}(y,a)|}{|(a_*+a)^4M_F(a)^3{}_1|} \\
    &\quad +\sum_{k=3}^4\sum_{m=1}^2
      \frac{|R^{nm}_{k}(y,a)|}{|(a_*+a)^4M_F(a)^3{}_1|(1+y)}
      +\sum_{m=1}^2
                   \frac{|p^{nm}(y)|}{1+y}+\sum_{m=1}^2 |q^{nm}(y)|
  \end{align*}
  with
  \begin{align*}
    P^{nm}_k(y,a)&:=c_k(1+y)(1-y^2)^n(a_*+a)^2[F_L(y)M_F(a)]^{m+2n}{}_k
    \\
    Q^{nm}_k(y,a)&:=c_k(1+y)(1-y^2)^n(a_*+a)^4
                   [F_L(y)M_F(a)]^{m+2n}{}_1M_F(a)^3{}_k \\
    R^{nm}_k(y,a)&:=c_k(1-y^2)^{1+n}(a_*+a)^4\sum_{\ell\in
                   \{1,2,4\}}F_L(y)^{m+2n}{}_\ell\left
                   (M_F(a)^\ell{}_1M_F(a)^3{}_k-M_F(a)^3{}_1M_F(a)^\ell{}_k\right)
  \end{align*}
  as well as
  \begin{align*}
    p^{nm}(y)&:=13(1+y)(1-y^2)^nF_L(y)^{m+2n}{}_4 \\
    q^{nm}(y)&:=39n(1-y^2)^nF_L(y)^m{}_4.
  \end{align*}
  We have $\deg(X^{nm}_k)\leq (21,12)$ for $X\in \{P,Q,R\}$ and compute
   \begin{align*}
    \|P^{nm}_k-P^{nm}_k(\cdot,0)\|_{T([-\frac12,0]\times J_*)}
    &\leq 3\cdot 10^{-10}
    &
      (n,m,k)&\in
      \{0,1\}\times
      \{1,2\}\times
      \{1,2\} \\
    \|Q^{nm}_k-Q^{nm}_k(\cdot,0)\|_{T([-\frac12,0]\times J_*)}
    &\leq 2\cdot 10^{-9}
    &
      (n,m,k)&\in
      \{0,1\}\times
      \{1,2\}\times
      \{3,4\} \\
    \|R^{nm}_k-R^{nm}_k(\cdot,0)\|_{T([-\frac12,0]\times J_*)}
    &\leq 2\cdot 10^{-9}
    &
      (n,m,k)&\in
      \{0,1\}\times
      \{1,2\}\times
      \{3,4\},
   \end{align*}
   as well as
   \[ \|\partial_1 X^{nm}_k(\cdot,0)\|_{T([-\frac12,0])}\leq 9,\qquad
     \|(q^{nm})'\|_{T([-\frac12,0])}\leq 134. \]
   Furthermore,
   \[ \|y\mapsto (1+y)\partial_1
     R_k^{nm}(y,0)-R_k^{nm}(y,0)\|_{T([-\frac12,0]}\leq 6 \]
   and
   \[ \|y\mapsto (1+y)(p^{nm})'(y)-p^{nm}(y)\|_{T([-\frac12,0])}\leq
     43. \]
   With this information and the bound
   \begin{align*}
     \beta_{F,M}^n(y,a)
    &\leq
\tfrac{100}{84}\sum_{k=1}^2\sum_{m=1}^2
    |P^{nm}_{k}(y,a)|+5\sum_{k=3}^4\sum_{m=1}^2
      |Q^{nm}_{k}(y,a)| \\
    &\quad +5\sum_{k=3}^4\sum_{m=1}^2
      \frac{|R^{nm}_{k}(y,a)|}{1+y}
      +\sum_{m=1}^2
                   \frac{|p^{nm}(y)|}{1+y}+\sum_{m=1}^2 |q^{nm}(y)|,
   \end{align*}
   we compute by explicit evaluation the estimates
  \begin{align*}
    \sup_{(-\frac12,0)\times J_*}\beta_{F,M}^0&\leq 74, &
                                                          \sup_{(-\frac12,0)\times
                                                          J_*}\beta_{F,M}^1&\leq 158.
  \end{align*}

  Finally, for $y\in (0,1)$, we have $F(y,a)=F_R(y,a)$ and by Lemma
  \ref{lem:FRdecomp}, we obtain the bound
  \begin{align*}
    \beta_{F,R}^n(y,a)
    &\leq \sum_{k=1}^2\sum_{m=1}^2\sum_{X\in
      \{0,C,S\}}\frac{|P_{X,k}^{nm}(y,a)|}{(a_*+a)^2}
      +\sum_{k=3}^4 \sum_{m=1}^2 \sum_{X\in
      \{0,C,S\}}\frac{|Q_{X,k}^{nm}(y,a)|}{(a_*+a)^4M_F(a)^3{}_1} \\
    &\quad +\sum_{k=3}^4 \sum_{m=1}^2 \sum_{X\in
      \{0,C,S\}}\frac{|R_{X,k}^{nm}(y,a)|}{(a_*+a)^4M_F(a)^3{}_1}
  \end{align*}
  with
  \begin{align*}
    P_{0,k}^{nm}(y,a)&:=c_k(a_*+a)(1+y)(1-y^2)^nP_0(y,a)^{m+2n}{}_k \\
    P_{X,k}^{nm}(y,a)&:=c_k\frac{1+y}{(1-y)^2}(1-y^2)^nP_X(y,a)^{m+2n}{}_k,
  \end{align*}
  \begin{align*}
    Q_{0,k}^{nm}(y,a)&:=c_k(1-y)(1-y^2)^n(a_*+a)^3 \left[
                       M_F(a)^3{}_kP_0(y,a)^{m+2n}{}_1
                       -M_F(a)^3{}_1P_0(y,a)^{m+2n}{}_k\right] \\
    Q_{X,k}^{nm}(y,a)&:=c_k\frac{(1-y^2)^n}{1-y}(a_*+a)^2 \left[
                       M_F(a)^3{}_kP_X(y,a)^{m+2n}{}_1
                       -M_F(a)^3{}_1P_X(y,a)^{m+2n}{}_k\right],
  \end{align*}
  and
  \begin{align*}
    R_{0,k}^{nm}(y,a)&:=c_k(1+y)(1-y^2)^n(a_*+a)^3 
                       M_F(a)^3{}_kP_0(y,a)^{m+2n}{}_1 \\
    R_{X,k}^{nm}(y,a)&:=c_k\frac{1+y}{(1-y)^2}(1-y^2)^n(a_*+a)^2 
                       M_F(a)^3{}_kP_X(y,a)^{m+2n}{}_1 
  \end{align*}
  for $X\in \{C,S\}$. By construction and Remarks \ref{rem:degP0} and \ref{rem:MF}, these are all polynomials
  of degree at most $(38,8)$. We compute
  \begin{align*}
    \sum_{n=0}^1 \sup_{(0,1)\times J_*}\beta_{F,R}^n
    &\leq \tfrac{100}{84}\sum_{n=0}^1\sum_{k=1}^2\sum_{m=1}^2\sum_{X\in
      \{0,C,S\}}\|P_{X,k}^{nm}\|_{T([0,1]\times J_*)} \\
      &\quad +5\sum_{n=0}^1 \sum_{k=3}^4 \sum_{m=1}^2 \sum_{X\in
      \{0,C,S\}}\|Q_{X,k}^{nm}\|_{T([0,1]\times J_*)} \\
    &\quad +5\sum_{n=0}^1\sum_{k=3}^4 \sum_{m=1}^2 \sum_{X\in
      \{0,C,S\}}\|R_{X,k}^{nm}\|_{T([0,1]\times J_*)} \\
    &\leq 173
  \end{align*}
  and Proposition \ref{prop:quantCJ} yields the claimed estimate.
\end{proof}

\section{Estimates on the coefficients}
\label{sec:coeff}

\noindent In this section we estimate the difference of the
coefficients of the approximate linear operator to the true linear
operator. These are the estimates that will show that the
approximate linear operator is suitably close to the true one, as
outlined in Section \ref{sec:strategy}, point (\ref{itm:approxlin}).

\subsection{The left-hand side}

\begin{lemma}[{\tt Lemma\_8.1.nb}]
  \label{lem:coeffL}
  For the coefficients $(p_1,p_2,q_1,q_2)$ associated to $F$, we have
  the bounds
    \begin{align*}
    \sup_{(y,a)\in (-1,0)\times J_*}(1-y)\left |p_1(y,a)-p_0(y,a_*+a)\right |
    &\leq 3\cdot 10^{-7} \\
    \sup_{(y,a)\in (-1,0)\times J_*}(1-y)|p_2(y,a)|&\leq 2\cdot 10^{-7}
  \end{align*}
  as well as
  \begin{align*}
    \sup_{(y,a)\in (-1,0)\times J_*}(1+y)(1-y)^2\left
    |q_1(y,a)-q_0(y,a_*+a)-\frac{2|f_*(y,a)|^2}{(1-y)^2}\right |
    &\leq 7\cdot 10^{-7} \\
     \sup_{(y,a)\in (-1,0)\times J_*}(1+y)(1-y)^2\left
    |q_2(y,a)-\frac{f_*(y,a)^2}{(1-y)^2}\right |
    &\leq 2\cdot 10^{-6}.
  \end{align*}
\end{lemma}

\begin{proof}
  Let $y\in (-1,0)$. 
  By definition, we have
 \begin{align*}
   F'(y,a)F(y,a)^{-1}
   &=F_L'(y)M_F(a)M_F(a)^{-1}F_L(y)^{-1}=F_L'(y)F_L(y)^{-1}=F_L'(y)\frac{\adj(F_L(y))}{\det(F_L(y))} \\
   &=\frac{(1+y)^4F_L'(y)\adj(F_L(y))}{P_{d,L}(y)} 
     =\frac{1}{(1+y)^3}\frac{P_L(y)}{P_{d,L}(y)}
 \end{align*}
 with $P_L(y):=(1+y)^7F_L'(y)\adj(F_L(y))$. Note that $P_L$ is a matrix with
 polynomial entries of degree at most $74$.
 Consequently, for $f\in \{p_1,p_2,q_1,q_2\}$, we obtain
 \begin{align*}
   f(y,a)&=-\frac{1}{2P_{d,L}(y)}\frac{P_L^f(y)}{(1+y)^3}
 \end{align*}
 with
 \begin{align*}
   P_L^{p_1}(y)&:=P_L(y)^3{}_3+P_L(y)^4{}_4+iP_L(y)^4{}_3-iP_L(y)^3{}_4
   \\
   P_L^{p_2}(y)&:=P_L(y)^3{}_3-P_L(y)^4{}_4+iP_L(y)^4{}_3+iP_L(y)^3{}_4
   \\
   P_L^{q_1}(y)&:=P_L(y)^3{}_1+P_L(y)^4{}_2+iP_L(y)^4{}_1-iP_L(y)^3{}_2
   \\
   P_L^{q_2}(y)&:=P_L(y)^3{}_1-P_L(y)^4{}_2+iP_L(y)^4{}_1+iP_L(y)^3{}_2,
 \end{align*}
 see Definition \ref{def:coeff}. Next, we recall that
 \begin{align*}
   p_0(y,a_*+a)&=\frac{4i(a_*+a)}{(1-y)^3}-\frac{2i(a_*+a)}{(1-y)^2}-\frac{2+\frac{2i}{a_*+a}}{1-y}+\frac{2}{1+y} \\
   q_0(y,a_*+a)&=-\frac{2-2i(a_*+a)}{(1-y)^3}-\frac{\frac{1}{(a_*+a)^2}-\frac{i}{a_*+a}}{(1-y)^2}
             -\frac{1+\frac{i}{a_*+a}}{1-y}-\frac{1+\frac{i}{a_*+a}}{1+y}
 \end{align*}
 as well as
 \[ f_*(y,a)=P_*(y)+\frac{2i(a_*+a)}{1-i(a_*+a)}P_*'(1)-P_*(1), \]
 where $P_*$ is a polynomial of degree 50.
 Based on this, we obtain
 \begin{align*}
   f_1(y,a)&:=(1-y)[p_1(y,a)-p_0(y,a_*+a)]=\frac{P_L^{f_1}(y,a)}{2(a_*+a)(1-y)^2P_{d,L}(y)} \\
   f_2(y,a)&:=(1-y)p_2(y,a)=\frac{P_L^{f_2}(y)}{2P_{d,L}(y)}
 \end{align*}
 and
 \begin{align*}
   g_1(y,a):&=(1+y)(1-y)^2\left
             [q_1(y,a)-q_0(y,a_*+a)-\frac{2|f_*(y,a)|^2}{(1-y)^2}\right] \\
             &=\frac{P_L^{g_1}(y,a)}{2(a_*+a)^2|1-i(a_*+a)|^2(1-y)P_{d,L}(y)}
   \\
   g_2(y,a):&=(1+y)(1-y)^2\left
              [q_2(y,a)-\frac{f_*(y,a)^2}{(1-y)^2}\right ]=\frac{P_L^{g_2}(y,a)}{2[1-i(a_*+a)]^2P_{d,L}(y)},
 \end{align*}
 where
 \begin{align*}
   P_L^{f_1}(y,a)&:=-(a_*+a)\frac{(1-y)^3P_L^{p_1}(y)}{(1+y)^3}-2(a_*+a)(1-y)^3P_{d,L}(y)p_0(y,a_*+a) \\
   P_L^{f_2}(y)&:=-\frac{(1-y)P_L^{p_2}(y)}{(1+y)^3}
 \end{align*}
 and
 \begin{align*}
   P_L^{g_1}(y,a)&:=-(a_*+a)^2|1-i(a_*+a)|^2\Big [
                 \frac{(1-y)^3P_L^{q_1}(y)}{(1+y)^2}+2(1+y)(1-y)^3P_{d,L}(y)q_0(y,a_*+a)
   \\
               &\qquad +4(1+y)(1-y)P_{d,L}(y)|f_*(y,a)|^2
                 \Big] \\
   P_L^{g_2}(y,a)&:=-[1-i(a_*+a)]^2\left [\frac{(1-y)^2P_L^{q_2}(y)}{(1+y)^2}+2(1+y)P_{d,L}(y)f_*(y,a)^2\right].
 \end{align*}
 By explicit evaluation, we find
 \[ P_L^{p_1}(-1)=(P_L^{p_1})'(-1)=P_L^{p_2}(-1)=(P_L^{p_2})'(-1)=
   (P_L^{p_2})''(-1)=0 \]
 as well as
 \[ P_L^{q_1}(-1)=(P_L^{q_1})'(-1)=P_L^{q_2}(-1)=(P_L^{q_2})'(-1)=0 \]
and thus, by Proposition \ref{prop:F}, $P_L^f$ for $f\in \{f_1,f_2,g_1,g_2\}$ are
polynomials with
$\deg P_L^{f}\leq (174, 6)$. We compute
\begin{align*}
  \left \|P_L^{f_1}\right \|_{T([-1,0]\times J_*)}&\leq 8.8\cdot 10^{-7},
  &
    \left \|P_L^{f_2}\right \|_{T([-1,0]\times J_*)}&\leq 3.8\cdot
                                                      10^{-7}, \\
  \left \|P_L^{g_1}\right \|_{T([-1,0]\times J_*)}&\leq 4\cdot 10^{-6},
  &
    \left \|P_L^{g_2}\right \|_{T([-1,0]\times J_*)}&\leq 4.7\cdot
                                                      10^{-6}, 
\end{align*}
and recall that $\min_{[-1,0]}P_{d,L}\geq \frac{99}{100}$ (Lemma
\ref{lem:FLadm}). Furthermore,
\[ \min_{[-1,0],\frac{1}{100}}\left[y\mapsto(1-y)P_{d,L}(y)\right]\geq 2 \]
and from these bounds, the claimed estimates follow.
\end{proof}

\subsection{The right-hand side}

\begin{definition}[Polynomials for $F_R'(y,a)$, {\tt Definition\_8.2.nb}]
  For $(y,a)\in (0,1)\times J_*$, we set
   \begin{align*}
    P_{C}^{(1)}(y,a):=
    &(a_*+a)(1-y)^3\varphi'(y,a_*+a)P_{S}(y,a)
      +2(a_*+a)(1-y)^2P_{C}(y,a) \\
    & +(a_*+a)(1-y)^3P_{C}'(y,a)                         
  \end{align*}
  and
  \begin{align*}
    P_{S}^{(1)}(y,a):=&-(a_*+a)(1-y)^3\varphi'(y,a_*+a)P_{C}(y,a)
                          +2(a_*+a)(1-y)^2P_{S}(y,a) \\
    &+(a_*+a)(1-y)^3P_{S}'(y,a),
  \end{align*}
  see Definition \ref{def:PC}.
\end{definition}

\begin{remark}
  \label{rem:degPC1}
  Note that $\deg (P_X^{(1)})^j{}_k\leq (37,5)$ for $X\in \{C,S\}$ and
  $j,k\in \{1,2,3,4\}$.
\end{remark}

\begin{definition}[Polynomials for $F_R'(y,a)F_R(y,a)^{-1}$, {\tt Definition\_8.4.nb}]
  For $(y,a)\in (0,1)\times J_*$, we set
  \begin{align*}
P_{CC}(y,a)&:=P_C^{(1)}(y,a)Q_C(y,a)
                           +(a_*+a)(1-y)^3P_0'(y,a)Q_{CC}(y,a)
  \\
  P_{CS}(y,a)&:=P_{C}^{(1)}(y,a)Q_S(y,a)+P_S^{(1)}(y,a)Q_C(y,a)+(a_*+a)(1-y)^3P_0'(y,a)Q_{CS}(y,a)
  \\
  P_{SS}(y,a)&:=P_S^{(1)}(y,a)Q_S(y,a)+(a_*+a)(1-y)^3P_0'(y,a)Q_{SS}(y,a),
  \end{align*}
  see Definition \ref{def:QC}.
\end{definition}

\begin{remark}
  \label{rem:degPCC}
  Note that $\deg (P_{XY})^j{}_k\leq (142,14)$ for $XY\in \{CC,CS,SS\}$
  and $j,k\in \{1,2,3,4\}$.
\end{remark}

\begin{definition}[Polynomials for coefficients, {\tt Definition\_8.6.nb}]
  For $XY\in \{CC, CS, SS\}$, we define
  \begin{align*}
    P_{XY}^{p_1}&:=-\tfrac12 \left
    [(P_{XY})^3{}_3+(P_{XY})^4{}_4\right]
                  -\tfrac{i}{2}\left [(P_{XY})^4{}_3-(P_{XY})^3{}_4\right] \\
    P_{XY}^{p_2}&:=-\tfrac12 \left
    [(P_{XY})^3{}_3-(P_{XY})^4{}_4\right]
                  -\tfrac{i}{2}\left
                  [(P_{XY})^4{}_3+(P_{XY})^3{}_4\right]     
  \end{align*}
  as well as
  \begin{align*}
    P_{XY}^{q_1}&:=-\tfrac12 \left
    [(P_{XY})^3{}_1+(P_{XY})^4{}_2\right]
                  -\tfrac{i}{2}\left
                  [(P_{XY})^4{}_1-(P_{XY})^3{}_2\right] \\
    P_{XY}^{q_2}&:=-\tfrac12 \left
    [(P_{XY})^3{}_1-(P_{XY})^4{}_2\right]
                  -\tfrac{i}{2}\left
                  [(P_{XY})^4{}_1+(P_{XY})^3{}_2\right].
  \end{align*}
\end{definition}

\begin{lemma}[{\tt Lemma\_8.7.nb}]
  \label{lem:polcoeff}
  For $j\in \{1,2\}$ we have the bounds
  \begin{align*}
    \sup_{(y,a)\in [0,1)\times J_*}\left
    |\frac{P_{CC}^{p_j}(y,a)-P_{SS}^{p_j}(y,a)}{(1-y)^2}\right |
    &\leq 10^{-16},
    &
   \sup_{(y,a)\in [0,1)\times J_*}\left
    |\frac{P_{CS}^{p_j}(y,a)}{(1-y)^2}\right |&\leq 10^{-16},
  \end{align*}
  and
   \begin{align*}
    \sup_{(y,a)\in [0,1)\times J_*}\left
    |\frac{P_{CC}^{q_j}(y,a)-P_{SS}^{q_j}(y,a)}{1-y}\right |
    &\leq 10^{-15},
    &
   \sup_{(y,a)\in [0,1)\times J_*}\left
    |\frac{P_{CS}^{q_j}(y,a)}{1-y}\right |&\leq 10^{-15}.
  \end{align*}
\end{lemma}

\begin{proof}
  By explicit evaluation we check that
  \[ P_{CC}^f(1,a)-P_{SS}^f(1,a)=(P_{CC}^f)'(1,a)-(P_{SS}^f)'(1,a)=P_{CS}^f(1,a)=(P_{CS}^f)'(1,a)=0 \]
  for $f\in \{p_1,p_2,q_1,q_2\}$, $j\in \{1,2\}$, and $a\in
  \{0,1,2,\dots,14\}$. Thus, since $\deg_2P_{XY}^f\leq 14$ for $XY\in
  \{CC,SS,CS\}$, we see that
  \[ (y,a)\mapsto
    \frac{P_{CC}^{p_j}(y,a)-P_{SS}^{p_j}(y,a)}{(1-y)^2},\qquad
    (y,a)\mapsto \frac{P_{CS}^{p_j}(y,a)}{(1-y)^2}\]
  and
  \[ (y,a)\mapsto
    \frac{P_{CC}^{q_j}(y,a)-P_{SS}^{q_j}(y,a)}{1-y},\qquad
    (y,a)\mapsto \frac{P_{CS}^{q_j}(y,a)}{1-y}\]
  are polynomials and by computing their $T$-norms on $[0,1]\times J_*$, the claim follows.
\end{proof}

\begin{lemma}
  \label{lem:coeffR}
  Let $(p_1,p_2,q_1,q_2)$ be the coefficients associated to $F$. Then,
  for $(y,a)\in (0,1)\times J_*$ and $f\in \{p_1,p_2,q_1,q_2\}$,
  we have
  \begin{align*}
    f(y,a)&=\frac{\cos(\varphi(y,a_*+a))^2}{(a_*+a)(1-y)^3}\frac{P^f_{CC}(y,a)}{P_{d,R}(y,a)+\varepsilon_{d,R}(y,a)} \\
    &\quad +\frac{\sin(\varphi(y,a_*+a))^2}{(a_*+a)(1-y)^3}\frac{P^f_{SS}(y,a)}{P_{d,R}(y,a)+\varepsilon_{d,R}(y,a)} \\
  &\quad +\frac{\cos(\varphi(y,a_*+a))\sin(\varphi(y,a_*+a))}{(a_*+a)(1-y)^3}\frac{P^f_{CS}(y,a)}{P_{d,R}(y,a)+\varepsilon_{d,R}(y,a)}.
  \end{align*}
\end{lemma}

\begin{proof}
  Recall the decomposition
   \[
    F_R(y,a)=\frac{1}{a_*+a}P_{0}(y,a)+\frac{\cos(\varphi(y,a_*+a))}{(a_*+a)^2(1-y)^2}P_{C}(y,a)
    +\frac{\sin(\varphi(y,a_*+a))}{(a_*+a)^2(1-y)^2}P_{S}(y,a) \]
  from Lemma \ref{lem:FRdecomp}, which immediately yields
  \[
    F_R'(y,a)=\frac{1}{a_*+a}P_{0}'(y,a)+\frac{\cos(\varphi(y,a_*+a))}{(a_*+a)^3(1-y)^5}P_{C}^{(1)}(y,a)
    +\frac{\sin(\varphi(y,a_*+a))}{(a_*+a)^3(1-y)^5}P_{S}^{(1)}(y,a). \]
  Next, recall that
  \[
    (Q_{CC})^{j+2}{}_k=(Q_{CS})^{j+2}{}_k=(Q_{SS})^{j+2}{}_k=(Q_C)^j{}_k=(Q_S)^j{}_k=(P_0)^k{}_{j+2}=(P_C)^k{}_{j}=(P_S)^k{}_{j}=0 \]
  for $j\in \{1,2\}$ and $k\in \{1,2,3,4\}$, which implies that
  \[ P_0' Q_C=P_0'
    Q_S=P_C^{(1)}Q_{CC}=P_C^{(1)}Q_{CS}=P_C^{(1)}Q_{SS}=P_S^{(1)}Q_{CC}=P_S^{(1)}Q_{CS}=P_S^{(1)}Q_{SS}=0. \]
Consequently, in view of the decomposition from Lemma
\ref{lem:adjFRdecomp}, we obtain
\begin{align*}
  F_R'(y,a)\adj(F_R(y,a))
  &=\frac{\cos(\varphi(y,a_*+a))^2}{(a_*+a)^7(1-y)^7}P_{CC}(y,a)
    +\frac{\sin(\varphi(y,a_*+a))^2}{(a_*+a)^7(1-y)^7}P_{SS}(y,a) \\
  &\quad +\frac{\cos(\varphi(y,a_*+a))\sin(\varphi(y,a_*+a))}{(a_*+a)^7(1-y)^7}P_{CS}(y,a),
\end{align*}
which, by Lemma \ref{lem:FRadm}, yields
  \begin{align*}
    F_R'(y,a)F_R(y,a)^{-1}
    &=F_R'(y,a)\frac{\adj(F_R(y,a))}{\det(F_R(y,a))}
                            =F_R'(y,a)\frac{(a_*+a)^6(1-y)^4\adj(F_R(y,a))}{P_{d,R}(y,a)+\varepsilon_{d,R}(y,a)}
    \\
    &=\frac{\cos(\varphi(y,a_*+a))^2}{(a_*+a)(1-y)^3}\frac{P_{CC}(y,a)}{P_{d,R}(y,a)+\varepsilon_{d,R}(y,a)} \\
    &\quad +\frac{\sin(\varphi(y,a_*+a))^2}{(a_*+a)(1-y)^3}\frac{P_{SS}(y,a)}{P_{d,R}(y,a)+\varepsilon_{d,R}(y,a)} \\
  &\quad +\frac{\cos(\varphi(y,a_*+a))\sin(\varphi(y,a_*+a))}{(a_*+a)(1-y)^3}\frac{P_{CS}(y,a)}{P_{d,R}(y,a)+\varepsilon_{d,R}(y,a)}.
  \end{align*}
  Thus, the claim follows by recalling Definition \ref{def:coeff}. 
\end{proof}

\begin{lemma}[{\tt Lemma\_8.9.nb}]
  \label{lem:estcoeffR}
  Let $(p_1,p_2,q_1,q_2)$ be the coefficients associated to $F$. Then
  we have the bounds
  \begin{align*}
    \sup_{(y,a)\in (0,1)\times J_*}(1-y)\left |p_1(y,a)-p_0(y,a_*+a)\right |
    &\leq
      3\cdot 10^{-7} \\
    \sup_{(y,a)\in (0,1)\times J_*}(1-y)|p_2(y,a)|&\leq 6\cdot 10^{-7}
  \end{align*}
  as well as
  \begin{align*}
    \sup_{(y,a)\in (0,1)\times J_*}(1+y)(1-y)^2\left
    |q_1(y,a)-q_0(y,a_*+a)-\frac{2|f_*(y,a)|^2}{(1-y)^2}\right |
    &\leq 1.5\cdot 10^{-6} \\
     \sup_{(y,a)\in (0,1)\times J_*}(1+y)(1-y)^2\left
    |q_2(y,a)-\frac{f_*(y,a)^2}{(1-y)^2}\right |
    &\leq 1.3\cdot 10^{-6}.
  \end{align*}
\end{lemma}

\begin{proof}
  We define
  \[ \widetilde f(y,a):=\frac{P_{CC}^f(y,a)}{(a_*+a)(1-y)^3P_{d,R}(y,a)} \]
  for $f\in \{p_1,p_2,q_1,q_2\}$. By Lemmas \ref{lem:FRadm},
  \ref{lem:polcoeff}, \ref{lem:coeffR}, and
  \[ \|P_{CC}^{p_j}\|_{T([0,1]\times J_*)}\leq 10^3, \qquad
  \|P_{CC}^{q_j}\|_{T([0,1]\times J_*)}\leq 10^4 \]
  for $j\in\{1,2\}$, we infer that
  \[
    \sup_{(y,a)\in (0,1)\times J_*}(1-y)\left |p_j(y,a)-\widetilde
    p_j(y,a)\right |\leq 10^{-15} 
    \]
  and
  \[ \sup_{(y,a)\in (0,1)\times J_*}(1+y)(1-y)^2\left
      |q_j(y,a)-\widetilde q_j(y,a)\right|\leq 10^{-14}. \]
  We set
  \begin{align*}
    f_1(y,a)&:=(1-y)\left [\widetilde p_1(y,a)-p_0(y,a_*+a)\right ]=:\frac{P_R^{f_1}(y,a)}{(a_*+a)(1+y)P_{d,R}(y,a)} \\
    f_2(y,a)&:=(1-y)\widetilde p_2(y,a)=:\frac{P_R^{f_2}(y,a)}{(a_*+a)P_{d,R}(y,a)}
  \end{align*}
  as well as
  \begin{align*}
    g_1(y,a)&:=(1+y)(1-y)^2\left
    [\widetilde
              q_1(y,a)-q_0(y,a_*+a)-\frac{2|f_*(y,a)|^2}{(1-y)^2}\right ] \\
    &=:\frac{P_R^{g_1}(y,a)}{(a_*+a)^2|1-i(a_*+a)|^2P_{d,R}(y,a)} \\
  g_2(y,a)&:=(1+y)(1-y)^2\left
    [ \widetilde q_2(y,a)-\frac{f_*(y,a)^2}{(1-y)^2}\right ]=:\frac{P_R^{g_2}(y,a)}{(a_*+a)|1-i(a_*+a)|^4P_{d,R}(y,a)}
  \end{align*}
  and Proposition \ref{prop:F} implies that $P_R^f$ for $f\in
  \{f_1,f_2,g_1,g_2\}$ are polynomials of degree at most $(245,18)$.
Recall from Lemma \ref{lem:FRadm} that $\min_{[0,1]\times
  J_*}P_{d,R}\geq 8$. Consequently, with
$\Omega_j:=[\frac{j}{10},\frac{j+1}{10}]\times J_*$, we compute
\begin{align*}
  \|\Re f_1\|_{L^\infty((0,1)\times J_*)}&\leq \max\left
      \{\frac{\|\Re P_R^{f_1}\|_{T(\Omega_j)}}
      {\min_{\Omega_j,1}[(y,a)\mapsto
                                           (a_*+a)(1+y)P_{d,R}(y,a)]-1}: j\in [0,9]\cap\Z\right \} \\
  &\leq 2\cdot 10^{-7}
  \end{align*}
  and
  \begin{align*}
  \|\Im f_1\|_{L^\infty((0,1)\times J_*)}&\leq \max\left
      \{\frac{\|\Im P_R^{f_1}\|_{T(\Omega_j)}}
      {\min_{\Omega_j,1}[(y,a)\mapsto
                                           (a_*+a)(1+y)P_{d,R}(y,a)]-1}: j\in [0,9]\cap\Z\right \} \\
  &\leq 2\cdot 10^{-7}
  \end{align*}
  and
  \begin{align*}
  \|\Re f_2\|_{L^\infty((0,1)\times J_*)}&\leq \max\left
      \{\frac{\|\Re P_R^{f_2}\|_{T(\Omega_j)}}
      {\min_{\Omega_j,1}[(y,a)\mapsto
                                           (a_*+a)P_{d,R}(y,a)]-1}: j\in [0,9]\cap\Z\right \} \\
                                         &\leq 2\cdot 10^{-7}
  \end{align*}
  and
  \begin{align*}
  \|\Im f_2\|_{L^\infty((0,1)\times J_*)}&\leq \max\left
      \{\frac{\|\Im P_R^{f_2}\|_{T(\Omega_j)}}
      {\min_{\Omega_j,1}[(y,a)\mapsto
                                           (a_*+a)P_{d,R}(y,a)]-1}: j\in [0,9]\cap\Z\right \} \\
  &\leq 5\cdot 10^{-7}
  \end{align*}
and
\begin{align*}
  \|\Re g_1&\|_{L^\infty((0,1)\times J_*)} \\
  &\leq \max\left
      \{\frac{\|\Re P_R^{g_1}\|_{T(\Omega_j)}}
      {\min_{\Omega_j,1}[(y,a)\mapsto
                                           (a_*+a)^2|1-i(a_*+a)|^2P_{d,R}(y,a)]-1}: j\in [0,9]\cap\Z\right \} \\
  &\leq 10^{-6}
  \end{align*}
  and
  \begin{align*}
    \|\Im g_1&\|_{L^\infty((0,1)\times J_*)} \\
    &\leq \max\left
      \{\frac{\|\Im P_R^{g_1}\|_{T(\Omega_j)}}
      {\min_{\Omega_j,1}[(y,a)\mapsto
                                           (a_*+a)^2|1-i(a_*+a)|^2P_{d,R}(y,a)]-1}: j\in [0,9]\cap\Z\right \} \\
  &\leq 5\cdot 10^{-7}
  \end{align*}
and
\begin{align*}
  \|\Re g_2&\|_{L^\infty((0,1)\times J_*)} \\
  &\leq \max\left
      \{\frac{\|\Re P_R^{g_2}\|_{T(\Omega_j)}}
      {\min_{\Omega_j,1}[(y,a)\mapsto
                                           (a_*+a)|1-i(a_*+a)|^4P_{d,R}(y,a)]-1}: j\in [0,9]\cap\Z\right \} \\
  &\leq 8\cdot 10^{-7}
  \end{align*}
  and
  \begin{align*}
    \|\Im g_2&\|_{L^\infty((0,1)\times J_*)}\\
    &\leq \max\left
      \{\frac{\|\Im P_R^{g_2}\|_{T(\Omega_j)}}
      {\min_{\Omega_j,1}[(y,a)\mapsto
                                           (a_*+a)|1-i(a_*+a)|^4P_{d,R}(y,a)]-1}: j\in [0,9]\cap\Z\right \} \\
  &\leq 10^{-6}.
  \end{align*}
 These estimates and the triangle inequality imply the claim.
\end{proof}

\section{Estimates on the functional $\psi_F$}
\label{sec:psi}

\noindent In this section we provide the necessary estimates on the
functional $\psi_F$ that are used to detect the sign change as
explained in Section \ref{sec:strategy}, point (\ref{itm:a}).

Unfortunately, the functional $\psi_F$ involves highly oscillatory
integrals which require special care.
As a consequence, we need to split the domain of integration and to this
end, it is convenient to introduce a more efficient notation.

\begin{definition}[{\tt Definition\_9.1.nb}]
  For $y_0,y_1\in [-1,1]$ and $a\in J_*$, 
  we define
  \begin{align*} I_F(a, y_0,y_1):=\sum_{k=3}^4\int_{y_0}^{y_1} \gamma_{F,k}(a)\alpha_F^k\left(a, \mc R\left (a_*+a,
    f_*(\cdot,a)\right)\right),
  \end{align*}
  where
  \[
    \gamma_{F,k}(a):=\frac{M_F(a)^4{}_1}{M_F(a)^3{}_1}M_F(a)^3{}_k-M_F(a)^4{}_k. \]
\end{definition}

  Note that
  \[ \psi_F\left (a, \mc R\left (a_*+a,
        f_*(\cdot,a)\right)\right)=I_F(a,-1,1). \]

  \begin{definition}[Polynomials for $\psi_F$, {\tt Definition\_9.2.nb}]
    For $y\in [0,1)\times J_*$, we set
    \begin{align*}
      P^\psi_C(y,a)&:=(a_*+a)^2(1+y)(1-y)^2\sum_{k=3}^4 \gamma_{F,k}(a)\left
                     [Q_C(y,a)^k{}_3\Re r(y,a)+Q_C(y,a)^k{}_4\Im
                     r(y,a)\right] \\
            P^\psi_S(y,a)&:=(a_*+a)^2(1+y)(1-y)^2\sum_{k=3}^4 \gamma_{F,k}(a)\left
                     [Q_S(y,a)^k{}_3\Re r(y,a)+Q_S(y,a)^k{}_4\Im
                     r(y,a)\right],
    \end{align*}
    where $r(y,a):=\mc R(a_*+a,
        f_*(\cdot,a))(y)$.
      \end{definition}

      \begin{remark}
        \label{rem:degPXpsi}
        Note that by definition,
        $P^\psi_X$ for $X\in \{C,S\}$ is a polynomial in the first
        variable and by Remark \ref{rem:degQC}, we have $\deg_1
        P_X^\psi\leq 256$.
      \end{remark}

 \begin{definition}
   For $y_0,y_1\in [0,1]$ and $a\in J_*$, we set
   \begin{align*}
     \widetilde I_F(a,y_0,y_1):=&\int_{y_0}^{y_1}
     \cos(\varphi(y,a_*+a))\frac{P^\psi_C(y,a)}{(1+y)P_{d,R}(y,a)}dy \\
     &+\int_{y_0}^{y_1}
       \sin(\varphi(y,a_*+a))\frac{P^\psi_S(y,a)}{(1+y)P_{d,R}(y,a)}dy.
   \end{align*}
   Furthermore, for brevity we set $a_\pm:=\pm 10^{-10}$.
 \end{definition}
 
 \begin{lemma}[{\tt Lemma\_9.5.nb}]
   \label{lem:Itw}
   Let $0\leq y_0\leq y_1\leq 1$. Then we have the bounds
   \[ \left |I_F(a_\pm, y_0, y_1)-\widetilde I_F(a_\pm, y_0, y_1)\right|\leq
     10^{-18}. \]
 \end{lemma}

 \begin{proof}
   From Lemma \ref{lem:adjFRdecomp} and Definition \ref{def:QC}, we obtain
   \[ \adj(F_R(y,a))^j{}_k=\frac{\cos(\varphi(y,a_*+a))}{(a_*+a)^4(1-y)^2}Q_{C}(y,a)^j{}_k
    +\frac{\sin(\varphi(y,a_*+a))}{(a_*+a)^4(1-y)^2}Q_{S}(y,a)^j{}_k \]
  for $j\in \{3,4\}$ and $k\in \{1,2,3,4\}$. Consequently, by Lemma \ref{lem:FRadm},
  \begin{align*} [F_R(y,a)^{-1}]^j{}_k
    &=\frac{\adj(F_R(y,a))^j{}_k}{\det(F_R(y,a))} \\
    &=
      \cos(\varphi(y,a_*+a))\frac{(a_*+a)^2(1-y)^2Q_C(y,a)^j{}_k}{P_{d,R}(y,a)+\varepsilon_{d,R}(y,a)} \\
    &\quad +
      \sin(\varphi(y,a_*+a))\frac{(a_*+a)^2(1-y)^2Q_S(y,a)^j{}_k}{P_{d,R}(y,a)+\varepsilon_{d,R}(y,a)}
  \end{align*}
  and by recalling that
  \[ \alpha_F^k(a,g)(y)=[F(y,a)^{-1}]^k{}_3 \Re
    g(y)+[F(y,a)^{-1}]^k{}_4 \Im g(y), \]
  we find
     \begin{align*} I_F\left (a, y_0, y_1\right)&=\int_{y_0}^{y_1}
     \cos(\varphi(y,a_*+a))\frac{P_C^\psi(y,a)}{(1+y)[P_{d,R}(y,a)+\varepsilon_{d,R}(y,a)]}dy \\
     &\quad +\int_{y_0}^{y_1} \sin(\varphi(y,a_*+a))\frac{P_S^\psi(y,a)}{(1+y)[P_{d,R}(y,a)+\varepsilon_{d,R}(y,a)]}dy.
   \end{align*}
   We compute
   \begin{align*}
     \|P_C^\psi(\cdot,a_\pm)\|_{T([0,1])}+\|P_S^\psi(\cdot,a_\pm)\|_{T([0,1])}\leq
     10^{-6}.
   \end{align*}
      Thanks to the bounds $\|\varepsilon_{d,R}\|_{L^\infty([0,1]\times J_*)}\leq
   10^{-17}$ and $\min_{[0,1]\times J_*}P_{d,R}\geq 8$ from
   Lemma \ref{lem:FRadm}, we see that
   \[ \left |I_F(a_\pm, y_0, y_1)-\widetilde I_F(a_\pm, y_0,y_1)\right|\leq
     10^{-18}, \]
   as claimed.
 \end{proof}
 
\subsection{Close to the right endpoint}
  Next, we show that the contribution near the irregular singularity at
$y=1$ is negligible.
  
  \begin{lemma}[{\tt Lemma\_9.6.nb}]
    \label{lem:psi1}
    We have the bound
    \[ \left |I_F\left (a_\pm,\tfrac{9}{10},1\right)\right |\leq 10^{-11}. \]
  \end{lemma}

  \begin{proof}
   By Lemma \ref{lem:Itw}, it suffices to consider $\widetilde I_F$.
   In order to exploit the oscillations, we write
   \begin{align*}
     \widetilde I_F(a,\tfrac{9}{10},1)&=\int_{\frac{9}{10}}^1
           \partial_y\sin(\varphi(y,a_*+a))\frac{P_C^\psi(y,a)}{(1+y)P_{d,R}(y,a)\varphi'(y,a_*+a)}dy \\
     &\quad -\int_{\frac{9}{10}}^1
       \partial_y\cos(\varphi(y,a_*+a))\frac{P_S^\psi(y,a)}{(1+y)P_{d,R}(y,a)\varphi'(y,a_*+a)}dy
   \end{align*}
   and integrate by parts, which yields
   \begin{align*}
           \widetilde I_F(a,\tfrac{9}{10},1)&=-\frac{\sin(\varphi(\frac{9}{10},a_*+a))P_C^\psi(\frac{9}{10},a)}{\frac{19}{10}P_{d,R}(\frac{9}{10},a)\varphi'(\frac{9}{10},a_*+a)}
           +\frac{\cos(\varphi(\frac{9}{10},a_*+a))P_S^\psi(\frac{9}{10},a)}{\frac{19}{10}P_{d,R}(\frac{9}{10},a)\varphi'(\frac{9}{10},a_*+a)}
     \\
     &\quad -\int_{\frac{9}{10}}^1
           \sin(\varphi(y,a_*+a))\partial_y\left
                 [\frac{P_C^\psi(y,a)}{(1+y)P_{d,R}(y,a)\varphi'(y,a_*+a)}\right ]dy \\
     &\quad 
           +\int_{\frac{9}{10}}^1
           \cos(\varphi(y,a_*+a))\partial_y\left
                 [\frac{P_S^\psi(y,a)}{(1+y)P_{d,R}(y,a)\varphi'(y,a_*+a)}\right ]dy
   \end{align*}
   since
   \[ 
     \frac{1}{\varphi'(y,a_*+a)}=-\frac{(a_*+a)(1-y)^3}{2(a_*+a)^2(1+y)-2(1-y)^2}. \]
   As a consequence, we obtain the bound
   \begin{equation}
     \label{eq:bdI}
     \begin{split}
     |\widetilde I_F(a,\tfrac{9}{10},1)|&\leq
     \frac{10}{19}\frac{|P_C^\psi(\frac{9}{10},a)|+|P_S^\psi(\frac{9}{10},a)|}{|P_{d,R}(\frac{9}{10},a)\varphi'(\frac{9}{10},a_*+a)|}
     +\frac{1}{10}\max_{y\in [\frac{9}{10},1]}\left |\partial_y
                 \frac{P_C^\psi(y,a)}{(1+y)P_{d,R}(y,a)\varphi'(y,a_*+a)}
             \right | \\
     &\quad +\frac{1}{10}\max_{y\in [\frac{9}{10},1]}\left |\partial_y
                 \frac{P_S^\psi(y,a)}{(1+y)P_{d,R}(y,a)\varphi'(y,a_*+a)}
       \right |.
       \end{split}
   \end{equation}
   We have
   \begin{align*}
     \partial_y\frac{P_C^\psi(y,a)}{(1+y)P_{d,R}(y,a)\varphi'(y,a_*+a)}
     &=\partial_y\frac{\widetilde
       P_C^\psi(y,a)}{\widetilde
       P_{d,R}(y,a)}=\frac{\partial_y\widetilde P_C^\psi(y,a)\widetilde
       P_{d,R}(y,a)-\widetilde P_C^\psi(y,a)\partial_y \widetilde
       P_{d,R}(y,a)}{\widetilde P_{d,R}(y,a)^2}
   \end{align*}
   with $\widetilde P_C^\psi(y,a):=-(a_*+a)(1-y)^3P_C^\psi(y,a)$ and
   \[ \widetilde
     P_{d,R}(y,a):=\left [2(a_*+a)^2(1+y)-2(1-y)^2\right
     ](1+y)P_{d,R}(y,a). \]
   Analogously, we set $\widetilde P_S^\psi(y,a):=-(a_*+a)(1-y)^3P_S^\psi(y,a)$.
   By Remarks \ref{rem:degPdR} and \ref{rem:degPXpsi}, we have $\deg_1
   \widetilde P_{d,R}\leq 143$ and $\deg_1\widetilde
   P_X^\psi\leq 259$ for $X\in \{C,S\}$. Thus, we compute
   \[ \min_{[\frac{9}{10},1],1}\widetilde P_{d,R}(\cdot,
     a_\pm)\geq 54
   \]
and
   \begin{align*} \left \|\widetilde P_C^\psi\left (\cdot,a_\pm\right)'\widetilde
       P_{d,R}\left(\cdot,a_\pm\right)-
     \widetilde P_C^\psi\left (\cdot,a_\pm\right)\widetilde
       P_{d,R}\left(\cdot,a_\pm\right)'\right
     \|_{T([\frac{9}{10},1])}&\leq 7\cdot 10^{-8} \\
     \left \|\widetilde P_S^\psi\left (\cdot,a_\pm\right)'\widetilde
       P_{d,R}\left(\cdot,a_\pm\right)-
     \widetilde P_S^\psi\left (\cdot,a_\pm\right)\widetilde
       P_{d,R}\left(\cdot,a_\pm\right)'\right
     \|_{T([\frac{9}{10},1])}&\leq 8\cdot 10^{-8} .
   \end{align*}
These bounds imply
   \begin{align*}
     \max_{y\in [\frac{9}{10},1]}&\left |\partial_y
                 \frac{P_C^\psi(y,a_\pm)}{(1+y)P_{d,R}(y,a_\pm)\varphi'(y,a_*+a_\pm)}
     \right |+
     \max_{y\in [\frac{9}{10},1]}\left |\partial_y
                 \frac{P_S^\psi(y,a_\pm)}{(1+y)P_{d,R}(y,a)\varphi'(y,a_*+a)}
             \right | \\
     &\leq \frac{1.5\cdot 10^{-7}}{53^2}\leq 6\cdot 10^{-11}.
   \end{align*}
   Finally, by explicit evaluation, we find
   \[
     \frac{10}{19}\frac{|P_C^\psi(\frac{9}{10},a_\pm)|+|P_S^\psi(\frac{9}{10},a_\pm)|}
     {|P_{d,R}(\frac{9}{10},a_\pm)\varphi'(\frac{9}{10},a_*+a_\pm)|}\leq
     10^{-12} \]
   and
 the claim follows from Eq.~\eqref{eq:bdI}.
  \end{proof}

\subsection{The highly oscillatory regime}

Next, we consider $I_F(a_\pm, \frac{7}{10}, \frac{9}{10})$. This is the
most delicate regime with a high number of oscillations that need to
be resolved.
The basis for the treatment is
provided by the following general observation.

\begin{lemma}
  \label{lem:hiosc}
  Let $a<b$ and $f\in C([a,b])$. Furthermore, let $g\in C^1([a,b])$ be
  real-valued with $|g'|>0$ on $[a,b]$ and
  let $p: \C\to\C$ be a polynomial.
  Then we have the bound
  \begin{align*} &\left |\int_a^b e^{ig(x)}f(x)dx-e^{ig(a)}\sum_{k=0}^{\deg
        p}i^{k+1}p^{(k)}(g(a))+e^{ig(b)}\sum_{k=0}^{\deg p}i^{k+1}p^{(k)}(g(b))\right
                   |
    \leq \left \|f-(p\circ g)g'\right\|_{L^1(a,b)}.
  \end{align*}
\end{lemma}

\begin{proof}
  Since $|g'|>0$ on $[a,b]$, $g$ is invertible and by a change of
  variables, we find
  \[ \int_a^b
    e^{ig(x)}f(x)dx=\int_{g(a)}^{g(b)}e^{iy}\frac{f(g^{-1}(y))}{g'(g^{-1}(y))}dy
  =\int_{g(a)}^{g(b)}e^{iy}p(y)dy+\int_{g(a)}^{g(b)}e^{iy}\left
    [\frac{f(g^{-1}(y))}{g'(g^{-1}(y))}-p(y)\right]dy. \]
By repeated integration by parts, we obtain
\[ \int_{g(a)}^{g(b)}e^{iy}p(y)dy=\sum_{k=0}^{\deg
    p}i^{k+1}\left[e^{ig(a)}p^{(k)}(g(a))-e^{ig(b)}p^{(k)}(g(b))\right] \]
and the claim follows from the estimate
\[ \left |\int_{g(a)}^{g(b)}e^{iy}\left
      [\frac{f(g^{-1}(y))}{g'(g^{-1}(y))}-p(y)\right]dy\right|
=\left |\int_a^b e^{ig(x)}\left [f(x)-p(g(x))g'(x)\right]dx
\right|\leq
\|f-(p\circ g)g'\|_{L^1(a,b)}.\]
\end{proof}

\begin{remark}
Lemma \ref{lem:hiosc} provides a means to compute oscillatory integrals to
a high degree of precision provided one is able to find a polynomial
$p$ such that $\|f-(p\circ g)g'\|_{L^1(a,b)}$ is sufficiently
small. The point is that the function $g$ is not oscillatory and thus,
if $f$ is well-behaved, there is no need to approximate oscillatory
functions as would be the case if one applied a standard quadrature method.
\end{remark}

In order to implement this idea for our problem, we first need a
polynomial approximation to the phase $\varphi$.

\begin{definition}[{\tt Definition\_9.9.nb}]
  \label{def:phitw}
  For $y\in [0,\frac{9}{10}]$, we set
  \[ \widetilde\varphi(y):=\sum_{n=0}^{40}
    c_n(\widetilde\varphi)T_n\left(\tfrac{10}{9}\left(2y-\tfrac{9}{10}\right)\right), \]
  where the coefficients $(c_n(\widetilde\varphi))_{n=0}^{40}\subset
  \Q$ are given in Appendix \ref{apx:phitw}.
\end{definition}

\begin{lemma}[{\tt Lemma\_9.10.nb}]
  \label{lem:phitw}
  We have the bounds
  \[ \left
      \|\varphi(\cdot,a_*+a_\pm)-\widetilde\varphi\right\|_{L^\infty(0,\frac{9}{10})}\leq
    10^{-7}. \]
\end{lemma}

\begin{proof}
  From the summation formula of the geometric series we have the identity
  \[ \log(1-y)+\sum_{k=1}^n \frac{y^k}{k}=-\int_0^y
    \frac{x^n}{1-x}dx, \]
  which yields the bound
  \[ \left |\log(1-y)+\sum_{k=1}^n \frac{y^k}{k} \right|\leq
    10\int_0^{\frac{9}{10}}x^n dx=\frac{10}{n+1}\left
      (\frac{9}{10}\right)^{n+1} \]
  for all $y\in [0,\frac{9}{10}]$ and $n\in \N$.
  Consequently, by setting
  \[
    \widehat\varphi_\pm(y):=-\frac{2(a_*+a_\pm)}{(1-y)^2}+\frac{2(a_*+a_\pm)}{1-y}
    +\frac{2}{a_*+a_\pm}\sum_{k=1}^{200}\frac{y^k}{k}, \]
  we obtain
  \[
    \|\varphi(\cdot,a_*+a_\pm)-\widehat\varphi_\pm\|_{L^\infty(0,\frac{9}{10})}\leq
    7\cdot 10^{-11}. \]
  Furthermore,
  \[
    \widehat\varphi_\pm(y)-\widetilde\varphi(y)=\frac{P_\pm(y)}{(a_*+a_\pm)(1-y)^2}, \]
  with
 \[
   P_\pm(y):=-2(a_*+a_\pm)^2+2(a_*+a_\pm)^2(1-y)+2(1-y)^2\sum_{k=1}^{200}\frac{y^k}{k}-(a_*+a_\pm)(1-y)^2
   \widetilde\varphi(y) \]
 and we compute
 \[
   \max_{[0,\frac{9}{10}],5\cdot 10^{-12}}\left
     |y\mapsto \frac{P_\pm(y)}{(a_*+a_\pm)(1-y)^2}\right|\leq 3\cdot 10^{-8},
 \]
 which yields the claim.
\end{proof}

Next, we define the polynomials that will be used to approximate the
oscillatory integrals.

\begin{definition}[{\tt Definition\_9.11.nb}]
  \label{def:sharps}
  For $x\in [-161, -10]$, we set
  \begin{align*}
    P_\pm^\sharp(x)&:=10^{-10}\sum_{n=0}^{40}c_n(P_\pm^\sharp)T_n\left(\frac{2x+171}{151}\right) \\
    Q_\pm^\sharp(x)&:=10^{-10}\sum_{n=0}^{40}c_n(Q_\pm^\sharp)T_n\left(\frac{2x+171}{151}\right),
  \end{align*}
  where the coefficients $(c_n(P_\pm^\sharp))_{n=0}^{40}\subset \Q$ and
  $(c_n(Q_\pm^\sharp))_{n=0}^{40}\subset \Q$ are given in Appendix \ref{apx:sharps}.
\end{definition}

\begin{lemma}[{\tt Lemma\_9.12.nb}]
  \label{lem:sharps}
  We have the bounds
  \begin{align*} \sup_{y\in(\frac{7}{10},\frac{9}{10})}\left
    |\frac{P_C^\psi(y,a_\pm)}{(1+y)P_{d,R}(y,a_\pm)}-P_\pm^\sharp\left(\widetilde\varphi(y)\right)
    \widetilde\varphi'(y)\right|&\leq
                                                                                                                                                                               10^{-12} \\
     \sup_{y\in(\frac{7}{10},\frac{9}{10})}\left
    |\frac{P_S^\psi(y,a_\pm)}{(1+y)P_{d,R}(y,a_\pm)}-Q_\pm^\sharp\left(\widetilde\varphi(y)\right)
    \widetilde\varphi'(y)\right|&\leq
    10^{-12}.
    \end{align*}
  \end{lemma}

  \begin{proof}
    We define
    \begin{align*}
      P_\pm(y)&:=P_C^\psi(y,a_\pm)-(1+y)P_{d,R}(y,a_\pm)
      P_\pm^\sharp\left(\widetilde\varphi(y)\right)\widetilde\varphi'(y)
      \\
      Q_\pm(y)&:=P_S^\psi(y,a_\pm)-(1+y)P_{d,R}(y,a_\pm)
      Q_\pm^\sharp\left(\widetilde\varphi(y)\right)\widetilde\varphi'(y)
    \end{align*}
    and note that by Remark \ref{rem:degPdR}, $\deg X_\pm\leq
    1780$ for $X\in \{P,Q\}$. We compute
    \[ \|X_\pm\|_{T([\frac{7}{10},\frac{9}{10}])}\leq 3\cdot 10^{-12} \]
    and together with $\min_{[0,1]\times J_*}P_{d,R}\geq 8$ from Lemma
    \ref{lem:FRadm}, the claimed bounds follow.
  \end{proof}
  
  \begin{lemma}[{\tt Lemma\_9.13.nb}]
    \label{lem:psi9/10}
    We have the bounds
    \begin{align*}
      I_F(a_-,\tfrac{7}{10},\tfrac{9}{10})&\geq 2\cdot 10^{-11} \\
      I_F(a_+,\tfrac{7}{10},\tfrac{9}{10})&\leq -2\cdot 10^{-11}.
      \end{align*}
    \end{lemma}

    \begin{proof}
      We explicitly compute
      \begin{align*}
        \Re\left
        [e^{i\widetilde\varphi(\frac{7}{10})}
        \sum_{k=0}^{40}i^{k+1}(P_-^\sharp)^{(k)}(\widetilde\varphi(\tfrac{7}{10}))
        -e^{i\widetilde\varphi(\frac{9}{10})}
        \sum_{k=0}^{40}i^{k+1}(P_-^\sharp)^{(k)}(\widetilde\varphi(\tfrac{9}{10}))\right
        ]&\geq 1.5\cdot 10^{-11}\\
        \Im\left
        [e^{i\widetilde\varphi(\frac{7}{10})}
        \sum_{k=0}^{40}i^{k+1}(Q_-^\sharp)^{(k)}(\widetilde\varphi(\tfrac{7}{10}))
        -e^{i\widetilde\varphi(\frac{9}{10})}
        \sum_{k=0}^{40}i^{k+1}(Q_-^\sharp)^{(k)}(\widetilde\varphi(\tfrac{9}{10}))\right
             ]&\geq 8\cdot 10^{-12}
      \end{align*}
      as well as
       \begin{align*}
        \Re\left
        [e^{i\widetilde\varphi(\frac{7}{10})}
        \sum_{k=0}^{40}i^{k+1}(P_+^\sharp)^{(k)}(\widetilde\varphi(\tfrac{7}{10}))
        -e^{i\widetilde\varphi(\frac{9}{10})}
        \sum_{k=0}^{40}i^{k+1}(P_+^\sharp)^{(k)}(\widetilde\varphi(\tfrac{9}{10}))\right
        ]&\leq -1.9\cdot 10^{-11} \\
        \Im\left
        [e^{i\widetilde\varphi(\frac{7}{10})}
        \sum_{k=0}^{40}i^{k+1}(Q_+^\sharp)^{(k)}(\widetilde\varphi(\tfrac{7}{10}))
        -e^{i\widetilde\varphi(\frac{9}{10})}
        \sum_{k=0}^{40}i^{k+1}(Q_+^\sharp)^{(k)}(\widetilde\varphi(\tfrac{9}{10}))\right
             ]&\leq -7\cdot 10^{-12}.
       \end{align*}
       Consequently, Lemmas \ref{lem:hiosc} and \ref{lem:sharps}
       yield
       \begin{align*}
         \int_{\frac{7}{10}}^{\frac{9}{10}}\cos(\widetilde\varphi(y))\frac{P_C^\psi(y,a_-)}{(1+y)P_{d,R}(y,a_-)}dy
         =\Re
         \int_{\frac{7}{10}}^{\frac{9}{10}}e^{i\widetilde\varphi(y)}\frac{P_C^\psi(y,a_-)}{(1+y)P_{d,R}(y,a_-)}dy
         &\geq
           1.4\cdot
           10^{-11} \\
         \int_{\frac{7}{10}}^{\frac{9}{10}}\sin(\widetilde\varphi(y))\frac{P_S^\psi(y,a_-)}{(1+y)P_{d,R}(y,a_-)}dy
         =\Im\int_{\frac{7}{10}}^{\frac{9}{10}}e^{i\widetilde\varphi(y)}\frac{P_S^\psi(y,a_-)}{(1+y)P_{d,R}(y,a_-)}dy
         &\geq 7\cdot 10^{-12}
       \end{align*}
       as well as
       \begin{align*}
         \int_{\frac{7}{10}}^{\frac{9}{10}}\cos(\widetilde\varphi(y))\frac{P_C^\psi(y,a_+)}{(1+y)P_{d,R}(y,a_+)}dy
         =\Re
         \int_{\frac{7}{10}}^{\frac{9}{10}}e^{i\widetilde\varphi(y)}\frac{P_C^\psi(y,a_+)}{(1+y)P_{d,R}(y,a_+)}dy
         &\leq
           -1.8\cdot
           10^{-11} \\
         \int_{\frac{7}{10}}^{\frac{9}{10}}\sin(\widetilde\varphi(y))\frac{P_S^\psi(y,a_+)}{(1+y)P_{d,R}(y,a_+)}dy
         =\Im\int_{\frac{7}{10}}^{\frac{9}{10}}e^{i\widetilde\varphi(y)}\frac{P_S^\psi(y,a_+)}{(1+y)P_{d,R}(y,a_+)}dy
         &\leq -6\cdot 10^{-12}.
       \end{align*}
The claim now follows by combining these bounds with Lemmas
\ref{lem:phitw}, \ref{lem:Itw},
\[
        \|P_C^\psi(\cdot,a_\pm)\|_{T([0,1])}+\|P_S^\psi(\cdot,a_\pm)\|_{T([0,1])}\leq
        10^{-6},\]
      and $\min_{[0,1]\times J_*}P_{d,R}\geq 8$ (Lemma \ref{lem:FRadm}).
    \end{proof}

    \subsection{The remaining contribution from the right-hand side}
    Next, we turn to the
    final contribution from the right-hand side and estimate
    $I_F(a_\pm,0,\frac{7}{10})$. Again, we need some auxiliary polynomials.

    \begin{definition}[{\tt Definition\_9.14.nb}]
      \label{def:flats}
      For $y \in [0,\frac{7}{10}]$, we set
      \begin{align*}
        P_\pm^\flat(y)&:=10^{-8}\sum_{n=0}^{40}c_n(P_\pm^\flat)T_n\left (\tfrac{10}{7}(2y-\tfrac{7}{10})\right)
        \\
        Q_\pm^\flat(y)&:=10^{-8}\sum_{n=0}^{40}c_n(Q_\pm^\flat)T_n\left(\tfrac{10}{7}(2y-\tfrac{7}{10})\right),
      \end{align*}
      where the coefficients $(c_n(P_{\pm}^\flat))_{n=0}^{40}\subset
      \Q$ and $(c_n(Q_{\pm}^\flat))_{n=0}^{40}\subset \Q$ are given in
      Appendix \ref{apx:flats}.
    \end{definition}

    \begin{lemma}[{\tt Lemma\_9.15.nb}]
      \label{lem:flats}
      We have the bounds
      \begin{align*}
        \sup_{y\in (0,\frac{7}{10})}\left
          |\cos(\varphi(y,a_*+a_\pm))\frac{P_C^\psi(y,a_\pm)}{(1+y)P_{d,R}(y,a_\pm)}-P_\pm^\flat(y)
        \right|&\leq 1.5\cdot 10^{-12} \\
          \sup_{y\in (0,\frac{7}{10})}\left
        |\sin(\varphi(y,a_*+a_\pm))
        \frac{P_S^\psi(y,a_\pm)}{(1+y)P_{d,R}(y,a_\pm)}-Q_\pm^\flat(y)
        \right |&\leq 3.2\cdot 10^{-12}.
        \end{align*}
    \end{lemma}

    \begin{proof}
      Let $y\in [0,\frac{7}{10}]$.
      From Lemma \ref{lem:phitw}, we have the bound
      \[ \left
          |\cos(\varphi(y,a_*+a_\pm))-\cos(\widetilde\varphi(y))\right
        |\leq
        \left \|\varphi(\cdot,a_*+a_\pm)-\widetilde\varphi\right\|_{L^\infty(0,\frac{7}{10})}\leq
        10^{-7}. \]
      Furthermore, $\widetilde\varphi(y)\in [-12,0]$ and
from the bound
      \[ \left |\cos(x)-\sum_{k=0}^{n-1}
          (-1)^k\frac{(x+2\pi)^{2k}}{(2k)!}\right|\leq
        \frac{|x+2\pi|^{2n}}{(2n)!}, \]
      valid for all $x\in\R$ and $n\in\N$,
      we obtain
      \[ \left
          |\cos(\widetilde\varphi(y))-\sum_{k=0}^{18}(-1)^{k}\frac{[\widetilde\varphi(y)+2\widetilde\pi]^{2k}}{(2k)!}
        \right |\leq 10^{-12}, \]
      where $\widetilde\pi:=3.1415926535897932 \in \Q$.
      We set
      \[
        P_\pm(y):=P_C^\psi(y,a_\pm)\sum_{k=0}^{18}(-1)^k\frac{[\widetilde\varphi(y)+2\widetilde\pi]^{2k}}{(2k)!}-(1+y)P_{d,R}(y,a_\pm)P_\pm^\flat(y) \]
      and note that by Remarks \ref{rem:degPdR} and
      \ref{rem:degPXpsi}, $\deg P_\pm\leq 1696$.
      Thus, we compute
      \[ \|P_\pm\|_{T([0,\frac{7}{10}])}\leq 1.1\cdot 10^{-11} \]
      and by combining this with the bounds
      \[
        \|P_C^\psi(\cdot,a_\pm)\|_{T([0,1])}+\|P_S^\psi(\cdot,a_\pm)\|_{T([0,1])}\leq
        10^{-6} \]
      from the proof of Lemma \ref{lem:psi9/10}
      and $\min_{[0,1]}P_{d,R}\geq 8$ (Lemma \ref{lem:FRadm}), the
      claimed estimate for $P_\pm^\flat$ follows from the triangle
      inequality.
      The bound for $Q_\pm^\flat$ is proved in an analogous fashion
      by considering
      \[
        Q_\pm(y):=P_S^\psi(y,a_\pm)\sum_{k=0}^{18}(-1)^k\frac{[\widetilde\varphi(y)+2\widetilde\pi]^{2k+1}}{(2k+1)!}-(1+y)P_{d,R}(y,a_\pm)Q_\pm^\flat(y), \]
      noting that $\deg Q_\pm \leq 1736$ (Remarks \ref{rem:degPdR} and
      \ref{rem:degPXpsi}), and computing
      \[ \|Q_\pm\|_{T([0,\frac{7}{10}])}\leq 2.4\cdot 10^{-11}. \]
    \end{proof}
    
    \begin{lemma}[{\tt Lemma\_9.16.nb}]
      \label{lem:psi7/10}
      We have the bounds
      \begin{align*}
        I_F\left(a_-,0,\tfrac{7}{10}\right)&\geq 6\cdot 10^{-10} \\
        I_F\left(a_+,0,\tfrac{7}{10}\right)&\leq -6\cdot 10^{-10}.
      \end{align*}
    \end{lemma}

    \begin{proof}
      By explicit integration, we find
      \begin{align*} \int_0^\frac{7}{10} \left[
        P_-^\flat(y)+Q_-^\flat(y)\right]dy&\geq 6.04\cdot 10^{-10} \\
        \int_0^\frac{7}{10} \left[
        P_+^\flat(y)+Q_+^\flat(y)\right]dy&\leq -6.47\cdot 10^{-10}
      \end{align*}
      and the claim follows by combining these estimates with Lemmas
      \ref{lem:flats} and \ref{lem:Itw}.
    \end{proof}

In summary, we obtain the following estimates for the contribution of
the right-hand side.

\begin{corollary}
  \label{cor:IFright}
  We have the bounds
  \begin{align*}
    I_F(a_-,0,1)&\geq 6\cdot 10^{-10} \\
    I_F(a_+,0,1)&\leq -6\cdot 10^{-10}.
  \end{align*}
\end{corollary}

\begin{proof}
  The claim follows by combining Lemmas \ref{lem:psi1},
  \ref{lem:psi9/10}, and \ref{lem:psi7/10}.
\end{proof}

\subsection{The contribution from the left-hand side}

\begin{definition}[{\tt Definition\_9.18.nb}]
  For $(y,a)\in (-1,0]\times J_*$, we set
  \begin{align*}
    P_L^\psi(y,a):=&(1+y)^4(1-y)^2\sum_{k=3}^4
    \gamma_{F,k}(a) \\
    &\times \sum_{\ell=1}^4 M_F^{-1}(a)^k{}_\ell\left
      [\adj(F_L(y))^\ell{}_3 \Re r(y,a) +\adj(F_L(y))^\ell{}_4\Im
      r(y,a)\right].
      \end{align*}
\end{definition}

\begin{remark}
  \label{rem:PLpsi}
  Note that $P_L^\psi(\cdot,a_\pm)$ is a polynomial of degree at most $205$. Furthermore,
  since
  \[
    F_L(y)^{-1}=\frac{(1+y)^4\adj(F_L(y))}{(1+y)^4\det(F_L(y))}=\frac{(1+y)^4\adj(F_L(y))}{P_{d,L}(y)} \]
  and $F(y,a)^{-1}=M_F(a)^{-1}F_L(y)^{-1}$ for $y\in [-1,0]$,
  we see that
  \begin{equation}
    \label{eq:IF-10}
    I_F(a_\pm,-1,0)=\int_{-1}^0
    \frac{P_L^\psi(y,a_\pm)}{(1-y)^2P_{d,L}(y)}dy.
    \end{equation}
\end{remark}

\begin{definition}({\tt Definition\_9.20.nb})
  \label{def:PLpm}
  For $y\in [-1,0]$, we set
  \[ P_\pm^L(y):=10^{-10}\sum_{n=0}^{50}c_n(P_\pm^L)T_n(2y+1), \]
  where the coefficients $(c_n(P_\pm^L))_{n=0}^{50}\subset \Q$ are given in
  Appendix \ref{apx:PLpm}.
\end{definition}

\begin{lemma}[{\tt Lemma\_9.21.nb}]
  \label{lem:PpmL}
  We have the bounds
  \[ \sup_{y\in (-1,0)}\left
      |\frac{P_L^\psi(y,a_\pm)}{(1-y)^2P_{d,L}(y)}-P_{\pm}^L(y)\right|\leq
    3\cdot 10^{-13}. \]
\end{lemma}

\begin{proof}
  We define
  \[ P_\pm(y):=P_L^\psi(y,a_\pm)-(1-y)^2P_{d,L}(y)P_{\pm}^L(y) \]
  and note that by Remarks \ref{rem:PdL} and \ref{rem:PLpsi}, we infer
  that $\deg
  P_\pm\leq 205$. We compute
  \[ \|P_\pm\|_{T([-1,0])}\leq 2.6\cdot 10^{-13} \]
  and with $\min_{[-1,0]}P_{d,L}\geq \frac{99}{100}$ from Lemma
  \ref{lem:FLadm}, the claimed bounds follow.
\end{proof}

\begin{lemma}[{\tt Lemma\_9.22.nb}]
  \label{lem:IFleft}
  We have the bounds
  \begin{align*}
    I_F(a_-,-1,0)&\geq -2\cdot 10^{-10} \\
    I_F(a_+, -1,0)&\leq 2\cdot 10^{-10}.
  \end{align*}
\end{lemma}

\begin{proof}
  By explicit integration, we find
  \begin{align*}
    \int_{-1}^0 P_-^L(y)dy&\geq -1.9\cdot 10^{-10} \\
    \int_{-1}^0 P_+^L(y)dy&\leq 1.9\cdot 10^{-10},
  \end{align*}
  and the claim follows by combining these estimates with Lemma
  \ref{lem:PpmL} and Eq.~\eqref{eq:IF-10}.
\end{proof}

\begin{corollary}
  \label{cor:IF}
  We have the bounds
  \begin{align*}
    \psi_F\left (a_-, \mc R(a_*+a_-, f_*(\cdot,a_-))\right)&\geq 4\cdot 10^{-10} \\
    \psi_F\left (a_+, \mc R(a_*+a_+, f_*(\cdot,a_+))\right)&\leq -4\cdot 10^{-10} .
  \end{align*}
\end{corollary}

\begin{proof}
  Recall that
  \[ \psi_F\left (a_\pm, \mc R(a_*+a_\pm,
      f_*(\cdot,a_\pm)\right)=I_F(a_\pm,-1,1)=I_F(a_\pm,-1,0)+I_F(a_\pm,0,1) \]
    and thus, the claim follows by combining Corollary \ref{cor:IFright} and
    Lemma \ref{lem:IFleft}.
\end{proof}

\section{Proof of the main result}
\label{sec:end}

\noindent Finally, we are in the position to conclude the proof of the main
result of this paper as outlined in
Section \ref{sec:strategy}, points (\ref{itm:contr}) and (\ref{itm:a}).

\subsection{Solution of Eq.~\eqref{eq:solsys}}

\begin{lemma}[{\tt Lemma\_10.1.nb}]
  \label{lem:N}
  We have the bound
   \begin{align*}
      \|\mc N(a_*, f_*, a, f)-\mc N(a_*, f_*, a, g)\|_Y&\leq 10 \left
   (\|f\|_X+\|g\|_X\right )\|f-g\|_X\\
   &\quad +3\left (\|f\|_X^2+\|g\|_X^2\right)\|f-g\|_X
    \end{align*}  
  for all $f,g\in X$ and $a\in J_*$.
\end{lemma}

\begin{proof}
  Recall that
     \[ \mc
       N(a_*,f_*,a,f)(y)=\frac{\overline{f_*(y,a)}f(y)^2+2f_*(y,a)|f(y)|^2+f(y)|f(y)|^2}{(1-y)^2}, \]
     which yields
     \begin{align*}
       (&1+y)(1-y)^2\left |\mc N(a_*,f_*,a,f)(y)-\mc
                     N(a_*,f_*,a,g)(y)\right| \\
                   &\leq (1+y)|f_*(y,a)|\left |f(y)^2-g(y)^2\right|
                     +2(1+y)|f_*(y,a)|\left ||f(y)|^2-|g(y)|^2\right|
       \\
        &\quad +(1+y)\left
          |f(y)|f(y)|^2-g(y)|g(y)|^2\right| \\
       &\leq 3(1+y)|f_*(y,a)|\left [|f(y)|+|g(y)|\right]\left
         |f(y)-g(y)\right|
         +\tfrac32(1+y)\left [|f(y)|^2+|g(y)|^2\right]|f(y)-g(y)|
     \end{align*}
     for all $(y,a)\in (-1,1)\times J_*$
     and by checking that
     \[ \max_{[-1,1]\times J_*,\frac{1}{10}}
       |(y,a)\mapsto (1+y)f_*(y,a)|\leq 3.2, \]
     the claim follows.
   \end{proof}

   \begin{proposition}
     \label{prop:FP}
     For every $a\in J_*$ there exists a unique $f_a\in X$ with
     $\|f_a\|_X\leq 1.2\cdot 10^{-6}$
     such that
     \[ f_a(y)=\mc J_F\left (a, \mc G_F(a_*, f_*,a,f_a)\right )(y) \]
     for all $y\in (-1,1)$. Furthermore, the map $a\mapsto f_a: J_*\to
     X$ is continuous.
   \end{proposition}

   \begin{proof}
     Let $a\in J_*$ and
     recall that
     \begin{align*}
      \mc
      G_F(a_*, f_*, a,f)(y)=&[p_1(y,a)-p_0(y,a_*+a)]f'(y)+p_2(y,a)\overline{f'(y)}
      \\
      &+\left
        [q_1(y,a)-q_0(y,a_*+a)-\frac{2|f_*(y,a)|^2}{(1-y)^2}\right
                     ]f(y) \\
                             &+\left
        [q_2(y,a)-\frac{f_*(y,a)^2}{(1-y)^2}\right]\overline{f(y)}
        -\mc N(a_*,f_*,a,f)(y)-\mc R(a_*+a,f_*(\cdot,a))(y),
    \end{align*}
    where $(p_1,p_2,q_1,q_2)$ are the coefficients associated to
    $F$. 
    From Lemmas \ref{lem:coeffL}, \ref{lem:estcoeffR}, \ref{lem:N},
    and Propositions \ref{prop:admapr}, \ref{prop:numCJ}, we obtain the bounds
    \begin{align*}
      \left \|\mc J_F\left (a, \mc G_F(a_*, f_*,a,f)\right )\right\|_X
      &\leq C_{\mc J}(F)\left \|\mc G_F(a_*, f_*,a,f)\right \|_Y \\
      &\leq C_{\mc J}(F)\left (5\cdot
        10^{-6}\|f\|_X+10\|f\|_X^2+3\|f\|_X^3+5\cdot 10^{-9}\right) \\
      &\leq 1.2\cdot 10^{-6}
    \end{align*}
    and
    \begin{align*}
      &\left \|\mc J_F\left (a, \mc G_F(a_*, f_*,a,f)\right )
      -\mc J_F\left (a, \mc G_F(a_*, f_*,a,g)\right )\right\|_X \\
      &\quad \leq C_{\mc J}(F)\left \|\mc G_F(a_*, f_*,a,f)-\mc G_F(a_*, f_*,a,g)\right \|_Y \\
      &\quad \leq C_{\mc J}(F)\left [5\cdot
        10^{-6}+10\left (\|f\|_X+\|g\|_X\right)+3\left(\|f\|_X^2+\|g\|_X^2\right)\right]\|f-g\|_X \\
      &\quad \leq \tfrac{1}{2}\|f-g\|_X
    \end{align*}
    for all $f,g\in X$ satisfying $\|f\|_X, \|g\|_X\leq 1.2\cdot
    10^{-6}$. Consequently,
    $f\mapsto \mc J_F(a, \mc G_F(a_*,
      f_*,a,f)$ is a contractive self-map on the ball $\{f\in X:
      \|f\|_X\leq 1.2\cdot 10^{-6}\}$ and the existence of a fixed
      point $f_a$ is a consequence of
      the contraction mapping principle.

      As for the continuity, it suffices to note that
      \begin{align*} \|f_a-f_b\|_X
        &=\|\mc J_F(a,\mc G_F(a_*,f_*,a,f_a))-
          \mc J_F(b,\mc G_F(a_*,f_*,b,f_b))\|_X \\
        &\leq \tfrac12 \|f_a-f_b\|_X+\|\mc J_F(a,\mc G_F(a_*,f_*,a,f_a))-
          \mc J_F(b,\mc G_F(a_*,f_*,b,f_a))\|_X
      \end{align*}
      and the claim follows from the continuity of
      \[ b\mapsto \mc J_F(b,\mc G_F(a_*,f_*,b,f_a)): J_*\to X, \]
      see Lemmas \ref{lem:contJF} and \ref{lem:contGF}.
   \end{proof}
   
   \begin{proposition}
     \label{prop:psi}
     Let $a\mapsto f_a: J_*\to X$ be the continuous map constructed in
     Proposition \ref{prop:FP}. 
     Then there exists an $a\in J_*$ such that
     \[ \psi_F\left (a, \mc G_F\left(a_*, f_*, a,
           f_a\right)\right)=0. \]
   \end{proposition}

   \begin{proof}
     Our goal is to reduce matters to the term $\psi_F(a,\mc
     R(a_*+a,f_*(\cdot,a))$ which is controlled by Corollary \ref{cor:IF}.
     Thus, in view of Definition \ref{def:GF}, we have to show that
     the contribution of
     \[ \psi_F(a, \mc G_F(a_*,f_*,a,f_a))+\psi_F(a,\mc
       R(a_*+a,f_*(\cdot,a))) \]
     is negligible (note the minus sign in front of the $\mc R$-term
     in the definition of $\mc G_F$).
     By Lemmas \ref{lem:numCpsi}, \ref{lem:coeffL},
     \ref{lem:estcoeffR}, and \ref{lem:N}, we obtain
     \begin{align*}
       &\left |\psi_F\left(a,\mc G_F\left(a_*, f_*, a,
       f_a\right)\right)+ \psi_F\left(a,
         \mc R(a_*+a, f_*(\cdot,a))\right)\right | \\
         &\leq 13\left (5\cdot
           10^{-6}\|f_a\|_X+10\|f_a\|_X^2+3\|f_a\|_X^3 \right) \\
       &\leq 2.7\cdot 10^{-10}
     \end{align*}
     for all $a\in J_*$
     and Corollary \ref{cor:IF} yields
     \begin{align*}
       \psi_F\left(a_-,\mc G_F\left(a_*, f_*, a_-,
       f_{a_-}\right)\right)&<0, &
       \psi_F\left(a_+,\mc G_F\left(a_*, f_*, a_+,
       f_{a_+}\right)\right)&>0 .                                 
     \end{align*}
     Consequently, the claim follows by applying the intermediate value theorem
     to the continuous map
     \[ a\mapsto \psi_F\left(a,\mc G_F\left(a_*, f_*, a,
           f_{a}\right)\right): J_*\to \R, \]
     see Lemmas \ref{lem:boundpsiF} and \ref{lem:contGF}.
\end{proof}

\begin{corollary}
  \label{cor:ex}
  There exist $(a,f)\in J_*\times C^2(-1,1)$ with $\|f\|_X\leq 1.2\cdot
  10^{-6}$ such that
  \[ \mc R(a_*+a,f_*(\cdot,a)+f)=0 \]
  on $(-1,1)$.
\end{corollary}

\begin{proof}
  By Propositions \ref{prop:FP} and \ref{prop:psi}, there exist
  $(a,f)\in J_*\times X$ with $\|f\|_X\leq 1.2\cdot 10^{-6}$ that
  satisfy
  Eq.~\eqref{eq:solsys}. By Lemma \ref{lem:solsys}, the claim follows.
\end{proof}

   \subsection{Proof of Theorem \ref{thm:quant}}

   It remains to prove the stated regularity properties.

   \begin{lemma}
     \label{lem:regX}
  Let $f\in X$, $\alpha\in
  \R\setminus\{0\}$, and define $h: \R^3\to\C$ by
  \[ h(x):=(1+|x|)^{-1-\frac{i}{\alpha}}f\left (\frac{|x|-1}{|x|+1}\right
    ). \]
  Then $h\in L^4(\R^3)\cap \dot H^1(\R^3)$.
\end{lemma}

\begin{proof}
  We define $\widetilde h: [0,\infty)\to\C$ by
  \[ \widetilde h(r):=(1+r)^{-1-\frac{i}{\alpha}}f\left (\frac{r-1}{r+1}\right ). \]
  Then $h(x)=\widetilde h(|x|)$ and we have
  \[ \|h\|_{L^4(\R^3)}^4\simeq \int_0^\infty |\widetilde h(r)|^4 r^2 dr,\qquad 
    \|h\|_{\dot H^1(\R^3)}^2\simeq \int_0^\infty |\widetilde h'(r)|^2 r^2 dr. \]
  For the change of variable $r=\frac{1+y}{1-y}$, we obtain
  $dr=\frac{2dy}{(1-y)^2}$ and thus,
  \[ \|h\|_{L^4(\R^3)}^4\simeq \int_{-1}^1 (1-y)^4 |f(y)|^4
    \left (\frac{1+y}{1-y}\right )^2\frac{dy}{(1-y)^2}\lesssim \|f\|_X^4 \]
  as well as
  \begin{align*}
    \|h\|_{\dot H^1(\R^3)}^2&\simeq \int_{-1}^1 (1-y)^6 |f'(y)|^2\left
      (\frac{1+y}{1-y}\right )^2
    \frac{dy}{(1-y)^2}+\int_{-1}^1 (1-y)^4 |f(y)|^2
    \left (\frac{1+y}{1-y}\right )^2\frac{dy}{(1-y)^2} \\
    &\lesssim
    \|f\|_X^2.
  \end{align*}
\end{proof}

\begin{lemma}
  \label{lem:regQ}
  Let $(a,f)\in J_*\times X\cap C^2(-1,1)$ satisfy
  \[ \mc R(a_*+a,f_*(\cdot,a)+f)=0 \]
  on $(-1,1)$.
  Define $Q: \R^3\to\C$ by
  \[ Q(x):=(1+|x|)^{-1-\frac{i}{a_*+a}}\left [ f_*\left (\frac{|x|-1}{|x|+1},a\right
    )+f\left (\frac{|x|-1}{|x|+1}\right
    )\right]. \]
Then $Q\in L^4(\R^3)\cap \dot H^1(\R^3)\cap C^\infty(\R^3)$ and $Q$
satisfies Eq.~\eqref{eq:profile} on $\R^3$. 
\end{lemma}

\begin{proof}
  By construction and Lemmas \ref{lem:solsys}, \ref{lem:regX}, $Q\in L^4(\R^3)\cap \dot
  H^1(\R^3)\cap C^2(\R^3\setminus\{0\})$ and $Q$ satisfies
  Eq.~\eqref{eq:profile} on $\R^3\setminus\{0\}$. By H\"older's
  inequality, $Q\in H_\mathrm{loc}^1(\R^3)$ and elliptic regularity
  yields $Q\in C^\infty(\R^3)$.
\end{proof}

Finally, the stated estimate on the function $g$ in Theorem
\ref{thm:quant} is a consequence of
the following bound.

\begin{lemma}
  \label{lem:gX}
  Let $f\in X$ and define $g: [0,\infty)\to\C$ by
  \[ g(r):=f\left(\frac{r-1}{r+1}\right). \]
  Then we have
  \[ 2\sup_{r>0}r|g'(r)|+\sup_{r>0}|g(r)|\leq \|f\|_X. \]
\end{lemma}

\begin{proof}
  By definition, we have
  \[ g'(r)=\frac{2}{(r+1)^2}f'\left(\frac{r-1}{r+1}\right) \]
  and thus,
  \[ |2rg'(r)|=\left
      |\frac{4r}{(r+1)^2}f'\left(\frac{r-1}{r+1}\right)\right |\leq
    \sup_{y\in (-1,1)}(1-y^2)|f'(y)| \]
  for all $r>0$. This implies the claim.
\end{proof}

\appendix
  
 \newpage
 
 \section{Tables of coefficients}
 \label{apx:tables}

\noindent In this appendix we list all the coefficients that are used in the
 paper. The data files are available on the arXiv page of this paper, \url{https://arxiv.org/abs/2406.16597}.
 
 \subsection{Coefficients for the approximate solution}
 \label{apx:Pstar}
 
These are the coefficients used in Definitions \ref{def:gstar} and
\ref{def:Pstar} (files {\tt Re\_Pstar.dat} and {\tt Im\_Pstar.dat}).

\renewcommand{\arraystretch}{1.4}

\begin{center}
\begin{tabular}{|ccc|ccc|}
   \hline
  $n$ & $\Re c_n(P_*)$ & $\Im c_n(P_*)$ & $n$ & $\Re c_n(P_*)$ & $\Im c_n(P_*)$ \\
  \hline
  0 & $-\frac{2721451470460628}{1505677490833565}$ & $-\frac{1078867762011625}{1320610492513519}$ & 26 & $-\frac{23584979}{942423027627962}$ & $\frac{16782531}{1527543763014338}$    \\
1 & $\frac{591830832555107}{1325882795067685}$ & $-\frac{503523409161778}{1215525208542319}$ &	 27 & $-\frac{10363411}{752762158039191}$ & $-\frac{9122195}{1273036310590791}$      \\
2 & $\frac{33309682807931}{117192328161321}$ & $\frac{249431537821016}{429463288550741}$ &	 28 & $-\frac{1037623}{523068233364354}$ & $-\frac{7695059}{900322851895972}$        \\
3 & $-\frac{123831666850433}{762568250071312}$ & $\frac{58969483417430}{562787431135597}$ &	 29 & $\frac{2636394}{920776893246139}$ & $-\frac{2427135}{627015566263018}$         \\
4 & $-\frac{94362296895241}{1495365721952099}$ & $-\frac{208278440111958}{1709495256184139}$ &	 30 & $\frac{1043029}{389679881702888}$ & $-\frac{137069}{885954612787875}$          \\
5 & $\frac{27826388742075}{756461494986487}$ & $-\frac{32027304107051}{811755765327665}$ &	 31 & $\frac{939092}{879936758617377}$ & $\frac{1397610}{1301624802001117}$        \\
6 & $\frac{28327529175143}{1309014475775705}$ & $\frac{21608223588887}{967282653919513}$ &	 32 & $-\frac{36787}{936368203529441}$ & $\frac{650042}{774473262050485}$          \\
7 & $-\frac{4593572122680}{901738855187701}$ & $\frac{11604846306461}{814952123459110}$ &	 33 & $-\frac{174055}{481797912391979}$ & $\frac{210261}{711166266347953}$         \\
8 & $-\frac{8387920120151}{1245468552699097}$ & $-\frac{2389431041534}{1241876150841803}$ &	 34 & $-\frac{281567}{1074934146609516}$ & $-\frac{39638}{1036666552988253}$         \\
9 & $-\frac{406742805160}{723051633758803}$ & $-\frac{4264232445251}{1077516180027017}$ &	 35 & $-\frac{83389}{959435095886281}$ & $-\frac{157667}{1281074250202765}$          \\
10 & $\frac{1634411901667}{1087757145912963}$ & $-\frac{771344875358}{1065295895342307}$ &	 36 & $\frac{28205}{1730718489838961}$ & $-\frac{113934}{1363173358100473}$          \\
11 & $\frac{437876257661}{696307764680561}$ & $\frac{265642346741}{379808011765194}$ &		 37 & $\frac{32611}{797073541452501}$ & $-\frac{17361}{669218624902568}$             \\
12 & $-\frac{89522509038}{638255170736941}$ & $\frac{244271216116}{598168785900209}$ &		 38 & $\frac{71967}{2654250117756974}$ & $\frac{4754}{717927786469769}$            \\
13 & $-\frac{117755492100}{589974697761101}$ & $-\frac{6820473784}{682731068968999}$ &		 39 & $\frac{4933}{600713092825087}$ & $\frac{13327}{968274850704167}$             \\
14 & $-\frac{31618874999}{687200813024600}$ & $-\frac{76969905434}{814968793309763}$ &		 40 & $-\frac{2587}{1091724494207482}$ & $\frac{6388}{719277775331407}$            \\
15 & $\frac{74641597491}{2908279600103189}$ & $-\frac{18025759659}{498226704994540}$ &		 41 & $-\frac{2187}{468954843525869}$ & $\frac{3259}{1225269357285833}$            \\
16 & $\frac{25500972924}{1200692005935517}$ & $\frac{1173876505}{270355500536671}$ &		 42 & $-\frac{5147}{1711259362214795}$ & $-\frac{782}{1032843841522687}$             \\
17 & $\frac{1407090135}{324017497006981}$ & $\frac{2453687289}{261078590300893}$ &		 43 & $-\frac{1793}{2031167694756949}$ & $-\frac{1186}{793758766058135}$             \\
18 & $-\frac{1455636732}{626828763236497}$ & $\frac{6222314039}{1665707814672911}$ &		 44 & $\frac{233}{747735360938753}$ & $-\frac{515}{516229281713762}$                 \\
19 & $-\frac{3196827854}{1480462540436741}$ & $\frac{88697433}{876534657690958}$ &		 45 & $\frac{493}{866534589214144}$ & $-\frac{659}{1846327333506367}$                \\
20 & $-\frac{784093409}{1033902289669424}$ & $-\frac{694824828}{951854540026661}$ &		 46 & $\frac{341}{940674986436191}$ & $\frac{12}{1534685279850239}$                \\
21 & $\frac{16311127}{845829326226001}$ & $-\frac{574584500}{1127166940199931}$ &		 47 & $\frac{100}{1139858083359883}$ & $\frac{149}{1160416855750006}$              \\
22 & $\frac{61819697}{291114080642220}$ & $-\frac{174441984}{991147170856703}$ &			 48 & $-\frac{167}{1430503381545037}$ & $\frac{237}{925315921317341}$              \\
23 & $\frac{66841451}{441859334922159}$ & $\frac{20324816}{1031487384778575}$ &		 49 & $-\frac{1299}{18049381579847260}$ & $\frac{69}{1181044907862509}$            \\
24 & $\frac{24335317}{514920381859084}$ & $\frac{38004608}{529759164589729}$ &			 50 & $\frac{31}{1479728935331543}$ & $-\frac{59}{575506931225862}$                  \\
25 & $-\frac{40105997}{2987332128606330}$ & $\frac{20671763}{439071227878684}$ & & & \\
  \hline
\end{tabular}
\end{center}

\newpage

\subsection{Coefficients for the left fundamental matrix}
\label{apx:PL}
These are
the coefficients used in Definition \ref{def:PL} (files {\tt
  Re\_PL1.dat}, {\tt Im\_PL1.dat}, {\tt Re\_PL2.dat}, {\tt
  Im\_PL2.dat},
{\tt Re\_PL3.dat}, {\tt Im\_PL3.dat}, {\tt Re\_PL4.dat},
{\tt Im\_PL4.dat}).

\begin{center}
  \begin{tabular}{|c|c|c|c|c|}
    \hline
    $n$ & $\Re c_n(P_{L,1})$ & $\Im c_n(P_{L,1})$ & $\Re c_n(P_{L,2})$ &
                                                             $\Im
                                                                         c_n(P_{L,2})$ \\
    \hline
    0 & $-\frac{25677447589672}{1163963890016041}$ & $\frac{530349753284166}{1047712021892891}$ & $-\frac{517945004884255}{724226719805958}$ & $\frac{251627013871075}{660395344710337}$    \\
1 & $-\frac{1082442536473116}{1635220915357765}$ & $-\frac{110367753663726}{708047493550373}$ & $-\frac{114302268418242}{761576205136705}$ & $-\frac{123913358194209}{741695303878733}$   \\
2 & $-\frac{185936764340935}{1296137430677681}$ & $-\frac{111920181288506}{778797133744727}$ &  $\frac{17151015646089}{541448058831086}$ & $-\frac{56801890756933}{1070300059045458}$   \\
3 & $\frac{2345025580829}{870128025908873}$ & $-\frac{34365822575423}{1337355174852400}$ &      $\frac{12280916504473}{899971311827585}$ & $-\frac{4036122681220}{978714325901481}$     \\
4 & $\frac{144189079060498}{18283075520062425}$ & $\frac{1118439732896}{361756199062829}$ &   $\frac{2367038770634}{1164366581657513}$ & $\frac{696494230513}{544434361047937}$       \\
5 & $\frac{4674586796888}{2744490668379641}$ & $\frac{2557817734831}{941852677758271}$ &      $-\frac{68218152895}{1289289733871468}$ & $\frac{628708864225}{1260394898925096}$         \\
6 & $\frac{7687896621}{2004608749423625}$ & $\frac{799499324755}{1174208276879177}$ &	       $-\frac{69582891023}{702483920307261}$ & $\frac{42725908350}{563156598813481}$            \\
7 & $-\frac{93374633783}{1038722783958176}$ & $\frac{18545361642}{275567875203109}$ &	       $-\frac{74265152988}{3065120819772145}$ & $-\frac{355734767}{460001483151423}$            \\
8 & $-\frac{74972967052}{3069745379188931}$ & $-\frac{8395948274}{694862258440025}$ &	       $-\frac{1425956526}{580260321528581}$ & $-\frac{3054521565}{911188844028754}$             \\
9 & $-\frac{1753498904}{892264564109643}$ & $-\frac{14325340118}{2370412550810421}$ &	       $\frac{350564155}{1114357638294561}$ & $-\frac{1090756380}{1249267251208579}$           \\
10 & $\frac{252981210}{329942940569983}$ & $-\frac{1427767523}{1455622877913618}$ &	        $\frac{814673447}{4479667553603919}$ & $-\frac{42698153}{422756615139272}$             \\
11 & $\frac{233811361}{675089370308832}$ & $\frac{65076613}{1748449524823006}$ &	        $\frac{31150042}{858129033616623}$ & $\frac{8792583}{1201122954408068}$                \\
12 & $\frac{88693487}{1239673993058033}$ & $\frac{65452033}{1077676849558065}$ &	        $\frac{1621271}{615294864255555}$ & $\frac{7412633}{1261169303904716}$                 \\
13 & $\frac{16152549}{2378642812142881}$ & $\frac{10309096}{646555332622629}$ &	        $-\frac{409762}{596152022550001}$ & $\frac{4166754}{3260798591775841}$                   \\
14 & $-\frac{1165135}{1521116212026049}$ & $\frac{687636}{361267137880831}$ &		        $-\frac{328995}{1149047074605608}$ & $\frac{70602}{589692303242011}$                     \\
15 & $-\frac{324441}{1082461988954614}$ & $-\frac{102346}{336321508849921}$ &		        $-\frac{17309}{385662716915208}$ & $-\frac{15957}{785266120923205}$                      \\
16 & $\frac{15487}{443684595725487}$ & $-\frac{347503}{1425037188487487}$ &		        $\frac{5141}{2306675406400974}$ & $-\frac{21625}{1828006420046424}$                    \\
\hline
  \end{tabular}
\end{center}
 
\newpage

\begin{center}
  \begin{tabular}{|c|c|c|c|c|}
    \hline
    $n$ & $\Re c_n(P_{L,3})$ & $\Im c_n(P_{L,3})$ & $\Re c_n(P_{L,4})$ &
                                                             $\Im
                                                                         c_n(P_{L,4})$ \\
    \hline
0 & $-\frac{1880579012407712}{1264572197250759}$ & $-\frac{327819574106897}{494498622613644}$ & $-\frac{1235423862838786}{2757559154064461}$ & $-\frac{285132495792879}{4057283039847173}$    \\
1 & $-\frac{269325222738631}{842677740476761}$ & $-\frac{279059406245370}{500773457669029}$ &   $-\frac{278104214936677}{783316830514894}$ & $\frac{192467438451300}{2252055563919871}$     \\
2 & $\frac{126135288121777}{786649538688151}$ & $\frac{13012660882672}{1325207140507713}$ &   $-\frac{90252694047103}{1031085290502292}$ & $\frac{123938360444803}{7617732080365628}$     \\
3 & $\frac{66966331840168}{1117986847339497}$ & $\frac{14126422263245}{254798882291938}$ &    $-\frac{11814364894261}{1510073138688502}$ & $-\frac{725600123961}{130748046621331}$          \\
4 & $\frac{6465480873820}{1017256720219203}$ & $\frac{10619128824323}{621177549874954}$ &     $\frac{3295269917673}{1147301164879810}$ & $-\frac{1596611369498}{587329124564689}$           \\
5 & $-\frac{1195359826000}{842901581462863}$ & $\frac{2448804123921}{1358009907248222}$ &     $\frac{800208476515}{593980553933492}$ & $-\frac{431437902245}{1355349427602094}$             \\
6 & $-\frac{767233757971}{1206025809809902}$ & $-\frac{295096738099}{663997884248891}$ &	       $\frac{320343108547}{1296813455849558}$ & $\frac{45943063991}{441280912817043}$             \\
7 & $-\frac{56213820029}{744001059411659}$ & $-\frac{201950984545}{864036276330126}$ &	       $\frac{1678611220}{1182482488098459}$ & $\frac{77581308389}{1394322745386835}$              \\
8 & $\frac{4008681509}{232745015528311}$ & $-\frac{46415387181}{975810303561101}$ &	       $-\frac{13917685903}{1132831548325475}$ & $\frac{20228340188}{1704140910599865}$            \\
9 & $\frac{9101448335}{936316295545854}$ & $-\frac{2666400359}{918527912661640}$ &	       $-\frac{2406729043}{668455332160438}$ & $\frac{2202762858}{2887200986290567}$               \\
10 & $\frac{1140382427}{546458830753244}$ & $\frac{532397964}{433896933852091}$ &	        $-\frac{187103959}{451158034677887}$ & $-\frac{466366777}{1364508737553223}$                 \\
11 & $\frac{136975753}{899505647020653}$ & $\frac{2045592387}{4885544477645626}$ &	        $\frac{32673485}{519815835953651}$ & $-\frac{86262448}{668106598542119}$                     \\
12 & $-\frac{129975825}{2525855610320201}$ & $\frac{32576477}{739747340611994}$ &	        $\frac{18670649}{461399449262856}$ & $-\frac{25268097}{1304183977126009}$                    \\
13 & $-\frac{18451201}{862738719240511}$ & $-\frac{6148497}{559990880016659}$ &		        $\frac{56345467}{5970307008258519}$ & $\frac{1277713}{1735339492096633}$                   \\
14 & $-\frac{2091976}{537632796652219}$ & $-\frac{7839474}{1287616149628645}$ &		        $\frac{2514531}{2339957268256139}$ & $\frac{1294620}{1116204772221739}$                    \\
15 & $-\frac{163144}{874171390981047}$ & $-\frac{1575517}{1122236720308525}$ &		        $-\frac{39037}{429277340876428}$ & $\frac{486093}{1850805774675719}$                       \\
16 & $\frac{87944}{872571465709971}$ & $-\frac{172115}{1216760034973378}$ &		        $-\frac{71911}{1010941875327823}$ & $-\frac{4989}{1698292824029957}$                         \\
    \hline
  \end{tabular}
\end{center}

\newpage

\subsection{Coefficients for the right fundamental matrix}
\label{apx:PR}
These are the coefficients used in Definition \ref{def:PR} (files {\tt
  Re\_PR1.dat}, {\tt Im\_PR1.dat}, {\tt Re\_PR2.dat}, {\tt
  Im\_PR2.dat},
{\tt Re\_PR3.dat}, {\tt Im\_PR3.dat}, {\tt Re\_PR4.dat},
{\tt Im\_PR4.dat},
{\tt Re\_QR3.dat}, {\tt Im\_QR3.dat}, {\tt Re\_QR4.dat},
{\tt Im\_QR4.dat}).

\begin{center}
  \begin{tabular}{|c|c|c|c|c|}
    \hline
    $n$ & $\Re c_n(P_{R,1})$ & $\Im c_n(P_{R,1})$ & $\Re c_n(P_{R,2})$ &
                                                             $\Im
                                                                         c_n(P_{R,2})$ \\
    \hline
0 & $\frac{905182892168009}{1210536500275806}$ & $-\frac{964124562956389}{764588616268277}$   & $\frac{229599367281501}{2318964790029964}$ & $-\frac{437213741432953}{1273601827018991}$      \\
1 & $-\frac{29543275157879}{878678715680883}$ & $\frac{1867829427849379}{1290374648689744}$ & $-\frac{1224654218910595}{4267202327370079}$ & $\frac{1158060336659301}{2244448991962928}$      \\
2 & $-\frac{63357277247875}{175246925958027}$ & $-\frac{313439385873153}{1248519869481776}$ &   $-\frac{4615735457747}{182245197945545}$ & $-\frac{109625271108935}{721111072614097}$           \\
3 & $\frac{171643411122718}{1439385508783457}$ & $-\frac{40427939601805}{290932896249328}$ &    $\frac{48133140500126}{905133989460001}$ & $-\frac{51314484136660}{1873337812961177}$         \\
4 & $\frac{24891665834664}{754600155259735}$ & $\frac{29414188024163}{421620319370787}$ &     $-\frac{2158192286492}{888659946510293}$ & $\frac{24568723216499}{813420496525666}$             \\
5 & $-\frac{22895742052880}{909723752359467}$ & $\frac{2370979013401}{251620082225998}$ &     $-\frac{9264298757713}{1015041263348322}$ & $-\frac{991210242019}{1074271419267915}$            \\
6 & $-\frac{351984366023}{155503237138870}$ & $-\frac{11082072690006}{946928032602073}$ &       $\frac{266663365757}{234184892421294}$ & $-\frac{3296112962521}{700243258667708}$             \\
7 & $\frac{2684717689919}{665397340312499}$ & $-\frac{936779909321}{1917134228555035}$ &	       $\frac{3716211666823}{2546912601321189}$ & $\frac{561012480343}{1039110453379923}$            \\
8 & $\frac{204495031733}{892063953585555}$ & $\frac{983138246457}{575177877165194}$ &	       $-\frac{197279420208}{965673491448683}$ & $\frac{889394790347}{1283912791353518}$               \\
9 & $-\frac{282079080365}{475205232780054}$ & $\frac{118544413214}{1337988439159219}$ &       $-\frac{199852951835}{856374836775734}$ & $-\frac{42930913826}{544804255184983}$                \\
10 & $-\frac{75234678489}{1146991988921768}$ & $-\frac{311455913477}{1307958831340280}$ &        $\frac{12604670095}{682298764674304}$ & $-\frac{104425124649}{992700518480605}$              \\
11 & $\frac{27446080969}{339185504352084}$ & $-\frac{47454977496}{1370191231278385}$ &	        $\frac{39389088692}{1054442482787855}$ & $\frac{1446074809}{481257886411013}$                \\
12 & $\frac{34471087395}{1689613159415371}$ & $\frac{46134027428}{1546999539500645}$ &	        $\frac{779106751}{405616788629073}$ & $\frac{25381533964}{1605587623458439}$                 \\
13 & $-\frac{5693269851}{641781551348606}$ & $\frac{4691296707}{456755583552215}$ &	        $-\frac{5097165836}{930554422274277}$ & $\frac{2815666246}{1372958236326965}$                  \\
14 & $-\frac{1526624585}{303828308783349}$ & $-\frac{1036570712}{412578113410611}$ &	        $-\frac{977336355}{708469650207641}$ & $-\frac{2175607579}{1066073009498197}$                  \\
15 & $\frac{185663646}{629433111835183}$ & $-\frac{2768069841}{1216023567070165}$ &	        $\frac{761796682}{1293158957171919}$ & $-\frac{102953705}{129554014255937}$                  \\
16 & $\frac{1187495175}{1259329621322999}$ & $-\frac{278541911}{1711100496646239}$ &	        $\frac{656489921}{1702055151458457}$ & $\frac{49547699}{348913642508819}$                    \\
17 & $\frac{432419895}{2089230096254686}$ & $\frac{328568765}{905848321145561}$ &	        $\frac{632938}{680946124449833}$ & $\frac{50557499}{281339303152100}$                        \\
18 & $-\frac{108311147}{946194004410941}$ & $\frac{128672807}{943052250448267}$ &	        $-\frac{15995173}{225750884329698}$ & $\frac{28515899}{1070558314508552}$                      \\
19 & $-\frac{207866887}{2812943312601145}$ & $-\frac{32232698}{1186108690826973}$ &	        $-\frac{11149469}{511236685737299}$ & $-\frac{31416532}{1220441323196857}$                     \\
20 & $-\frac{1157633}{1021508877076545}$ & $-\frac{24829791}{711272401151756}$ &		        $-\frac{1694723}{249616541638119}$ & $-\frac{3511729}{263959769361999}$                       \\
21 & $\frac{10779565}{765327632541236}$ & $-\frac{6994135}{1142907242394279}$ &		        $\frac{3376427}{518049036863803}$ & $\frac{569388}{864034798879697}$                         \\
22 & $\frac{1197415}{242248789906637}$ & $\frac{3950018}{815296383956159}$ &		        $\frac{484957}{571359614317252}$ & $\frac{4999839}{1750261175451749}$                        \\
23 & $-\frac{617141}{557185087932570}$ & $\frac{3020612}{1023745657645005}$ &		        $-\frac{772488}{753765149157353}$ & $\frac{207458}{236124336006689}$                           \\
24 & $-\frac{742697}{540112325150848}$ & $\frac{74917}{801607829597733}$ &		        $-\frac{551167}{1028765608432637}$ & $-\frac{100702}{397829137552219}$                         \\
25 & $-\frac{199319}{749381056956021}$ & $-\frac{416717}{691331029767839}$ &		        $\frac{9719}{389320976820880}$ & $-\frac{623859}{2201810291272148}$                          \\
26 & $\frac{313398}{1824536757806963}$ & $-\frac{345604}{1313864215913825}$ &		        $\frac{47930}{401620769507521}$ & $-\frac{39667}{626967832812710}$                           \\
27 & $\frac{97817}{906373931856942}$ & $\frac{111076}{2036801353243733}$ &		        $\frac{26736}{912224592309695}$ & $\frac{41026}{947198033341063}$                            \\
28 & $\frac{7163}{1057510687908509}$ & $\frac{66183}{835033288759151}$ &		        $-\frac{21319}{1526009128296655}$ & $\frac{25675}{818125449081334}$                            \\
29 & $-\frac{3238}{236551642035411}$ & $\frac{1045}{751514878587959}$ &		        $-\frac{6659}{1221774686284569}$ & $-\frac{1520}{644594865315677}$                             \\
30 & $-\frac{8985}{2096914333847348}$ & $-\frac{7963}{719902311931441}$ &		        $\frac{1967}{2836559485514530}$ & $-\frac{365}{71800999828723}$                              \\
    \hline
  \end{tabular}
\end{center}

\newpage

\begin{center}
  \begin{tabular}{|c|c|c|c|c|}
    \hline
    $n$ & $\Re c_n(P_{R,3})$ & $\Im c_n(P_{R,3})$ & $\Re c_n(P_{R,4})$ &
                                                             $\Im
                                                                         c_n(P_{R,4})$ \\
    \hline
0 & $\frac{740650831460613}{934942792365922}$ & $\frac{2689410496949907}{2001499430084200}$ & $\frac{761862582957073}{961718937524583}$ & $\frac{1182470673057526}{880012322737655}$   \\
1 & $\frac{569427687308973}{601170187939750}$ & $-\frac{1024935540756591}{672478931558924}$ &   $\frac{873060318475391}{921728689064955}$ & $-\frac{1624916621990511}{1066137479262242}$   \\
2 & $-\frac{532979823280321}{633320510278501}$ & $-\frac{69957624743533}{625432248141991}$ &    $-\frac{599990717626501}{712947115166357}$ & $-\frac{399698341481122}{3573366494468815}$   \\
3 & $\frac{28464248861417}{736126048327861}$ & $\frac{288594245548787}{785222675847615}$ &    $\frac{177493227234231}{4590227854071298}$ & $\frac{737446315483133}{2006483421166275}$  \\
4 & $\frac{73682279716033}{470643857894806}$ & $-\frac{135692132917951}{2621852039059466}$ &    $\frac{126650470605855}{808976952394186}$ & $-\frac{34594079422525}{668429007068348}$      \\
5 & $-\frac{17265510261599}{580749122199510}$ & $-\frac{119844136829124}{1946128179334357}$ &   $-\frac{5291285009522}{177979629795357}$ & $-\frac{66233280652543}{1075550772011991}$      \\
6 & $-\frac{3120721941593}{124168371845008}$ & $\frac{12054167583115}{841493452290073}$ &     $-\frac{9755284617771}{388146663029824}$ & $\frac{17294129832504}{1207291745096663}$     \\
7 & $\frac{4332350001641}{716914600094349}$ & $\frac{10045935101009}{1000604889763212}$ &     $\frac{1801317645348}{298080930411533}$ & $\frac{11140336511600}{1109610511618277}$      \\
8 & $\frac{2803101758312}{664325668187493}$ & $-\frac{1962960292922}{856141345624497}$ &	       $\frac{19349818686402}{4585841805433331}$ & $-\frac{2064861595960}{900585402398553}$       \\
9 & $-\frac{300547693681}{389171029419735}$ & $-\frac{672359056400}{381241246995181}$ &	       $-\frac{217878391779}{282124800155504}$ & $-\frac{2445608211877}{1386709549730462}$        \\
10 & $-\frac{690193576016}{913435044344523}$ & $\frac{138621865565}{667063957884079}$ &        $-\frac{311859615025}{412729864748876}$ & $\frac{29907316969}{143917362139959}$         \\
11 & $\frac{40653294702}{1566290108545715}$ & $\frac{180251714823}{569879310845225}$ &	        $\frac{16409201114}{632213688518857}$ & $\frac{304214814195}{961797943684436}$          \\
12 & $\frac{177505964856}{1364189386789441}$ & $\frac{15221828077}{781193124852170}$ &	        $\frac{127488203139}{979787095065619}$ & $\frac{16999326924}{872415405887335}$          \\
13 & $\frac{21592898956}{970260455485067}$ & $-\frac{74128523975}{1472152645404764}$ &	        $\frac{1841798254}{82759800639973}$ & $-\frac{44159771999}{876989338014120}$              \\
14 & $-\frac{16993848037}{947923689335838}$ & $-\frac{17223588445}{1122828114044253}$ &	        $-\frac{20059202841}{1118910415155941}$ & $-\frac{7384597329}{481411496702614}$           \\
15 & $-\frac{7342603046}{836112599056501}$ & $\frac{4762129268}{892540250279987}$ &	        $-\frac{8139973110}{926910256568993}$ & $\frac{4347565127}{814840725268847}$            \\
16 & $\frac{536256341}{549789972102272}$ & $\frac{1142895923}{257839662913172}$ &	        $\frac{719932883}{738102003459719}$ & $\frac{3168543301}{714829863519978}$              \\
17 & $\frac{1807540235}{896019492770726}$ & $\frac{247927811}{956783412035527}$ &	        $\frac{4801663923}{2380242712960663}$ & $\frac{217098274}{837808499606943}$             \\
18 & $\frac{924598819}{2165838881659034}$ & $-\frac{1331113277}{1640260084905636}$ &	        $\frac{380384387}{891036499734743}$ & $-\frac{349636508}{430838470479917}$                \\
19 & $-\frac{327603758}{1189789003082799}$ & $-\frac{232368144}{741234500984239}$ &	        $-\frac{70921121}{257570823873469}$ & $-\frac{376608293}{1201348236995778}$               \\
20 & $-\frac{147579221}{820091368279374}$ & $\frac{17954603}{282176838429834}$ &	        $-\frac{174007412}{966951686225751}$ & $\frac{8878000}{139527784133131}$                \\
21 & $-\frac{3984606}{952669591918255}$ & $\frac{73548109}{840499395336333}$ &		        $-\frac{3093027}{739504174285273}$ & $\frac{68363360}{781248675518769}$                 \\
22 & $\frac{33018762}{901835766983351}$ & $\frac{24706327}{1391569899147199}$ &	        $\frac{4704407}{128490659190882}$ & $\frac{32585943}{1835384815157930}$                 \\
23 & $\frac{9953571}{699358207989101}$ & $-\frac{4805715}{396862377504766}$ &		        $\frac{28139951}{1977170374759080}$ & $-\frac{2884679}{238220653176077}$                  \\
24 & $-\frac{2275043}{871158267297332}$ & $-\frac{10105489}{1194103065479818}$ &		        $-\frac{5809473}{2224560341316904}$ & $-\frac{5127745}{605913877448069}$                  \\
25 & $-\frac{1508434}{370804653670483}$ & $-\frac{342776}{458401222363013}$ &		        $-\frac{611348}{150282135918537}$ & $-\frac{968837}{1295645159143331}$                    \\
26 & $-\frac{887599}{960034900081162}$ & $\frac{1726039}{1103539410688527}$ &		        $-\frac{1521889}{1646088553558104}$ & $\frac{2072312}{1324928326209757}$                \\
27 & $\frac{970196}{1817456220397669}$ & $\frac{294121}{343115647286057}$ &		        $\frac{575969}{1078955635568715}$ & $\frac{2533931}{2956032977050961}$                  \\
28 & $\frac{152106}{468549910175669}$ & $\frac{17638}{1209308611042749}$ &		        $\frac{284158}{875325137573125}$ & $\frac{11638}{797932510223127}$                      \\
29 & $-\frac{5753}{513103030712973}$ & $-\frac{56845}{323718828377007}$ &		        $-\frac{1152}{102745470429575}$ & $-\frac{61937}{352716563869939}$                        \\
30 & $-\frac{114693}{3555716884771193}$ & $-\frac{41540}{700182517676521}$ &		        $-\frac{21524}{667287892267315}$ & $-\frac{70979}{1196395159416509}$                      \\
\hline
  \end{tabular}
\end{center}

\newpage

\begin{center}
  \begin{tabular}{|c|c|c|c|c|}
    \hline
    $n$ & $\Re c_n(Q_{R,3})$ & $\Im c_n(Q_{R,3})$ & $\Re c_n(Q_{R,4})$ &
                                                             $\Im
                                                                         c_n(Q_{R,4})$ \\
    \hline
0 & $\frac{33999531320775}{130417803622346}$ & $-\frac{115215147451424}{803828617788823}$    &   $-\frac{215362893343035}{826104197336093}$ & $\frac{154457357338925}{1077612161277739}$   \\
1 & $-\frac{10572359682749}{63161708216629}$ & $\frac{746784388391963}{1635700668949216}$ &	$\frac{168645512739677}{1007527078849589}$ & $-\frac{317958620607164}{696432781013537}$     \\
2 & $-\frac{175915122460926}{811520164472617}$ & $-\frac{243583922679691}{835394424568651}$ &	$\frac{533699456369471}{2462027508226876}$ & $\frac{194060507809468}{665549123560885}$    \\
3 & $\frac{19473410038635}{88507063473734}$ & $-\frac{21045572025109}{514527245308103}$ &	$-\frac{169979945792361}{772562474762171}$ & $\frac{55499247893443}{1356859062857279}$    \\
4 & $-\frac{27203093314906}{1379800723772467}$ & $\frac{107427966706222}{878713199135315}$ &	$\frac{10592335313465}{537266544020241}$ & $-\frac{122037231436089}{998210515722362}$       \\
5 & $-\frac{148538381407323}{2528542018588064}$ & $-\frac{71585036482556}{2713561137057569}$ &	$\frac{50735576511685}{863662548572517}$ & $\frac{20810520327337}{788860661068031}$       \\
6 & $\frac{22668386868535}{1234204302330018}$ & $-\frac{28630716086205}{1111852083268219}$ &	$-\frac{10566382260292}{575297859585135}$ & $\frac{23666109621316}{919055367182821}$      \\
7 & $\frac{16700171964815}{1539997663421869}$ & $\frac{13629148906684}{1326405874518603}$ &	$-\frac{6615922888841}{610083884864192}$ & $-\frac{12226684511265}{1189916279634608}$       \\
8 & $-\frac{9318355126389}{1825506441204187}$ & $\frac{3565203794275}{790985430829102}$ &	$\frac{3826673106546}{749661963901733}$ & $-\frac{5739913308203}{1273472166808886}$         \\
9 & $-\frac{567030959136}{298196167204823}$ & $-\frac{3448396151248}{1474386964872959}$ &	$\frac{2321617326598}{1220916384464621}$ & $\frac{3061173673413}{1308827165828666}$       \\
10 & $\frac{1008214961432}{1002184627055011}$ & $-\frac{1180949097475}{1434495188229149}$ &	 $-\frac{699863207957}{695677186797248}$ & $\frac{795907345558}{966786172173145}$         \\
11 & $\frac{300540964851}{820355743082368}$ & $\frac{271331115501}{666370594916120}$ &		 $-\frac{503135583604}{1373357424879735}$ & $-\frac{701921307898}{1723870550801489}$        \\
12 & $-\frac{71023225441}{461580005147271}$ & $\frac{149246212351}{901923910345033}$ &		 $\frac{131901068753}{857225162837987}$ & $-\frac{135042574257}{816088560763511}$           \\
13 & $-\frac{98315576929}{1318912162058657}$ & $-\frac{24695622455}{466774332191571}$ &		 $\frac{102662046401}{1377220434538992}$ & $\frac{10539597265}{199209940300267}$          \\
14 & $\frac{15054662375}{970125874534439}$ & $-\frac{47489023456}{1445086203759829}$ &		 $-\frac{10157201126}{654532355474557}$ & $\frac{40965910347}{1246588528436716}$          \\
15 & $\frac{13142842223}{949074650501087}$ & $\frac{5076245364}{1645116043925125}$ &		 $-\frac{9342805723}{674665336903952}$ & $-\frac{2188398995}{709219125363011}$              \\
16 & $\frac{226298380}{767392447913163}$ & $\frac{6845158852}{1249648977249809}$ &		 $-\frac{365469193}{1239329679041972}$ & $-\frac{4830796255}{881907889924153}$              \\
17 & $-\frac{1005807501}{523850955588044}$ & $\frac{663280168}{845902234697095}$ &		 $\frac{1711840595}{891571528959154}$ & $-\frac{588284646}{750258066949153}$                \\
18 & $-\frac{1319447118}{2369576824364017}$ & $-\frac{88934573}{154880512839497}$ &		 $\frac{571497853}{1026345087399398}$ & $\frac{1399566979}{2437360906435628}$             \\
19 & $\frac{55980467}{672098004050896}$ & $-\frac{249788978}{849181992024207}$ &			 $-\frac{69112865}{829764757424032}$ & $\frac{439436359}{1493906759582821}$               \\
20 & $\frac{443264437}{3800279681354425}$ & $-\frac{21439958}{1044729047273153}$ &		 $-\frac{118387033}{1014978415797737}$ & $\frac{7639299}{372248750025758}$                \\
21 & $\frac{6228350}{129225359166053}$ & $\frac{22653891}{565430884939006}$ &			 $-\frac{59762936}{1239957110537757}$ & $-\frac{38607422}{963623811321139}$                 \\
22 & $\frac{43198}{634626999675329}$ & $\frac{27316354}{1320151184291751}$ &			 $-\frac{223990}{3290663957990577}$ & $-\frac{71991424}{3479218480345129}$                  \\
23 & $-\frac{23340089}{1457800707355487}$ & $\frac{1851731}{1022824680125384}$ &		 $\frac{44561482}{2783269591663031}$ & $-\frac{1493525}{824965521657446}$                   \\
24 & $-\frac{3303005}{481981729454043}$ & $-\frac{1924111}{795104633913382}$ &			 $\frac{7004037}{1022044431183152}$ & $\frac{489143}{202129641141438}$                    \\
25 & $\frac{551231}{262645469615304}$ & $-\frac{2063861}{1738090169892127}$ &			 $-\frac{3555861}{1694263897044514}$ & $\frac{1026501}{864472606190261}$                  \\
26 & $\frac{884777}{582587972454744}$ & $-\frac{157524}{727771329356959}$ &			 $-\frac{2780251}{1830676874517844}$ & $\frac{157524}{727771329356959}$                   \\
    \hline
  \end{tabular}
\end{center}

\newpage

\subsection{Coefficients for approximation of phase}
\label{apx:phitw}
These are the coefficients used in Definition \ref{def:phitw} (file
{\tt phitw.dat}).
\begin{center}
  \begin{tabular}{|cc|cc|cc|cc|}
    \hline
    $n$ & $c_n(\widetilde\varphi)$ & $n$ & $c_n(\widetilde\varphi)$ & $n$ & $c_n(\widetilde\varphi)$ & $n$ & $c_n(\widetilde\varphi)$ \\
    \hline
0 & $-\frac{355950139}{14657721}$ &  11 & $-\frac{46328383}{136777634}$ &    22 & $-\frac{1512364}{3182339229}$ &   33 & $-\frac{68819}{132457417335}$ \\ 
1 & $-\frac{351922817}{8012162}$ &   12 & $-\frac{33010805}{173605428}$ &    23 & $-\frac{3688957}{14330077878}$ &  34 & $-\frac{26569}{95627042824}$    \\ 
2 & $-\frac{982442337}{29494607}$ &  13 & $-\frac{15683952}{147751241}$ &    24 & $-\frac{554124}{3980489795}$ &    35 & $-\frac{26051}{175267834537}$    \\
3 & $-\frac{927585801}{41134127}$ &  14 & $-\frac{25169019}{426787852}$ &    25 & $-\frac{295574}{3932362577}$ &    36 & $-\frac{52249}{655677095601}$    \\
4 & $-\frac{106335447}{7377385}$ &   15 & $-\frac{4349081}{133306007}$ &     26 & $-\frac{574309}{14171446290}$ &   37 & $-\frac{21233}{491171341160}$    \\
5 & $-\frac{222295829}{25020823}$ &  16 & $-\frac{18067173}{1004771323}$ &   27 & $-\frac{243338}{11151617411}$ &   38 & $-\frac{17458}{717858338557}$    \\
6 & $-\frac{635353759}{118991106}$ & 17 & $-\frac{9593701}{971236636}$ &     28 & $-\frac{233428}{19892041241}$ &   39 & $-\frac{9171}{594311525288}$     \\
7 & $-\frac{99180037}{31489914}$ &   18 & $-\frac{2882805}{532844677}$ &     29 & $-\frac{783613}{124316164231}$ &  40 & $-\frac{5123}{802284206161}$     \\
8 & $-\frac{90284711}{49302362}$ &   19 & $-\frac{3384291}{1145130091}$ &    30 & $-\frac{140780}{41623422413}$ &   & \\
9 & $-\frac{50679761}{48147039}$ &   20 & $-\frac{6910633}{4290928466}$ &    31 & $-\frac{181237}{99964939686}$ &   & \\
10 & $-\frac{39578500}{66027997}$ &  21 & $-\frac{3840951}{4385988911}$ &    32 & $-\frac{87979}{90614186264}$ &    & \\
    \hline
  \end{tabular}
\end{center}

\newpage

\subsection{Coefficients for functional $\psi_F$ on the right-hand
  side at high frequency.}
\label{apx:sharps}
These are the coefficients used in Definition \ref{def:sharps}
(files {\tt Pshm.dat}, {\tt Pshp.dat}, {\tt Qshm.dat}, {\tt Qshp.dat}).

\begin{center}
\begin{tabular}{|ccc|ccc|}
   \hline
  $n$ & $c_n(P_-^\sharp)$ & $c_n(P_+^\sharp)$ & $n$ & $c_n(P_-^\sharp)$ & $c_n(P_+^\sharp)$ \\
  \hline
  0 & $-\frac{20631567}{455512634}$    & $\frac{9568160}{201881083}$      & 21 & $\frac{444379}{11176237362}$   & $\frac{464738}{9889823197}$        \\
1 & $-\frac{69554432}{1023019957}$   & $\frac{47047045}{649018097}$     & 22 & $\frac{58611}{800400353}$      & $\frac{700381}{9033942495}$       \\
2 & $-\frac{12329231}{269313212}$    & $\frac{15204955}{303630746}$     & 23 & $\frac{376613}{4875608921}$    & $\frac{1861973}{23332528565}$     \\
3 & $-\frac{24877483}{853350175}$    & $\frac{13738349}{412452315}$     & 24 & $\frac{369375}{5539349437}$    & $\frac{856040}{12551471521}$      \\
4 & $-\frac{7426145}{423099917}$     & $\frac{12395978}{560685879}$     & 25 & $\frac{425169}{8311530079}$    & $\frac{786717}{15112362946}$      \\
5 & $-\frac{11757787}{1144307006}$   & $\frac{4674985}{322495783}$      & 26 & $\frac{8540}{237393733}$       & $\frac{470938}{12898482173}$      \\
6 & $-\frac{13897943}{2368284443}$   & $\frac{7144027}{757618270}$      & 27 & $\frac{1192200}{50764918679}$  & $\frac{331739}{13936308798}$      \\
7 & $-\frac{2608823}{897244087}$     & $\frac{5488408}{849193287}$      & 28 & $\frac{184101}{12908496413}$   & $\frac{280721}{19424774289}$      \\
8 & $-\frac{6500988}{7285807193}$    & $\frac{6339738}{1317306931}$     & 29 & $\frac{116228}{14542704637}$   & $\frac{260555}{32147771403}$      \\
9 & $\frac{546214}{4782016773}$      & $\frac{6617837}{1853174972}$     & 30 & $\frac{217996}{54065032051}$   & $\frac{171382}{41809662947}$      \\
10 & $\frac{829499}{5077812811}$      & $\frac{5588966}{2483699731}$    & 31 & $\frac{205018}{119920816427}$  & $\frac{20157}{11522056256}$       \\
11 & $-\frac{1461956}{3824446197}$    & $\frac{2120393}{2426452111}$    & 32 & $\frac{39536}{85699224537}$    & $\frac{55952}{115365215341}$      \\
12 & $-\frac{1682122}{1598994471}$    & $-\frac{577767}{1940668792}$    & 33 & $-\frac{14154}{106459721287}$  & $-\frac{12646}{106373436087}$     \\
13 & $-\frac{4038402}{2681337091}$    & $-\frac{1891895}{1794830111}$   & 34 & $-\frac{38879}{108463987388}$  & $-\frac{12622}{36054520295}$      \\
14 & $-\frac{2429071}{1496074848}$    & $-\frac{6247835}{4617218998}$   & 35 & $-\frac{120350}{303451706711}$ & $-\frac{680501}{1737699972493}$   \\
15 & $-\frac{3992099}{2744509183}$    & $-\frac{2951339}{2282624457}$   & 36 & $-\frac{67049}{189613539928}$  & $-\frac{46268}{131964663029}$     \\
16 & $-\frac{1677957}{1492071964}$    & $-\frac{1954021}{1900589114}$   & 37 & $-\frac{16593}{57560035534}$   & $-\frac{7513}{26229120710}$       \\
17 & $-\frac{6135114}{8103636319}$    & $-\frac{1708530}{2442318371}$   & 38 & $-\frac{142643}{621173500099}$ & $-\frac{9833}{43040520729}$       \\
18 & $-\frac{1185727}{2724946876}$    & $-\frac{494221}{1232914518}$    & 39 & $-\frac{68896}{360364723503}$  & $-\frac{47271}{248344211236}$     \\
19 & $-\frac{615665}{3140688998}$     & $-\frac{12658987}{72084729713}$ & 40 & $-\frac{7439}{83635247148}$    & $-\frac{35159}{396928578929}$     \\
20 & $-\frac{303045}{7068751433}$     & $-\frac{157896}{5140050545}$ &
                                                    & & \\
\hline
\end{tabular}
\end{center}

\newpage

\begin{center}
\begin{tabular}{|ccc|ccc|}
   \hline
  $n$ & $c_n(Q_-^\sharp)$ & $c_n(Q_+^\sharp)$ & $n$ & $c_n(Q_-^\sharp)$ & $c_n(Q_+^\sharp)$ \\
  \hline
0 & $\frac{15024113}{419334545}$ &      $-\frac{84863998}{2670638095}$ &    21 & $\frac{644854}{5545449541}$ &      $\frac{834944}{7476783719}$    \\
1 & $\frac{26338241}{486449897}$ & 	$-\frac{9625016}{212402275}$ & 	    22 & $\frac{577559}{3414678615}$ &      $\frac{1878683}{11290298852}$ \\
2 & $\frac{27404960}{734705327}$ & 	$-\frac{33166439}{1145664248}$ &    23 & $\frac{764515}{4581725254}$ &      $\frac{1108028}{6705865483}$ \\
3 & $\frac{49269755}{1955111421}$ & 	$-\frac{3615029}{210454920}$ & 	    24 & $\frac{1152742}{8299295951}$ &     $\frac{829552}{6014353219}$ \\
4 & $\frac{14736779}{829173919}$ & 	$-\frac{20461453}{2336769147}$ &    25 & $\frac{692198}{6670546807}$ &      $\frac{2114434}{20489753259}$ \\
5 & $\frac{9864830}{799321543}$ & 	$-\frac{7751669}{1914492761}$ &     26 & $\frac{1104826}{15494381429}$ &    $\frac{308637}{4349217023}$ \\
6 & $\frac{5807541}{688797160}$ & 	$-\frac{3239765}{2016502842}$ &     27 & $\frac{768301}{16894467943}$ &     $\frac{778702}{17199766551}$ \\
7 & $\frac{2581561}{397061491}$ & 	$\frac{3193777}{8145377075}$ & 	    28 & $\frac{632143}{23507582867}$ &     $\frac{376901}{14078789498}$ \\
8 & $\frac{3477497}{601816550}$ & 	$\frac{20730289}{9981418642}$ &     29 & $\frac{387908}{26655806505}$ &     $\frac{177583}{12263047605}$ \\
9 & $\frac{2453911}{503078748}$ & 	$\frac{2919390}{1104484453}$ & 	    30 & $\frac{193625}{27855685003}$ &     $\frac{694582}{100537691067}$ \\
10 & $\frac{1828028}{582905913}$ & 	 $\frac{2764644}{1543787783}$ &     31 & $\frac{102053}{38897983149}$ &     $\frac{104861}{40354712764}$ \\
11 & $\frac{1210513}{1390895317}$ & 	 $\frac{3351387}{53798325847}$ &    32 & $\frac{54610}{139362283877}$ &     $\frac{34666}{91965684667}$ \\
12 & $-\frac{2763189}{2312900014}$ & 	 $-\frac{3152020}{1877195641}$ &    33 & $-\frac{50339}{84179707475}$ &     $-\frac{66308}{109266512019}$ \\
13 & $-\frac{3105950}{1233038341}$ & 	 $-\frac{2501676}{890621917}$ &     34 & $-\frac{24265}{26638113301}$ &     $-\frac{85546}{93373352285}$ \\
14 & $-\frac{6430687}{2165888529}$ & 	 $-\frac{7269183}{2313247177}$ &    35 & $-\frac{48293}{53828091949}$ &     $-\frac{60407}{67096252029}$ \\
15 & $-\frac{3392927}{1242562631}$ & 	 $-\frac{2439474}{860760853}$ &     36 & $-\frac{141873}{187707853304}$ &   $-\frac{78067}{103031870398}$ \\
16 & $-\frac{2483649}{1174604665}$ & 	 $-\frac{3250017}{1493439299}$ &    37 & $-\frac{52688}{88842876251}$ &     $-\frac{119432}{200997801417}$ \\
17 & $-\frac{423589}{301989567}$ & 	 $-\frac{5370137}{3730662129}$ &    38 & $-\frac{48285}{105324245683}$ &    $-\frac{131581}{286559012168}$ \\
18 & $-\frac{2486208}{3197249045}$ & 	 $-\frac{9075478}{11351091637}$ &   39 & $-\frac{25525}{68383259559}$ &     $-\frac{33818}{90473887987}$ \\
19 & $-\frac{8556765}{26851400932}$ & 	 $-\frac{2007859}{6052955872}$ &    40 & $-\frac{39253}{228035775642}$ &    $-\frac{28319}{164297023388}$ \\
20 & $-\frac{694132}{21959417587}$ & 	 $-\frac{976123}{24793963305}$ & & & \\
  \hline
\end{tabular}
\end{center}

\newpage

\subsection{Coefficients for functional $\psi_F$ on the right-hand
  side at low frequency}
\label{apx:flats}
These are the coefficients used in Definition \ref{def:flats}
(files {\tt Pflm.dat}, {\tt Pflp.dat}, {\tt Qflm.dat}, {\tt Qflp.dat}).

\begin{center}
\begin{tabular}{|ccc|ccc|}
   \hline
  $n$ & $c_n(P_-^\flat)$ & $c_n(P_+^\flat)$ & $n$ & $c_n(P_-^\flat)$ & $c_n(P_+^\flat)$ \\
  \hline
0 & $\frac{6446693}{169820556}$ &       $-\frac{14137311}{328631966}$ &       21 & $\frac{2166057}{2090984221}$ &      $-\frac{8193927}{5581197824}$ \\
1 & $-\frac{12327282}{264684349}$ & 	$\frac{11693343}{210302048}$ & 	      22 & $\frac{22085969}{2757489833}$ &     $\frac{8113899}{1279324939}$ \\
2 & $-\frac{23803678}{393848217}$ & 	$\frac{19064463}{376977364}$ & 	      23 & $\frac{11145118}{2388280363}$ &     $\frac{6807137}{1806168845}$ \\
3 & $\frac{21048531}{269406664}$ & 	$-\frac{14724959}{224838528}$ &       24 & $-\frac{3638881}{2329474793}$ &     $-\frac{1964186}{1016010475}$ \\
4 & $\frac{28821313}{954788917}$ & 	$-\frac{19344345}{453315719}$ &       25 & $-\frac{5023983}{1478063533}$ &     $-\frac{4227356}{1214592129}$ \\
5 & $-\frac{20768287}{646063235}$ & 	$\frac{41780932}{940844507}$ & 	      26 & $-\frac{4084347}{2103870572}$ &     $-\frac{1131210}{595573079}$ \\
6 & $-\frac{8620650}{158968559}$ & 	$\frac{10086808}{280153789}$ & 	      27 & $-\frac{3892379}{13647365033}$ &    $-\frac{3456692}{16247544233}$ \\
7 & $\frac{8822238}{1388095199}$ & 	$-\frac{2236863}{4220113688}$ &       28 & $\frac{939208}{2090279385}$ &       $\frac{1064265}{2079124508}$ \\
8 & $\frac{14533007}{284261596}$ & 	$-\frac{43264411}{685770018}$ &       29 & $\frac{6288371}{12709007896}$ &     $\frac{1558501}{2903960931}$ \\
9 & $\frac{26458656}{338746937}$ & 	$-\frac{13703011}{172299852}$ &       30 & $\frac{1734665}{5973228508}$ &      $\frac{1896802}{6042571493}$ \\
10 & $\frac{22678807}{320218732}$ & 	 $-\frac{16346877}{279671722}$ &      31 & $\frac{536630}{5651461237}$ &       $\frac{1307750}{12339426353}$ \\
11 & $\frac{9797368}{378006561}$ & 	 $-\frac{14168711}{360974638}$ &      32 & $-\frac{382358}{33393758311}$ &     $-\frac{474408}{63937042567}$ \\
12 & $\frac{4288859}{436809097}$ & 	 $\frac{9449297}{1907064162}$ &       33 & $-\frac{590177}{13884609616}$ &     $-\frac{411365}{9852163702}$ \\
13 & $-\frac{70935700}{3118935447}$ & 	 $\frac{7001722}{864204645}$ & 	      34 & $-\frac{213203}{5926095679}$ &      $-\frac{563584}{15465731401}$ \\
14 & $-\frac{22961994}{762738821}$ & 	 $\frac{10084709}{946783989}$ &       35 & $-\frac{9792948}{475362427189}$ &   $-\frac{231121}{10909405681}$ \\
15 & $-\frac{4520657}{1003124969}$ & 	 $\frac{8660835}{292877546}$ & 	      36 & $-\frac{87697}{10557331092}$ &      $-\frac{196004}{22666823599}$ \\
16 & $-\frac{6645115}{562404224}$ & 	 $\frac{19207703}{2016347648}$ &      37 & $-\frac{44622}{28797895063}$ &      $-\frac{106183}{64107389914}$ \\
17 & $-\frac{15404221}{1036278934}$ & 	 $-\frac{5292167}{1012788664}$ &      38 & $\frac{56003}{96311651760}$ &       $\frac{49371}{63716889130}$ \\
18 & $\frac{31509477}{7369719805}$ & 	 $\frac{2686367}{435921537}$ & 	      39 & $\frac{439154}{367625188367}$ &     $\frac{97211}{63656058176}$ \\
19 & $\frac{366127}{78792700}$ & 	 $\frac{3349107}{1212439678}$ &       40 & $\frac{337601}{508207249578}$ &     $\frac{45469}{53232547777}$ \\
20 & $-\frac{4449739}{1029484476}$ & 	 $-\frac{5559403}{769910962}$ & & & \\
  \hline
\end{tabular}
  \end{center}

\newpage

\begin{center}
\begin{tabular}{|ccc|ccc|}
   \hline
  $n$ & $c_n(Q_-^\flat)$ & $c_n(Q_+^\flat)$ & $n$ & $c_n(Q_-^\flat)$ & $c_n(Q_+^\flat)$ \\
  \hline
0 & $\frac{22292801}{3541672815}$ &       $-\frac{7320085}{422726802}$ &      21 & $\frac{1205233}{136249056}$ &      $\frac{3269741}{340039997}$ \\
1 & $-\frac{17183123}{628486997}$ & 	  $\frac{5370257}{763909565}$ &       22 & $\frac{7689066}{3782048833}$ &     $\frac{1714269}{851396435}$ \\
2 & $-\frac{2601881}{36861921}$ & 	  $\frac{6497353}{117477620}$ &       23 & $-\frac{13382928}{1921692661}$ &   $-\frac{6544562}{899773689}$ \\
3 & $-\frac{12788236}{351943543}$ & 	  $\frac{14885347}{530360674}$ &      24 & $-\frac{29677129}{3876084991}$ &   $-\frac{6647810}{832327799}$ \\
4 & $\frac{28269278}{1953066793}$ & 	  $-\frac{16788921}{1059683705}$ &    25 & $-\frac{19841725}{8067194896}$ &   $-\frac{3631559}{1341549134}$ \\
5 & $\frac{20469314}{2291987951}$ & 	  $-\frac{4389249}{756215519}$ &      26 & $\frac{1962597}{1098180905}$ &     $\frac{14462219}{8835686444}$ \\
6 & $-\frac{4604735}{149513089}$ & 	  $\frac{38096807}{1136416257}$ &     27 & $\frac{1511361}{569900654}$ &      $\frac{6573893}{2551571809}$ \\
7 & $-\frac{21852472}{393873431}$ & 	  $\frac{10456207}{194323924}$ &      28 & $\frac{1703825}{1059871716}$ &     $\frac{1508811}{956103488}$ \\
8 & $-\frac{12134291}{257196164}$ & 	  $\frac{12813735}{325285216}$ &      29 & $\frac{760219}{1708397708}$ &      $\frac{1342975}{3059758387}$ \\
9 & $-\frac{16365871}{971101671}$ & 	  $\frac{3830009}{573237032}$ &       30 & $-\frac{399962}{2441922197}$ &     $-\frac{706790}{4402882723}$ \\
10 & $\frac{6889747}{460862719}$ & 	   $-\frac{8628562}{427825761}$ &     31 & $-\frac{1322611}{4403040654}$ &    $-\frac{822089}{2786494425}$ \\
11 & $\frac{26137905}{785469161}$ & 	   $-\frac{18694903}{622076748}$ &    32 & $-\frac{1965647}{8887295246}$ &    $-\frac{914359}{4220758055}$ \\
12 & $\frac{202329748}{5789596889}$ & 	   $-\frac{9472768}{366945325}$ &     33 & $-\frac{848635}{8065438807}$ &     $-\frac{1109517}{10858798918}$ \\
13 & $\frac{32134415}{1418682452}$ & 	   $-\frac{47560750}{2582279807}$ &   34 & $-\frac{711346}{28065675695}$ &    $-\frac{244921}{10370454135}$ \\
14 & $\frac{9236098}{1889427287}$ & 	   $-\frac{4228843}{306292600}$ &     35 & $\frac{164251}{15310750346}$ &     $\frac{751693}{65123424948}$ \\
15 & $-\frac{12439359}{3554243900}$ & 	   $-\frac{4001267}{703804442}$ &     36 & $\frac{347400}{18439634939}$ &     $\frac{1362251}{71079921794}$ \\
16 & $\frac{9589627}{5230046955}$ & 	   $\frac{7262256}{894504715}$ &      37 & $\frac{211901}{13949644989}$ &     $\frac{745769}{48570592853}$ \\
17 & $\frac{1634541}{355455424}$ & 	   $\frac{12443095}{956856786}$ &     38 & $\frac{176189}{18981293769}$ &     $\frac{94148}{10017012737}$ \\
18 & $-\frac{412771}{109761066}$ & 	   $\frac{3275390}{1002114789}$ &     39 & $\frac{509239}{99105942367}$ &     $\frac{258699}{49907652418}$ \\
19 & $-\frac{24644619}{3526953038}$ & 	   $-\frac{3518545}{1432492754}$ &    40 & $\frac{282209}{151571651735}$ &    $\frac{260032}{139671794219}$ \\
20 & $\frac{1621784}{612479555}$ & 	   $\frac{5987552}{1215060451}$ &     & & \\
  \hline
\end{tabular}
\end{center}

\newpage

\subsection{Coefficients for functional $\psi_F$ on the left-hand
  side}
\label{apx:PLpm}
These are the coefficients used in Definition \ref{def:PLpm}
(files {\tt PLm.dat}, {\tt PLp.dat}).
\begin{center}
  \begin{tabular}{|ccc|ccc|}
    \hline
    $n$ & $c_n(P_-^L)$ & $c_n(P_+^L)$ & $n$ & $c_n(P_-^L)$ & $c_n(P_+^L)$ \\
    \hline
0 & $-\frac{171003977}{103918830}$ &      $\frac{300931817}{241723206}$ &      26 & $\frac{96976718}{349936601}$ &     $\frac{92590229}{334108131}$ \\
1 & $-\frac{166671913}{118566186}$ & 	  $\frac{31563190}{53581289}$ &        27 & $\frac{151400197}{305788942}$ &    $\frac{22436507}{45315897}$    \\   
2 & $\frac{45964957}{58483290}$ & 	  $-\frac{99394800}{60361067}$ &       28 & $-\frac{17543371}{94545237}$ &     $-\frac{46126776}{248587741}$ \\
3 & $\frac{48279065}{56023151}$ & 	  $-\frac{140995963}{78831741}$ &      29 & $-\frac{82657111}{170462639}$ &    $-\frac{93590633}{193010693}$ \\
4 & $\frac{8972827}{48587872}$ & 	  $-\frac{115016911}{96883787}$ &      30 & $\frac{49807827}{180113107}$ &     $\frac{9174673}{33177092}$      \\
5 & $-\frac{41493513}{97613005}$ & 	  $-\frac{62369530}{96861583}$ &       31 & $\frac{24408481}{84951138}$ &      $\frac{28186382}{98099723}$ \\
6 & $-\frac{124868039}{211620567}$ & 	  $-\frac{85729517}{168054687}$ &      32 & $-\frac{58360311}{168209960}$ &    $-\frac{36952057}{106505670}$ \\
7 & $-\frac{82444695}{147905063}$ & 	  $-\frac{70884043}{139773967}$ &      33 & $\frac{14741103}{458100637}$ &     $\frac{23630575}{734353558}$ \\
8 & $-\frac{237703833}{506032154}$ & 	  $-\frac{65403365}{142220181}$ &      34 & $\frac{23982523}{146077235}$ &     $\frac{35055715}{213523903}$ \\
9 & $-\frac{452690956}{1349473269}$ & 	  $-\frac{37359657}{111070835}$ &      35 & $-\frac{23716921}{182213911}$ &    $-\frac{120267449}{923998618}$ \\
10 & $-\frac{18816773}{130013559}$ & 	   $-\frac{60126713}{412382715}$ &     36 & $\frac{16609307}{561307166}$ &     $\frac{4069334}{137522073}$ \\
11 & $\frac{40238064}{457449389}$ & 	   $\frac{25758991}{293806979}$ &      37 & $\frac{16141679}{720463986}$ &     $\frac{10456315}{466704758}$ \\
12 & $\frac{297981729}{933827237}$ & 	   $\frac{169452468}{531069001}$ &     38 & $-\frac{6565187}{256652917}$ &     $-\frac{19154980}{748825813}$ \\
13 & $\frac{53894425}{111747446}$ & 	   $\frac{50340647}{104375994}$ &      39 & $\frac{6445520}{465342699}$ &      $\frac{7490204}{540765019}$ \\
14 & $\frac{39121627}{77223319}$ & 	   $\frac{56404525}{111337301}$ &      40 & $-\frac{4486162}{876691085}$ &     $-\frac{8167318}{1596066945}$ \\
15 & $\frac{25655501}{73282493}$ & 	   $\frac{42504024}{121408339}$ &      41 & $\frac{170205}{120603977}$ &       $\frac{1276468}{904480581}$ \\
16 & $\frac{24126233}{599965344}$ & 	   $\frac{21968416}{546305221}$ &      42 & $-\frac{3620268}{12159078431}$ &   $-\frac{1882565}{6322806898}$ \\
17 & $-\frac{83229183}{274756991}$ & 	   $-\frac{62067841}{204898926}$ &     43 & $\frac{663057}{12981857515}$ &     $\frac{602699}{11800120564}$ \\
18 & $-\frac{19346066}{39135121}$ & 	   $-\frac{47572821}{96234968}$ &      44 & $-\frac{318302}{64657214463}$ &    $-\frac{21113}{4288718797}$ \\
19 & $-\frac{30112253}{78550984}$ & 	   $-\frac{24747641}{64556828}$ &      45 & $\frac{157007}{87593065450}$ &     $\frac{97065}{54151858808}$ \\
20 & $\frac{8376683}{1097584769}$ & 	   $\frac{3237094}{424151787}$ &       46 & $\frac{45784}{39375676069}$ &      $\frac{144114}{123942559835}$ \\
21 & $\frac{62855398}{152433253}$ & 	   $\frac{31137242}{75512227}$ &       47 & $\frac{11965}{21891917444}$ &      $\frac{92393}{169048051048}$ \\
22 & $\frac{206082020}{449291361}$ & 	   $\frac{174951822}{381422611}$ &     48 & $\frac{37259}{62899729312}$ &      $\frac{1671737}{2822185366593}$\\ 
23 & $\frac{10197255}{358894151}$ & 	   $\frac{26642213}{937677287}$ &      49 & $\frac{119486}{790116130157}$ &    $\frac{47253}{312466375855}$ \\
24 & $-\frac{36390813}{78742075}$ & 	   $-\frac{193849193}{419448933}$ &    50 & $\frac{60673}{454939111170}$ &     $\frac{10847}{81333122376}$ \\
25 & $-\frac{40487945}{109429569}$ & 	   $-\frac{32480152}{87786353}$ &     & & \\
   \hline
  \end{tabular}
\end{center}

\newpage

\section{Numerical approximation}
\label{apx:numerical}
\noindent The numerical approximation of the self-similar solution that leads to the definition of the
polynomial $P_*$ (see Definition \ref{def:gstar}) is obtained by a
standard Chebyshev pseudo-spectral method, see e.g.~\cite{Boy01}. We start
from the radial profile equation Eq.~\eqref{eq:profilerad},
\[ q''(r)+\left (\frac{2}{r}+i\alpha r\right )q'(r)+(i\alpha-1)q(r)
  +q(r)|q(r)|^2=0,\qquad r>0, \]
and compactify the problem by writing
\[ q(r)=(1+r)^{-1-\frac{i}{\alpha}}f\left (\frac{r-1}{r+1}\right). \]
This leads to the equation
\[
  \underbrace{f''(y)+p_0(y,\alpha)f'(y)+q_0(y,\alpha)f(y)+\frac{f(y)|f(y)|^2}{(1-y)^2}}_{=:\mc R(\alpha,f)(y)}=0,\qquad
  y\in (-1,1), \]
with the coefficients
\begin{align*}
  p_0(y,\alpha)&:=\frac{4i\alpha}{(1-y)^3}-\frac{2i\alpha}{(1-y)^2}-\frac{2+\frac{2i}{\alpha}}{1-y}+\frac{2}{1+y} \\
  q_0(y,\alpha)&:=-\frac{2-2i\alpha}{(1-y)^3}-\frac{\frac{1}{\alpha^2}-\frac{i}{\alpha}}{(1-y)^2}-\frac{1+\frac{i}{\alpha}}{1-y}-\frac{1+\frac{i}{\alpha}}{1+y},
\end{align*}
cf.~Definition \ref{def:R}. Now we insert the ansatz
\[ f(y)=\sum_{n=0}^{50}c_nT_n(y) \]
into the equation, i.e., we consider
\[ \mc R\left (\alpha,\sum_{n=0}^{50}c_nT_n\right)(y)=0. \]
We evaluate this equation at  $y=y_k$ for $k\in \{0,1,\dots,51\}$,
where $y_k$ are the
standard Gau\ss{}-Lobatto collocation points \cite{Boy01}. This yields
the coupled system
\[ \mc R\left (\alpha,\sum_{n=0}^{50}c_nT_n\right)(y_k)=0,\qquad k\in
  \{0,1,\dots,51\}, \]
of $52$ algebraic equations for the $52$ unknowns $c_n$ and
$\alpha$. This system is now fed into a black-box numerical root finder
with very simple starting values like $\alpha=c_0=1$ and
$c_n=0$ for $n\in \{1,2,\dots,51\}$. Concretely, we use the function {\tt FindRoot} provided by \emph{Wolfram
Mathematica} with 100 digits of precision. On standard
desktop hardware of 2024, the numerical root finder converges after a
few minutes.
The fact that it is very easy to find the self-similar solution
numerically shows that it is a very strong attractor for the root
finding problem. The numerical values for $c_n$
and $\alpha$ are then
rounded to give the rational numbers listed in the table in Appendix
\ref{apx:Pstar}.
The other tables in Appendix \ref{apx:tables} are generated analogously.

\bibliography{nls3ex}
\bibliographystyle{plain}

\end{document}